\documentclass[fleqn]{cas-sc}

\usepackage{hyperref}
\hypersetup{
    colorlinks = true,
    linkcolor = blue
    }
\usepackage{amsmath,amsthm,amssymb}
\usepackage{bm}
\usepackage{graphicx}
\usepackage[noabbrev,capitalise,nameinlink]{cleveref}
\usepackage{caption}
\usepackage{mathtools}
\usepackage{subcaption}
\usepackage{siunitx}
\usepackage{fonts}
\usepackage{enumitem}
\usepackage{multirow}
\usepackage{placeins}
\usepackage{xcolor}

\usepackage{algorithm}
\usepackage{algpseudocode}

\algrenewcommand\algorithmiccomment[1]{\hfill\(\triangleright\) #1}

\crefalias{step}{enumi}
\crefname{step}{step}{steps}
\Crefname{step}{Step}{Steps}

\crefalias{property}{enumi}
\crefname{property}{property}{properties}
\Crefname{property}{Property}{Properties}

\newcommand{\bmc}{\bm{c}}
\newcommand{\bmi}{\bm{i}}
\newcommand{\bmx}{\bm{x}}
\newcommand{\bmv}{\bm{v}}
\newcommand{\bmu}{\bm{u}}

\newcommand{\bmz}{\bm{z}}
\newcommand{\bfe}{\bf{e}}
\newcommand{\bfn}{{\bf{n}}}

\newcommand{\bmrho}{\bm{\rho}}
\newcommand{\dx}{{\rm\,d}}

\newcommand{\bmj}{\bm{j}}
\newcommand{\bmell}{\bm{\ell}}

\newcommand{\bmN}{\bm{N}}
\newcommand{\bmW}{\bm{W}}
\newcommand{\bmV}{\bm{V}}
\newcommand{\bmX}{\bm{X}}
\newcommand{\bmZ}{\bm{Z}}

\newcommand{\Kn}{{\mathrm{Kn}}}
\newcommand{\dofvavg}{\mathrm{DoF}_{\bmv}^{\mathrm{avg}}}

\newcommand{\grad}{\nabla}
\newcommand{\dt}{\Delta t}

\newcommand{\asgard}{{\rm ASGarD}}

\newcommand{\lavg}{\{\!\!\{}
\newcommand{\ravg}{\}\!\!\}}
\newcommand{\ljmp}{[\![}
\newcommand{\rjmp}{]\!]}
\newcommand{\ledge}{\big<\!\!\big<}
\newcommand{\redge}{\big>\!\!\big>}
\newcommand{\mT}{\mathcal{T}}
\newcommand{\jmp}[1]{\ljmp #1 \rjmp}
\newcommand{\avg}[1]{\lavg{#1}\ravg} 
\newcommand{\AP}{\cA_{\text{V}}}

\newcommand{\W}{\Omega}

\DeclareMathOperator{\spann}{span}

\newtheorem{prop}{Proposition}
\newtheorem{defn}{Definition}
\newtheorem{remark}{Remark}

\numberwithin{equation}{section}
\numberwithin{figure}{subsection}
\numberwithin{table}{subsection}
\numberwithin{prop}{subsection}
\numberwithin{defn}{subsection}
\numberwithin{remark}{subsection}
\numberwithin{algorithm}{subsection}

\begin{document}
\let\WriteBookmarks\relax
\def\floatpagepagefraction{1}
\def\textpagefraction{.001}

\shorttitle{Adaptive Sparse-grid BGK}

\shortauthors{S.~Schnake et al.}

\title[mode=title]{Adaptive Sparse-Grid Discontinuous Galerkin Approximations to the Bhatnagar--Gross--Krook Model}

\tnotemark[1]

\tnotetext[1]{
This manuscript has been authored by UT-Battelle, LLC under Contract No.~DE-AC05-00OR22725 with the U.S.~Department of Energy. The United States Government retains and the publisher, by accepting the article for publication, acknowledges that the United States Government retains a non-exclusive, paid-up, irrevocable, world-wide license to publish or reproduce the published form of this manuscript, or allow others to do so, for United States Government purposes. The Department of Energy will provide public access to these results of federally sponsored research in accordance with the DOE Public Access Plan(\url{http://energy.gov/downloads/doe-public-access-plan}).\\
S.~Schnake and C.D.~Hauck were supported by the Office of Science, Advanced Scientific Computing Research program of the U.S. Department of Energy.
M.~Stoyanov and E.~Endeve were supported by the Office of Science, Scientific Discovery through Advanced Computing program of the U.S.~Department of Energy.
M.~Stoyanov was supported through the Laboratory Directed Research and Development Program of Oak Ridge National Laboratory, managed by UT-Battelle, LLC, for the U.S.~Department of Energy.
}

\author[1]{Stefan Schnake}[orcid=0000-0002-1518-3538]
\cormark[1]
\ead{schnakesr@ornl.gov}
\credit{Conceptualization, Methodology, Software, Formal Analysis, Investigation, Writing - Original Draft, Visualization}

\affiliation[1]{organization={Computer Science and Mathematics Division, Oak Ridge National Laboratory},city={Oak Ridge},state={TN},postcode={37830},country={USA}}

\author[1]{Miroslav Stoyanov}[orcid=0000-0002-8199-5577]
\ead{stoyanovmk@ornl.gov}
\credit{Methodology, Software, Writing - Review \& Editing}

\author[1,2]{Eirik Endeve}[orcid=0000-0003-1251-9507]
\ead{endevee@ornl.gov}
\credit{Validation, Investigation, Writing - Review \& Editing}

\affiliation[2]{organization={Department of Physics and Astronomy, The University of Tennessee},city={Knoxville},state={TN},postcode={37996},country={USA}}

\author[1,3]{Cory D. Hauck}[orcid=0000-0001-5559-502X]
\ead{hauckc@ornl.gov}
\credit{Writing - Review \& Editing, Supervision}

\affiliation[3]{organization={Department of Mathematics, The University of Tennessee},city={Knoxville},state={TN},postcode={37996},country={USA}}

\cortext[1]{Corresponding author}

\begin{abstract}
    This work studies adaptive sparse-grid discontinuous Galerkin (DG) discretizations for the Bhatnagar--Gross--Krook (BGK) model, a kinetic equation posed in four- and six-dimensional phase-space. Standard DG methods are rendered impractical for the BGK model by the curse of dimensionality, motivating compressed representations that adapt to the solution in time. Using the adaptive sparse-grid DG method, we quantify accuracy and compression by comparing the adaptive degrees of freedom to full-grid DG methods and by assessing the resulting kinetic and fluid quantities in both fluid and rarefied regimes. Test cases include a relaxation problem, a multidimensional Sod shock tube, and shear/expansion flows used in prior low-rank BGK studies. To build an efficient Maxwellian evaluation without violating conservation, a central obstacle for structure-perserving BGK simulations, we introduce a hybrid interpolation strategy that exploits velocity separability to recover the correct discrete collision invariants and prove conservation of the resulting discrete collision operator on adaptive sparse grids. Our results show that the adaptive sparse-grid strategy can recover accurate and physically relevant solutions with sharp gradients, and the method reduces the active degrees of freedom by factors ranging from several-fold to several orders of magnitude, with the largest reductions occurring in the six-dimensional examples. All computations are performed with the open-source ASGarD adaptive sparse-grid DG library.
\end{abstract}



\begin{NoHyper}
\maketitle
\end{NoHyper}

\section{Introduction}

In this paper, we investigate adaptive sparse-grid solutions to kinetic equations using the Bhatnagar--Gross--Krook (BGK) model.  In a general setting, equations of this type are defined in terms of a kinetic distribution $f$ that evolves in time over a six-dimensional phase space (three position and three velocity variables).

The BGK model features the BGK collision operator which is a  relaxation of the distribution to a local Maxwellian equilibrium \cite{bhatnagar1954ModelCollision}. 
The BGK collision operator is a simplification of the Boltzmann collision operator, which requires evaluation of a five-dimensional integral at every point in phase space.  In spite of the simplification, the BGK operator inherits several fundamental properties of the Boltzmann operator, including the preservation of collision invariants which induce local conservations laws, an entropy dissipation law, and the unique characterization of local thermal equilibrium (the so-called H-theorem).
Additionally, the BGK model features multi-scale behavior that arises due to the balance between phase-space advection and collisional dynamics.
Due to the above mentioned properties, the BGK well aligns with the compressible Euler and Navier--Stokes equations in the \textit{fluid regime} where collisions dominate.
The main drawback of the BGK collision operator is the inability to reproduce the correct Prandtl number of the Boltzmann equation, motivating extensions such as the ES-BGK \cite{holwayjr1965KineticTheory} and Shakhov \cite{shakhov1968GeneralizationKrook} collision operators.

The approximation for the phase-space distribution $f$ is built with the discontinuous Galerkin (DG) method, which was introduced for kinetic models in radiation transport \cite{reedHill_1973} and remains popular for kinetic simulations \cite{hakim_etal_2020,endeveHauck_2022,alekseenko2012application,cheng2014energy,cheng2011brief,endeve2015bound,garrett2018fast,hong2022generalized,qiu2011positivity,abdelmalik2016entropy}.
The DG method is high-order in accuracy and can be modified to preserve many physical constraints (e.g., local conservation, positivity, entropy stability, etc.) \cite{zhangShu_2011,zhang2022energy,guermond2010asymptotic,yan2023entropy}.
However, when applied to high-dimensional PDEs, Eulerian grid-based methods, including DG, suffer from the \textit{curse of dimensionality} \cite{bellman1959adaptive}, where the cost to approximate a general measurable function scales like $\mathcal{O}(N^d)$, with $d$ the dimension of the domain and $N$ the degrees of freedom in a single dimension.
Such a scaling in six dimensions makes the standard DG method intractable for approximating general kinetic equations, even on leadership class computing facilities \cite{evans2026SweepbasedImplicit,dimarco2018EfficientNumerical}, and until recently limited simulations of the BGK model to particle based methods \cite{macrossan2001ParticleSimulation,pfeiffer2019ExtensionParticlebased,carrillo2020particle}.  

The curse of dimensionality motivates the use of various methods to compress grid-based solutions of the BGK model during simulations to mitigate the large memory and compute footprint.
Most popular among these methods are low-rank methods, where a low-rank decomposition of the discrete phase-space distribution, formed as a matrix or tensor, is advanced in time \cite{einkemmer2021EfficientDynamical,dektor2026InterpolatoryDynamical,sands2026AdaptiverankApproach,galindo-olarte2025NodalDiscontinuous,baumann2026StableMultiplicative}.
These methods show impressive compression rates from the BGK model in the fluid regime, but low-rank simulations of the non-linear BGK model in the \textit{rarefied} regime, where phase-space advection matches or dominates collisions, often require significantly larger ranks to adequately capture the non-equilibrium behavior of the kinetic distribution \cite{dektor2026InterpolatoryDynamical,guo2026HighlyEfficient,sands2026AdaptiverankApproach}.
However, the rarefied regime is precisely where continuum models break down and where efficient compression strategies are needed \cite{struchtrup2005MacroscopicTransport,sharipov2016RarefiedGas}.

Along with low-rank methods, sparse-grids methods have been developed for grid-based discretizations of high-dimensional partial differential equations (PDEs) \cite{leentvaar2006pricing,hemker1995sparse,mishra2012sparse,schwab2008sparse, shen2010sparse, shen2010efficient, balder1996solution}, including the DG method \cite{schnake2024SparsegridDiscontinuous,wang2016sparse,guo2016sparse}. 
Instead of a traditional local modal or nodal basis, the sparse-grid DG method utilizes a multiwavelet basis \cite{alpert1993class} expansion that induces a decay in the coefficient size over finer levels based on the solution regularity.
The decay allows model-independent criteria for local coarsening or refinement in time, and this procedure, called the adaptive sparse-grid DG method, has been shown to significantly reduce the degrees of freedom (DoFs) required to store the kinetic distribution, even when solutions exhibit discontinuous behavior \cite{schnake2024SparsegridDiscontinuous,kormann2016sparse,tao2019sparse}.

The goal of this work is to construct efficient and accurate simulations for the BGK model utilizing adaptive sparse grids.
DG sparse-grid simulations of the BGK model have been traditionally hampered by the lack of separability of the Maxwellian source in the BGK collision operator, but recent work in sparse-grid interpolation methods \cite{huang2020adaptive,tao2021AdaptiveHighorder} have enabled an efficient evaluation of the Maxwellian.
However, direct application of these interpolatory procedures results in a discrete collision operator that does not capture the correct collision invariants and leads to non-physical oscillations during simulations.
This behavior is even more pronounced on adaptive sparse grids, where velocity resolution can be coarse in certain regions.
In this work, we detail a hybrid interpolation procedure that takes advantage of the Maxwellian being separable in velocity to recover the correct collision invariants.
We prove that the hybrid interpolation is conservative on adaptive sparse grids and show in several numerical experiments why such an approach is preferable.

We assess the quality and computational savings of the adaptive sparse-grid DG method applied to the BGK model in four and six dimensions.
The main metric of compression is the ratio of DoFs of the adaptive sparse-grid DG method to the standard ``full-grid'' DG method, and we additionally compare to solutions derived from other fluid or low-rank kinetic models for quality.
The test problems are: a spatially varying initial condition relaxing toward equilibrium, a multi-dimensional version of the Sod Shock tube problem \cite{toro2009RiemannSolvers} featuring sharp gradients, and a shear flow and expansion problem referenced in the low-rank BGK literature \cite{einkemmer2021EfficientDynamical,dektor2026InterpolatoryDynamical}.
In general, we find that the adaptive sparse-grid method well approximates the kinetic distribution and its associated velocity moments with a fraction of the DoFs in both the fluid and rarefied regimes.
For the Sod problem in particular, where an initial contact discontinuity quickly produces non-equilibrium behavior, we compare the sparse-grid kinetic velocity distribution against the first-order Chapman--Enskog expansion used to derive the Navier--Stokes equations.
We find that the velocity profiles are both qualitatively (e.g., smoothness) and quantitatively (e.g., perturbation size) different away from the fluid regime.

Complementing this work is the development of the adaptive sparse-grid DG library \asgard~(Adaptive Sparse-Grid Discretization) \cite{hahn2024ASGarDAdaptive}.  
The goal of this open-source project is to facilitate and promote the use of adaptive sparse-grid methods for the approximation of kinetic models by providing a robust yet flexible adaptive sparse-grid library that is deployable on CPU and GPU architectures.
All sparse-grid results in this work were computed using \asgard.

The rest of the paper is organized as follows.  In \Cref{sec:prelim} we list the BGK model, the full-grid DG method, and the timestepping method used to advance the discretized equations.
\Cref{sec:sparse-grid} provides an overview of the adaptive sparse-grid DG method, while \Cref{sec:interp} introduces the adaptive DG interpolation and the new hybrid interpolation procedure.
In \Cref{sec:numerical_experiments} we analyze the results of the adaptive sparse-grid DG method on the suite of chosen test problems.
\Cref{sec:conclusion} gives final conclusions and future work.

\section{Preliminaries, notation, and the BGK model}
\label{sec:prelim}

\subsection{The BGK model}

Let $d\in\mathbb{N}$.  Given $\bmx\in\Omega_{x}\subset\mathbb{R}^d$ and $\bmv\in\mathbb{R}^d$, the BGK model for a distribution  $f=f(\bmx,\bmv,t)$ in a $dxdv$ phase-space dimensions reads
\begin{equation}\label{eqn:bgk}
	\partial_t f + \bmv\cdot\grad_{\bmx} f = \nu(M[\bmrho_f]-f),\qquad (\bmx,\bmv,t)\in\Omega_x\times\bbR^d\times(0,\infty)
\end{equation} 
where we assume for simplicity that $\nu\geq 0$ is a constant collision frequency.
In \eqref{eqn:bgk}, the $d+2$ moments $\bmrho_f=\bmrho_f(\bmx,t)$ are the density, momentum, and energy of the distribution $f$ and are defined by 
\begin{equation}
	\bmrho_f(\bmx,t) = \int_{\mathbb{R}^d} {\bfe}(\bmv) f(\bmx,\bmv,t)~{\rm d}\bmv~~\text{where}~~{\bfe}(\bmv)=\begin{pmatrix} 1 \\ \bmv \\ \frac{1}{2}|\bmv|^2 \end{pmatrix}.
\end{equation}
Additionally, $M[\bmrho]$ is a Maxwellian distribution defined for generic moments $\bmrho(\bmx)\in\bbR^{d+2}$ by
\begin{equation}
	M[\bmrho(\bmx)](\bmv) = \frac{n(\bmx)}{(2\pi\theta(\bmx))^{d/2}}\exp{\bigg(\frac{-|\bmv-\bmu(\bmx)|^2}{2\theta(\bmx)}\bigg)},
\end{equation} 
where the density $n$, bulk velocity $\bmu$, and temperature $\theta$ are the fluid variables that are defined through the moments $\bmrho$ by
\begin{equation}
	n = \rho_0,\qquad \bmu = \frac{(\rho_1,\ldots,\rho_d)^\top}{n},\qquad \theta = \frac{1}{d}\Big(\frac{2\rho_{d+1}}{n} - \|\bmu\|^2\Big).
\end{equation}
We equip \eqref{eqn:bgk} with either periodic boundary conditions or inflow boundary conditions where the outflow(+) and inflow(-) boundaries are given, respectively, by
\begin{equation}
\Omega_\pm = \{(\bmx,\bmv)\in\partial\Omega_x\times\mathbb{R}^d:\pm\bmv\cdot\bfn(\bmx) >0\},
\end{equation}
and $\bfn$ is the unit outward normal of $\partial\Omega_x$.

An important property of the BGK collision operator $f \mapsto \nu(M[\bmrho_f]-f)$ is the conservation of number, momentum, and energy, i.e., for any $\bmx\in\Omega_{x}$
\begin{equation}\label{eqn:BGK_conservation}
    \int_{\bbR^d} {\bfe}(\bmv)(M[\bmrho_f(\bmx)](\bmv)-f(\bmx,\bmv))\dx{\bmv} = {\bf 0}.
\end{equation}

\subsection{Notation}\label{subsec:notation}
Let $\ell_x\in\bbN_0 = \{0,1,2,\ldots\}$, $\W_{\bmx}=(-L_x,L_x)^d\subset \bbR^d$ be the domain in position space, and $\cT_{x,\ell_x}$ be a uniform tensor mesh on $\W_x$ with $2^{\ell_x}$ elements in each dimension.
Let $\cE_{x,\ell_x}$ be the skeleton of $\cT_{x,\ell_x}$.  

Similarly, let $\ell_v\in\bbN_0$, $\Omega_v=(-L_v,L_v)^d \subset \bbR^d$, and $\cT_{v,\ell_v}$ be a uniform tensor mesh on $\W_v$ with $2^{\ell_v}$ elements in each dimension.

We let $\W=\W_x\times\W_v\subset\bbR^{2d}$, and denote $L^2(\W)$ and $H^s(\W)$ to be the standard Lebesgue and Sobolev spaces on $\W$.  
Let $(\cdot\,,\cdot)$ be the $L^2(\W)$-inner product with norm $\|\cdot\|_{L^2(\W)}$.
We denote by $L^2(D)$ and $(\cdot\,,\cdot)_D$ the $L^2$ space with standard inner product on some domain $D$ which is typically $\W_x$ or $\W_v$.  Any of the inner products mentioned above can be trivially extended to vector-valued functions with the standard Euclidean inner product.  

Denote the discontinuous Galerkin finite element spaces $V_{x,\ell_{x}}\subset L^2(\W_x)$ and $V_{v,\ell_{v}}\subset L^2(\W_v)$ by 
\begin{equation}\label{eqn:DG_spaces}
\begin{split}
    V_{x,\ell_x} &= \{g\in L^2(\W_x):g\big|_{K} = \bbQ_k(K)~\forall K\in\cT_{x,\ell_x}\}, \\
    V_{v,\ell_v} &= \{g\in L^2(\W_v):g\big|_{K} = \bbQ_k(K)~\forall K\in\cT_{v,\ell_v}\}
\end{split}
\end{equation}
where $\bbQ_k(K)$ is the set of all polynomials of maximum degree $k$ in any direction on $K$.
Let $\mathcal{V}_\ell=V_{x,\ell_x}\otimes V_{v,\ell_v}$.  

Given an interior edge $e\in\cE_{x,\ell_x}$ and $\bmx\in e=\partial K^+\cap\partial K^-$ with unit outward normals $\bn^{\pm}$, let $g$ be a function with traces $g^\pm(\bmx_*):=\lim_{K^\pm\ni \bmx\to \bmx^*}g(x)$ well defined. 
Define the average and jump of $g$ in $\bmx$, respectively, by
\begin{equation}\label{eqn:avg_jmp_x}
    \avg{g(\bmx)} = \tfrac{1}{2}(g^+(\bmx)+g^-(\bmx))
    \qquad~\text{and}~\qquad
    \ljmp{g(\bmx)}\rjmp = g^-(\bmx)\bn^-+g^+(\bmx)\bn^+.
\end{equation}
We account for the periodic boundary in $\cE_{x,\ell_x}$ by defining the jumps and averages on the boundary using the appropriate periodic boundary.

Let $\ledge\cdot\,,\cdot\redge_e$ be the $L^2$ inner product over an edge $e$ and denote
$\ledge\cdot\,,\cdot\redge_{\cE_{x,\ell_x}}=\sum_{e\in\cE_{x,\ell_x}}\ledge\cdot,\cdot\redge_e$.
The skeleton inner product can be extended to phase-space edge $e\times \Omega_v$ for $e\in\cE_{x,\ell_x}$ denoted by $\ledge\cdot\,,\cdot\redge_{\cE_{x,\ell_x}\times\Omega_v}$
For functions $g$ in $\mathcal{V}_\ell$, let $\grad_x$ represent the piece-wise spatial derivative $g$.

Finally, for time integration, let $\Delta t>0$ be the timestep, assumed for our purposes to be uniform.  For $\mfn\in\bbN_0$ define $t^\mfn=\mfn\Delta t$ and denote $f^\mfn$ to be an approximation to $f(t^\mfn)$.

\subsection{Discontinuous Galerkin method}\label{subsec:discretization}

We first discretize \eqref{eqn:bgk} in phase space on a to-be-determined adaptive sparse-grid DG space paramaterized by a grid $\Theta$ (see \Cref{subsec:adaptive_sparse_grids}) $\bmV_{\Theta}\subseteq\mathcal{V}_\ell$ by the following semi-discrete problem:  Find $f_h\in C([0,\infty];\bmV_\Theta)$ such that

\begin{equation}\label{eqn:discrete_form}
    (\partial_t f_h,g_h) + 
    \AP(f_h,g_h) + \nu(f_h,g_h) = \nu(M_{\Theta}[\bmrho_f],g_h)
\end{equation}
holds for all $g_h\in \bmV_\Theta$.
The Vlasov portion, $\AP$, is discretized with upwind fluxes; specifically, 
\begin{align}\label{eqn:VP_discrete_def}
    \AP(w_h,g_h) &= -(\bmv w_h,\grad_x g_h) + \ledge \bmv\avg{w_h}+\tfrac{|\bmv\cdot\bn|}{2}\jmp{w_h},\jmp{g_h}\redge_{\cE_{x,\ell_x}\times\W_v}
\end{align}
for all $w_h,g_h\in \bmV_\Theta$.
If the boundary conditions are inflow/outflow, we instead prescribe the inflow data in \eqref{eqn:VP_discrete_def} on $\Omega_-$.
The discrete Maxwellian $M_{\Theta}[\bmrho_f]\in \bmV_\Theta$ uses the hybrid interpolation procedure that is defined in \Cref{subsec:discrete_maxwellian}.

\subsection{Time stepping method}\label{subsec:time_stepping}

We discretize \eqref{eqn:discrete_form} in time using Implicit-Explicit (IMEX) Runge--Kutta (RK) methods \cite{ascher1997implicit}.  
Such methods are popular time steppers for evolving kinetic models that feature multiple time scales \cite{pareschiRusso_2005,chu2019realizability,endeveHauck_2022}.  
The Vlasov portion $\AP$ is evolved explicitly and the BGK collision operator is evolved implicitly.
We use the IMEX-RK method of \cite{chu2019realizability} which is given for any $g_h\in\bmV_\Theta$ by:
\begin{subequations}\label{eqn:IMEX_RK}
\begin{align}
(f_h^{(1,*)},g_h) &= (f_h^{\mathfrak{n}},g_h) - \Delta{t}\AP(f_h^{\mathfrak{n}},g_h) \label{eqn:IMEX_RK:ex1}, \\
(f_h^{(1)},g_h) + \nu\dt (f_h^{(1)},g_h)  &= (f_h^{(1,*)},g_h) + \nu\Delta t(M_\Theta[\bm{\rho}_{f_h^{(1)}}],g_h), \label{eqn:IMEX_RK:im1}\\ 
(f_h^{(2,*)},g_h) &= \tfrac{1}{2}(f_h^{\mathfrak{n}},g_h) + \tfrac{1}{2}\big((f_h^{(1)},g_h)-\Delta{t}\AP(f_h^{(1)},g_h)\big), \label{eqn:IMEX_RK:ex2} \\
(f_h^{(2)},g_h) + \tfrac12 \nu\Delta t(f_h^{(2)},g_h) &= (f_h^{(2,*)},g_h) + \tfrac{1}{2}\nu\Delta t(M_\Theta[\bm{\rho}_{f_h^{(2)}}],g_h), \label{eqn:IMEX_RK:im2}
\end{align}
and $f_h^{\mathfrak{n}+1}:=f_h^{(2)}$.
\end{subequations}

\begin{remark}
As written, \eqref{eqn:IMEX_RK:im1} and \eqref{eqn:IMEX_RK:im2} are non-linear implicit equations due to the Maxwellian.
However, if the discrete Maxwellian is constructed to satisfy \eqref{eqn:BGK_conservation}, then $\bmrho_{f_h^{(\cdot,*)}} = \bmrho_{f_h^{(\cdot)}}$.
By substituting in the explicit moments $\bmrho_{f_h^{(\cdot,*)}}$, the right-hand sides of \eqref{eqn:IMEX_RK:im1} and \eqref{eqn:IMEX_RK:im2} are sources, and thus the respective implicit solves can be computed via a scalar multiplication.
This trick, first introduced in \cite{coron}, is standard in the literature.
\end{remark}

\section{Adaptive sparse grids}
\label{sec:sparse-grid}

In this section we provide an overview of the adaptive sparse-grid DG method used for the BGK model.
The method utilizes the wavelet basis to construct a hierarchical and multi-resolution representation of the standard tensor-based full-grid DG space \cite{guo2016sparse}.
The hierarchical representation induces a decay in the associated expansion coefficients over finer levels that is used to adaptively remove or add basis functions in time \cite{guo2017adaptive}.
Below is an abridged treatment of the sparse-grid DG method; we refer the reader to \cite{schnake2024SparsegridDiscontinuous} for a more in-depth presentation.
\Cref{subsec:1D_basis,subsec:nD_basis} construct the single and multi-dimension wavelet basis respectively.
\Cref{subsec:adaptive_sparse_grids} constructs the adaptive grid $\Theta$ and the associated adaptive sparse-grid DG space $\bmV_{\Theta}$.

\subsection{Single dimension multiwavelet basis}\label{subsec:1D_basis}

The one-dimensional multiwavelet basis is a hierarchical basis in which additional basis functions for resolving fine scale features are introduced using orthogonal complements to current functions in the basis.  
To simplify the presentation, we assume a domain $\W=[0,1]$. Given a level $\ell\in\{0,\ldots,N\}$, let $\mT_\ell$ be a uniform mesh on $\W$ with mesh size $h_\ell=2^{-\ell}$.  
The partition of $\mT_\ell$ is characterized by the union of disjoint intervals $I_{\ell,j}:=(2^{-\ell}j,2^{-\ell}(j+1))$ for $j=0,\ldots,2^\ell-1$.  
Given this mesh, define the corresponding DG finite element space $V_\ell$ by
\begin{equation}\label{eqn:1D_DG_space}
    V_\ell := V_\ell^k = \left\{ g\in L^2(\W) : g\big|_{I_{\ell,j}} \in \mathbb{P}_k(I_{\ell,j})~\forall j=0,\ldots,2^\ell-1\right\},
\end{equation}
where $\mathbb{P}_k$ is the space of polynomials of degree up to $k$.  This space has dimension $\mathrm{dim}(V_\ell) = 2^\ell(k+1)$.  
Additionally, due to the uniform partitioning, 
\begin{equation}
    V_0 \subset V_1\subset V_2 \subset \cdots \subset V_N .
\end{equation}

Let $W_\ell$ be the orthogonal complement of $V_{\ell-1}$ in $V_\ell$ with respect to the $L^2(\W)$ inner product; that is, $W_0 = V_0$, while for $\ell \geq 1$, 
\begin{equation}\label{eqn:W_def}
V_\ell = V_{\ell-1}\oplus W_\ell
\qquad\text{and}\qquad
W_\ell \perp V_{\ell-1},
\end{equation}
where $\oplus$ is the direct sum and $\mathrm{dim}(W_\ell) = \max\{0,2^{\ell-1}(k+1)\}$.  Then 
\begin{equation}\label{eqn:hierarchical_decomp_1d}
    V_N = \bigoplus_{\ell=0}^N W_\ell.
\end{equation}

A standard choice for the basis of $W_\ell$ for $\ell\geq 1$ are \textit{wavelets} -- functions that are scaled and shifted to capture finer-scale features.  
The prototype wavelet is the piece-wise constant Haar basis \cite{haar1909theorie}. Here we use the Alpert wavelets \cite{alpert1993class}.

\begin{defn}\label{defn:alpert}
The Alpert wavelets are a set of a functions $\{\phi_i(z):i=1,\ldots,k+1\}\subset L^2(\mathbb{R})$ with support in $[-1,1]$ and defined such that
\begin{enumerate}
    \item $\phi_i\big|_{(0,1)}\in\mathbb{P}_k(0,1)$.
    \item $\phi_i(z) = (-1)^{i+k}\phi_i(-z)$.
    \item $\int_{-1}^1\phi_i(z) y^j\dx{y} = 0$ for all $j=0,1,\ldots,i+k-1$. \label[property]{enum:polynomial_orthogonality}
    \item $\int_{-1}^1 \phi_i(z)\phi_j(z) \dx{y} = \delta_{ij}$ for all $i,j=1,\ldots,k+1$ where $\delta_{ij}$ is the Kronecker delta.
\end{enumerate}
\end{defn}
Construction of the wavelets and examples for various polynomial degrees can be found in \cite[Page 5]{alpert1993class}. 
For each $\ell \geq 0$, we use the Alpert wavelets to define a basis set $\{g_{\ell,j}^{i}\}$ of $W_\ell$. For $\ell=0$, we choose $g_{0,0}^i$ to be the shifted Legendre polynomials normalized on $L^2(\W)$.  For $\ell \geq 1$, we shift and rescale the Alpert wavelets so that for each $z\in(0,1)$,
\begin{equation}\label{eqn:wavelet_basis_1D_def}
    g_{\ell,j}^i(z) = 2^{(\ell-1)/2}\gamma_i(2^{\ell-1}z-j) 
    ,\quad \text{where}
    \quad
    \gamma_i(z):=\sqrt{2}\phi_i(2z-1).
\end{equation}
Here $\ell$ is the level, $j=0,\ldots,2^{\ell-1}-1$ is the level index, and $i=1,\ldots,k+1$ is the polynomial index.  
\Cref{enum:polynomial_orthogonality} of \Cref{defn:alpert} ensures that the wavelet bases are all orthonormal; that is,
\begin{equation}\label{eqn:wavelet_orthonormality}
    \int_0^1 g_{\ell,j}^i(z)g_{\ell',j'}^{i'}(z) \dx{z} = \delta_{ii'}\delta_{\ell\ell'}\delta_{jj'}.
\end{equation}

Although the wavelet representation is just a change of basis fromstandard nodal/modal DG basis expansions, each successive $W_\ell$ comes with a decrease of the support of the basis functions, which leads to a rapid decay of the coefficients of the orthogonal projection.
This is critical for the adaptive procedure shown in \Cref{subsec:adaptive_sparse_grids} as it allows for a fine-grained control of the approximation error, especially in the multi-dimensional context.

\subsection{Multiwavelet basis in multiple dimensions}\label{subsec:nD_basis}

A $d$-dimensional basis is achieved through a tensor product extension.  
Let $\W^{d}=(0,1)^d$ with $\bm{z}=(z_1,\ldots,z_d)\in\W^{d}$. 
Let $\bm{\ell}=(\ell_1,\ldots,\ell_d)\in\mathbb{N}_0^d$ and $\bmj=(j_1,\ldots,j_d)\in\mathbb{N}_0^d$ be multi-index sets, where $\ell_m$ and $j_m$ defines the level and index respectively for dimension $m \in\{1, \dots, d\}$, and define $|\bm{\ell}|_\infty = \max_{1\leq m\leq d}\ell_m$ and $|\bm{\ell}|_1=\ell_1+\cdots+\ell_d$.  Let $\mathcal{T}_{\bm{\ell}}$ be a tensor product mesh with multi-dimensional mesh parameter $\bm{h}:=(2^{-\ell_1},\ldots,2^{-\ell_d})$.
We label all elements in $\mathcal{T}_{\bm{\ell}}$ by $I_{\bm{\ell},\bm{j}} = \{\bmz:z_m\in(2^{-\ell_m}j_m,2^{-\ell_m}(j_m+1)\}$ and define the tensor product finite element space by
\begin{equation}
    \bm{V_\ell}:=\bm{V_\ell}^k=\{g\in L^2(\W):g\big|_{I_{\bm{\ell},\bm{j}}}\in\mathbb{Q}_k(I_{\bm{\ell},\bm{j}}),~\forall~0\leq j_m\leq 2^{\ell_m}-1, m=1,\ldots,d\},
\end{equation}
where $\mathbb{Q}_k(I_{\bm{\ell},\bm{j}})$ represents the set of polynomials of degree up to $k$ in each dimension on $I_{\bm{\ell},\bm{j}}$.

Recall the one-dimensional hierarchical decomposition in \Cref{subsec:1D_basis}.  Given the complementary sets $W_{\ell_m}$ defined in \eqref{eqn:W_def}, let
\begin{equation}
    \bm{W_\ell} = W_{\ell_1}\otimes W_{\ell_2}\otimes\cdots\otimes W_{\ell_d}.
\end{equation}
The multiwavelet basis that we choose for $\bm{W_\ell}$ is constructed from products of the 1D wavelets in \eqref{eqn:wavelet_basis_1D_def}:
\begin{equation}\label{eqn:wavelet_basis_nD_def}
    g_{\bm{\ell},\bm{j}}^{\bm{i}}(\bm{z}) := \prod_{m=1}^d g_{\ell_m,j_m}^{i_m}(z_m),~\text{where}~j_m=0,\ldots,\max\{0,2^{\ell_m-1}-1\},i_m=1,\ldots,k+1.
\end{equation}
It follows from repeated application of \eqref{eqn:wavelet_orthonormality} in each dimension that these multiwavelets are orthonormal in $L^2(\W)$.
We use $c_{\bmell,\bmj}:=c_{\bmell,\bmj}^{\bmi}$ to denote the $(k+1)^d$ coefficients for the multiwavelet expansion on $\bmW_{\bmell,\bmj}$ with respect to $g_{\bmell,\bmj}^{\bmi}$.

Using $\bm{W}_\ell$, the full-grid space at level $N$ can be written as
\begin{equation}\label{eqn:full_grid_def}
    \bm{V}_N = \bigoplus_{|\bm{\ell}|_\infty\leq N}\bm{W_\ell}.
\end{equation}
This space has dimension $\mathrm{dim}(\bm{V}_N) = (k+1)^d2^{Nd}$.
The sparse grid space $\hat{\bm{V}}_N \subseteq \bm{V}_N$ at level $N$ is formed by the following relaxation of the index norm \cite{wang2016sparse,Bungartz_Griebel_2004}:
\begin{equation}\label{eqn:sparse_grid_def}
    \hat{\bm{V}}_N = \bigoplus_{|\bm{\ell}|_1\leq N}\bm{W_\ell}.
\end{equation}    
It was shown in \cite[Lemma 2.3]{wang2016sparse} that  
$\textrm{dim}(\hat{\bm{V}}_N)=\mathcal{O}((k+1)^d 2^N N^{d-1})$.

\subsection{Adaptive sparse grid construction}\label{subsec:adaptive_sparse_grids}

The sparse grid DG space provides a significant savings in memory and compute for larger dimensions as compared to the full-grid degrees of freedom, while still yielding a comparable approximation error due to the $\ell^1$ norm being the optimal choice of indexes based on a priori asymptotic estimates.
However, applying the sparse-grid DG space to the non-linear BGK model leads to numerical instabilities and negative density and temperature, particularly in the pre-asymptotic regime.
Even restricting the sparse-grid reduction to only the velocity DoFs has been shown to not efficiently capture higher order moments of the distribution \cite{schnake2024SparsegridDiscontinuous}.

In order to efficiently resolve the solution in phase space, we employ the adaptive sparse-grid DG method to dynamically identify the optimal set of sparse-grid indexes for a specific function $f$ \cite{guo2017adaptive}.
The method uses the decay in the coefficients $c_{\bmell,\bmj}$ of $f$ over finer levels and parent/child relations as a heuristic to build an adaptive grid $\Theta$ by adding/removing elements from $\Theta$ at every $\mfn$.

The first step is to further decompose the orthogonal complements $\bm{W_\ell}$ by their level $\bm{\ell}$ and position $\bm{j}$ within the level.
This position $\bm{j}$ in the level is based on the multiwavelet basis.
Given the basis in \eqref{eqn:wavelet_basis_nD_def}, we define the space $\bm{W_{\ell,j}}\subset \bm{W_\ell}$ by
\begin{equation}\label{eqn:adaptive_basis_def}
    \bm{W_{\ell,j}} = \spann_{\substack{1\leq i_m\leq k+1 \\ 1\leq m\leq d}} \{g_{\bm{\ell,j}}^{\bm{i}}\}.
\end{equation}
This space has dimension $\textrm{dim}(\bm{W_{\ell,j}})=(k+1)^d$ and
\begin{equation}
    \bm{W_\ell} = \bigoplus_{\bm{j}\in\mathcal{B}_{\bm{\ell}}} \bm{W_{\ell,j}}
    \quad\text{where}\quad
    \mathcal{B}_{\bm{\ell}} := \{\bm{j}=(j_1,\ldots,j_d):j_m=0,\ldots,\max\{0,2^{\ell_m-1}-1\},\forall m=1,\dots, d\}.
\end{equation}

We define adaptive sparse-grid spaces using $\bm{W}_{\bm{\ell},\bm{j}}$.
\begin{defn}[Adaptive Sparse Grid]\label{defn:adapt_grid}
    Let $\bmN=(N_1,\ldots,N_d)$ be the max level in each dimension and let $\Theta=\{(\bm{\ell}^\iota, 
    \bm{j}^\iota)\}_{\iota=1}^M$ denote an adaptive grid such that $\ell_m^\iota \leq N_m$ and $\bm{j}^\iota\in\mathcal{B}_{\bm{\ell^\iota}}$ for all $\iota \in\{1,\dots, M\}$.  Then
    the adaptive sparse-grid DG space $\bm{V}_{\Theta}\subseteq \bm{V}_{\bm{N}}$ is defined as 
    \begin{equation}\label{eqn:adaptive_grid_def}
        \bm{V}_\Theta = \bigoplus_{(\bmell,\bmj)\in\Theta} \bm{W}_{\bm{\ell},\bm{j}}.
    \end{equation}
    Here $M=|\Theta|$ is said to be the number of active elements of the adaptive sparse-grid space $\bm{V}_{\Theta}$, and the degrees of freedom of $\bmV_\Theta$ is given by $\dim(\bmV_\Theta)=M (k+1)^{d}$.
\end{defn}

The hierarchy definition of $\bmW_{\bmell}$ additionally provides a parent/children relationship defined by the following.

\begin{defn}[Children, Parents, and Ancestors]\label{defn:parents_and_children}
    Let $\Theta$ be an adaptive grid with max level $\bm{N}\in\bbN_0^{d}$.  The children of $(\bm{\ell},\bm{j})$, with up to two per dimension, are defined for each dimension $m=1,\ldots,d$ by the following:
    \begin{itemize}
        \item If $\ell_m=0$, then 
        \begin{equation}
            (\bm{\ell}',\bm{j}')=\big( (\ell_1,\ldots,\ell_{m-1},1,\ell_{m+1},\ldots,\ell_d),(j_1,\ldots,j_{m-1},0,j_{m+1},\ldots,j_d) \big),
        \end{equation} is a child of $({\bm{\ell},\bm{j}})$.
        \item If $0<\ell_m<N_{m}$, then
        \begin{subequations}
          \begin{align}
            (\bm{\ell}',\bm{j}')&=\big( (\ell_1,\ldots,\ell_{m-1},\ell_m+1,\ell_{m+1},\ldots,\ell_d),(j_1,\ldots,j_{m-1},2j_m,j_{m+1},\ldots,j_d) \big)  \\
            \shortintertext{and}
            (\bm{\ell}'',\bm{j}'') &=\big( (\ell_1,\ldots,\ell_{m-1},\ell_m+1,\ell_{m+1},\ldots,\ell_d),(j_1,\ldots,j_{m-1},2j_m+1,j_{m+1},\ldots,j_d) \big),
        \end{align}   
        \end{subequations}
        are children of $(\bm{\ell},\bm{j})$.
        \item If $\ell_m=N_{m}$, then there are no children of $(\bm{\ell},\bm{j})$ in dimension $m$.
    \end{itemize}
    The parents of an element $(\bm{\ell},\bm{j})$ are all elements $(\bm{\ell}',\bm{j}')$ such that $(\bm{\ell},\bm{j})$ is a child of $(\bm{\ell}',\bm{j}')$.
    An element $(\bm{\ell}',\bm{j}')$ is an ancestor of $(\bm{\ell},\bm{j})$ if there is a link from $(\bm{\ell},\bm{j})$ to $(\bm{\ell}',\bm{j}')$ through parents only, i.e., a link is formed through parents, grandparents, etc.
    We say a grid $\Theta$ is ancestor complete provided that all parents of $(\bmell,\bmj)$ are in $\Theta$ whenever $(\bmell,\bmj)\in\Theta$.
\end{defn}

We postpone the coarsening/refinement procedure of $\Theta$ until \Cref{subsec:refine_and_coarsen}.

\section{Adaptive sparse-grid DG interpolation method}
\label{sec:interp}

In this section we specify Lagrange interpolation methods to evaluate the discrete Maxwellian $M_{\Theta}[\bmrho_f]$ (see \Cref{subsec:discretization}) on the adaptive sparse-grid DG space $\bmV_\Theta$.
Operators that are separable, i.e., can be formed as a product of 1D operators (for example $f\to v_m\partial_{x_m}f$), can be efficiently evaluated in an adaptive and non-adaptive sparse grid \cite{huang2023adaptive}.
However, the non-trivial coupling of $x$ and $v$ in the  Maxwellian is not separable.

The standard sparse-grid DG interpolation operation $I_\Theta$, given in  \Cref{subsec:dg_interp}, is derived from the work in \cite{huang2020adaptive,huang2023adaptive,tao2021AdaptiveHighorder} and is used as a basis to define the conservative discrete Maxwellian $M_\Theta$ in \Cref{subsec:discrete_maxwellian}.

\subsection{Standard adaptive sparse-grid DG interpolation}
\label{subsec:dg_interp}

The interpolation multiwavelet basis is constructed in a similar fashion to the Alpert basis $g_{\bmell,\bmj}$ in \Cref{sec:sparse-grid} and which we explain in \Cref{subsec:interp_basis_construction}.

\subsubsection{Single dimension construction}
\label{subsec:interp_basis_construction}

We first construct a set of hierarchical (nested) interpolation points in 1D.
Consider a set of $k+1$ interpolation points $Z_0 = \{z^i\}_{i=1}^{k+1}$ on $[0,1]$. 
Define the points $Z_\ell = \{z^i_{\ell,j}\}$ over levels $\ell\geq 1$ by 
\begin{align}
    z_{\ell,j}^i = 2^{-\ell}(z^i+j),~\text{for}~ i=1,\ldots,k+1,~ j=0,\ldots,2^\ell-1.
\end{align}
Setting $z_{0,0}^i:=z^i$ gives us $\{z_{\ell,j}^i\}$ for all levels $\ell\geq 0$.
We require the points to be nested, that is, $Z_0\subset Z_1\subset Z_2\subset\cdots\subset Z_N$; if $Z_0\subset Z_1$ (which is always the case for the points we use), then the rest of the nesting following automatically.  

Define $\widetilde{Z}_1 = Z_1\setminus Z_0 = \{\tilde{z}_{1,0}^1,\ldots,\tilde{z}_{1,0}^{k+1}\}$.
Define 
\begin{align}
    \widetilde{Z}_\ell = \{ Z_{\ell,j}^i := 2^{-(\ell-1)}(j+\tilde{z}_{1,0}^i),~i=0,\ldots,k+1,~j=0,\ldots,2^{\ell-1}-1\}.
\end{align}
Note that $\widetilde{Z}_{\ell} = Z_\ell\setminus Z_{\ell-1}$;  thus by defining $\widetilde{Z}_0 = Z_0$, we have $Z_\ell = \cup_{\ell'=0}^\ell \widetilde{Z}_{\ell'}$.

We now define the interpolatory wavelets.
\begin{defn}\label{defn:interpolatory_wavelets}
    The set of functions $\{\varphi_i\}_{i=1}^{k+1}\subset W_1$ are interpolatory wavelets on $\widetilde{Z}$ with respect to $\widetilde{Z}_0$ and $\widetilde{Z}_1= \{\tilde{z}_{1,0}^1,\ldots,\tilde{z}_{1,0}^{k+1}\}$ provided
    \begin{equation}\label{eqn:interpolatory_wavlets}
        \varphi_i(\tilde{z}_{1,0}^{i'}) = \delta_{ii'}\qquad\text{and}\qquad\varphi_i(\widetilde{z})=0\quad\forall \widetilde{z}\in \widetilde{Z}_0.
    \end{equation}
\end{defn}

We list the interpolation points $\widetilde{Z}_0$ and $\widetilde{Z}_1$ used in the numerical experiments of \Cref{sec:numerical_experiments} in \Cref{tab:interp_points}.
Alternative points are also admissible, but in our experience the points listed yield the smallest interpolation errors of any we tried.

\begin{table}[h!]
    \centering
    \renewcommand{\arraystretch}{1.5}
    \caption{ Interpolation points $\widetilde{Z}_0$ and $\widetilde{Z}_1$ used in \Cref{sec:numerical_experiments} for each polynomial degree.
    Since each interpolatory wavelet in \Cref{defn:interpolatory_wavelets} is possibly discontinuous at $z=\tfrac12$, the $\pm$ denotes the side of the discontinuity for which \eqref{eqn:interpolatory_wavlets} must be satisfied.}
    \begin{tabular}{c|c|c}
        $k$ & $\widetilde{Z}_0$ & $\widetilde{Z}_1$ \\ \hline
         1  & $\tfrac13$, $\tfrac23$ & $\tfrac16$, $\tfrac56$ \\
         2  & 0, ${\tfrac12}^-$, 1 & $\tfrac14$,$ {\tfrac12}^+$, $\tfrac34$ \\
         3  & 0, $\tfrac16$, $\tfrac23$, 1 & $\tfrac16$, ${\tfrac12}^-$, ${\tfrac12}^+$, $\tfrac56$ \\
    \end{tabular}
    \label{tab:interp_points}
\end{table}

We construct the interpolative multiwavelet expansion in a similar manner to the Alpert wavelet basis, that is,
\begin{align}
    \psi_{\ell,j}^{i}(z) = \varphi_{i}(2^{\ell-1}z-j)
\end{align}
for $i=1,\ldots,k+1$, and $j=0,\ldots,2^{\ell-1}-1$, and span over $i,j$ to yield a basis for $W_\ell$.
Using the Lagrange nodal functions on $\widetilde{Z}_0$, denoted $\{\psi_{0,0}^i\}$, which give a basis for $W_0$, we have a basis for $V_N$ given by $\{\psi_{\ell,j}^i\}$ over all $\ell$ and appropriate $j$.
We use $\{b_{\ell,j}^i\}$ to denote a coefficient expansion with respect to $\{\psi_{\ell,j}^i\}$.

An important property of the interpolatory wavelet construction is that
\begin{align}\label{eqn:wavelet_sparsity}
    \psi_{\ell',j'}^{i'}(\tilde{z}_{\ell,j}^{i}) = 0~\forall \ell<\ell'~~\text{and}~~\psi_{\ell,j'}^{i'}(\tilde{z}_{\ell,j}^{i}) = \delta_{ii'}\delta_{jj'}.
\end{align}
This property is used to determine the coefficients $\{b_{\ell,j}\}$ of a function $f$ from $f(\widetilde{z}_{\ell,j})$ via a triangular solve (see \cite[Section 2.3]{tao2021AdaptiveHighorder}). 

\subsubsection{Multi-dimension construction}
\label{subsec:interp_basis_multi_construction}

We extend to $d$ dimensions in the same manner as the Alpert wavelets. 
Given $(\bmell,\bmj)$, denote the $(k+1)^d$ interpolation points by
\begin{equation}\label{eqn:Z_ell_j}
    \widetilde{\bmZ}_{\bmell,\bmj} = \{\widetilde{\bmz}_{\bmell,\bmj}^{\bmi} = (\widetilde{z}_{\ell_1,j_1}^{i_1},\ldots,\widetilde{z}_{\ell_d,j_d}^{i_d}):1\leq i_m\leq k+1\}.
\end{equation}
The $d$-dimensional interpolatory multiwavelets are given by
\begin{equation}\label{eqn:interpolets}
    \psi_{\bmell,\bmj}^{\bmi}(\bmz) = \prod_{m=1}^d \psi_{\ell_m,j_m}^{i_m}(z_m).
\end{equation}
Let $b_{\bmell,\bmj}:=b_{\bmell,\bmj}^{\bmi}$ denote the coefficient expansion with respect to $\psi_{\bmell,\bmj}^{\bmi}$.
Note that the interpolation property \eqref{eqn:wavelet_sparsity} extends to the multidimensional case through the following proposition of which we omit the proof.
\begin{prop}\label{prop:wavelet_sparsity_multidim}
The functions $\{\psi_{\bmell,\bmj}\}_{(\bmell,\bmj)}$ satisfy the following property:
\begin{enumerate}
    \item$\psi_{\bmell,\bmj}^{\bmi}(\widetilde{\bmz})=0$ for all $\widetilde{\bmz}\in\widetilde{\bmZ}_{\bmell',\bmj'}$ where $(\bmell',\bmj')$ is an ancestor of $(\bmell,\bmj)$.
    \label[property]{prop:wavelet_sparsity_multidim:1}
    \item $\psi_{\bmell,\bmj}^{\bmi'}(\widetilde{\bmz}_{\bmell,\bmj}^{\bmi})=\delta_{\bmi\bmi'}$.
    \label[property]{prop:wavelet_sparsity_multidim:2}
    \item Let $f=\sum_{(\bmell,\bmj)\in\Theta}\sum_{\bmi}b_{\bmell,\bmj}^{\bmi}\psi_{\bmell,\bmj}^{\bmi}$ for an ancestor complete grid $\Theta$, then $b_{\bmell,\bmj}$ is a linear combination of $f(\widetilde{\bmz})$ for $\widetilde{\bmz}\in\widetilde{\bmZ}_{\bmell,\bmj}$ and $b_{\bmell',\bmj'}$ for all ancestors $(\bmell',\bmj')$ of $(\bmell,\bmj)$.
    \label[property]{prop:wavelet_sparsity_multidim:3}
\end{enumerate}
\end{prop}

\Cref{prop:interp_equiv} gives the following representation between function values on $\widetilde{\bmZ}$ and basis coefficients for $g_{\bmell,\bmj}$ and $\psi_{\bmell,\bmj}$.
The equivalences shown are vital for the interpolation procedure in \Cref{subsec:adaptive_interpolation}.

\begin{prop}\label{prop:interp_equiv}
Let $\Theta$ be an ancestor complete set.
Denote $\widetilde{\bmZ}_{\Theta}=\cup_{(\bmell,\bmj)\in\Theta}\widetilde{\bmZ}_{\bmell,\bmj}$.
Let $f\in\bmV_\Theta$, then $f$ is uniquely determined by any of the following three sets
\begin{equation}
    f(\widetilde{\bmZ}_{\Theta}):=\{f(\widetilde{\bmz}):\widetilde{\bmz}\in\widetilde{\bmZ}_{\Theta}\},\qquad b_\Theta:=\{b_{\bmell,\bmj}:(\bmell,\bmj)\in\Theta\},\qquad\text{and}\qquad c_\Theta:=\{c_{\bmell,\bmj}:(\bmell,\bmj)\in\Theta\}.
\end{equation}
\end{prop}

\begin{proof}
    Since the Alpert wavelets $\{g_{1,0}^i\}_{i=1}^{k+1}$ and the interpolatory wavelets $\{\varphi^i\}_{i=1}^{k+1}$ both form a basis for $W_1$, then $\{g_{\bmell,\bmj}^{\bmi}\}_{\bmi}$ and $\{\psi_{\bmell,\bmj}^{\bmi}\}_{\bmi}$ both form a basis for $W_{\bmell,\bmj}$.  Hence, their respective collections over $(\bmell,\bmj)\in\Theta$ form a basis for $\bmV_\Theta$.
    Therefore both $b_\Theta$ and $c_\Theta$ uniquely determine from $f\in\bmV_\Theta$.

    To show the unisolvency of $f(\widetilde{\bmZ}_\Theta)$, it is sufficient to show that if $f\in \bmV_\Theta$ such that $f(\widetilde{\bmz})=0$ for all $\widetilde{\bmz}\in \widetilde{\bmZ}_{\Theta}$, then $f\equiv 0$.
    Since $b_{\bf{0,0}}$ are formed from the Lagrange nodal basis on $\widetilde{\bmZ}_{\bf 0,0}$, then $b_{\bf{0,0}}={\bf 0}$.
    We then proceed to all children by strong induction. 
    Fix $(\bmell,\bmj)$ and assume $b_{\bmell',\bmj'}={\bf 0}$ for all ancestors $(\bmell',\bmj')$ of $(\bmell,\bmj)$. 
    Then by \Cref{prop:wavelet_sparsity_multidim:3} of \Cref{prop:wavelet_sparsity_multidim}, $b_{\bmell,\bmj}$ only depends linearly on $f(\widetilde{\bmz})$ for $\widetilde{\bmz}\in\widetilde{\bmZ}_{\bmell,\bmj}$ which are all 0.  Hence, using \Cref{prop:wavelet_sparsity_multidim:1} of \Cref{prop:wavelet_sparsity_multidim}, $b_{\bmell,\bmj}={\bf 0}$.  Therefore $b_{\Theta}={\bf 0}$ and thus $f=0$ in $\bmV_\Theta$.
\end{proof}

\subsubsection{Adaptive Interpolation}
\label{subsec:adaptive_interpolation}

We now specify the adaptive interpolation that uses the function values on $\widetilde{\bmZ}_\Theta$:

\begin{defn}[Adaptive Interpolation]\label{defn:interpolation}
Let $\Theta$ be an adaptive index set. 
Given a function $\mathcal{F}(f)$ and $f\in\bmV_\Theta$ we construct the interpolation map $f\mapsto I_\Theta\mathcal{F}(f)\in\bmV_\Theta$ with respect to the Alpert basis, $\{g_{\bmell,\bmj}\}_{(\bmell,\bmj)\in\Theta}$, by the following steps:
\begin{enumerate}
    \item Function evaluation: Evaluate $f$ at $\widetilde{\bmZ}_\Theta$.\label[step]{defn:interpolation:evaluate}
    \item Collocation: Build $\mathcal{F}(f(\widetilde{\bmZ}_{\Theta})) = \{\mathcal{F}(f(\widetilde{\bmz})):\widetilde{\bmz}\in\widetilde{\bmZ}_{\Theta}\}$.
    \label[step]{defn:interpolation:collocate}
    \item Map to interpolation wavelet coefficients: Determine $b_\Theta$ given $\mathcal{F}(f(\widetilde{\bmZ}_{\Theta}))$.
    \label[step]{defn:interpolation:to_interpolets}
    \item Map to Alpert wavelet coefficients: Determine $c_\Theta$ from $b_\Theta$.
    \label[step]{defn:interpolation:to_alpert}
\end{enumerate}
\end{defn}

\begin{remark}
Concerning the adaptive interpolation $I_\Theta$ in \Cref{defn:interpolation}:
\begin{enumerate}
    \item \Cref{prop:interp_equiv} shows $I_\Theta$ is well-defined.  
    \item We refer the reader to \cite{tao2021AdaptiveHighorder} for error estimates.
    \item \Cref{defn:interpolation:evaluate,defn:interpolation:to_interpolets,defn:interpolation:to_alpert} can be computed efficiently, e.g., see \cite[Section 3]{tao2021AdaptiveHighorder} for an efficient implementation of \Cref{defn:interpolation:to_interpolets}.
    However, we omit the exact details of the steps for a future work.
    The efficiency of \Cref{defn:interpolation:collocate} is dependent on $\mathcal{F}$.
\end{enumerate}    
\end{remark}

\subsection{Discrete Maxwellian and hybrid interpolation}\label{subsec:discrete_maxwellian}

Using the interpolation procedure in \Cref{subsec:adaptive_interpolation}, we could build a discrete Maxwellian by applying the map $\mathcal{F}(f)(\bmx,\bmv) = M[\bmrho_f(\bmx)](\bmv)$. 
However, such a procedure preserves the discrete conservation property of the collision operator \eqref{eqn:BGK_conservation} only up to the consistency error of the grid.
This error is small for a tensor velocity grid, but for an adaptive sparse grid, where velocity resolution is culled for compression, the error can be several orders above the truncation tolerance (see \Cref{tab:conservation_errors}) and thus not sufficient for a robust method.
We note that the error incurred by $I_\Theta$ (and thus the conservation error) could be reduced by applying a higher polynomial order interpolation operator as done in \cite{huang2020adaptive}, but this does not fully resolve the conservation issue.

To build a conservative discrete Maxwellian, we describe a hybrid interpolation procedure that employs interpolation only in position space, but exploits the separability of the Maxwellian in velocity, i.e., for a given $\bmx$, $M[\bmrho(\bmx)](\cdot)$ is a product of single dimension functions in $\bmv$ (see \eqref{eqn:hybrid_interpolation:analytic_integral} in \Cref{defn:hybrid_interpolation}).

Using the above notation, let $\bmz=(\bmx,\bmv)\in\mathbb{R}^{2d}$.
Let $\Theta=\{(\bm{\ell_x,\ell_v,j_x,j_v})\}$ be an adaptive index set and let $\bmV_{\Theta}$ denote the corresponding adaptive sparse-grid space. 
Define the position grid $\Theta_{\bmx}$ by
\begin{equation}\label{eqn:position_grid}
    \Theta_{\bmx}=\Big\{(\bm{\ell_x,j_x}):(\bm{\ell_x,0,j_x,0})\in\Theta\Big\}.
\end{equation}
and associated position adaptive sparse-grid DG space $\bmV_{\Theta_{\bmx}}\subset V_{\bmx,N}$ and adaptive interpolation points $\widetilde{\bmX}_{\Theta_{\bmx}}$.

\begin{defn}[Hybrid Interpolation]\label{defn:hybrid_interpolation}
Let $\Theta$ be an ancestor complete adaptive index set. 
We define the adaptive hybrid interpolation of the Maxwellian  $\bmrho\in[\bmV_{\Theta_{\bmx}}]^{d+2}\mapsto M_\Theta[\bmrho]\in \bmV_{\Theta}$ by the following process:
\begin{enumerate}
    \item Evaluate $\bmrho(\widetilde{\bmx})$ for all $\widetilde{\bmx}\in\widetilde{\bmX}_{\Theta_{\bmx}}$.
    \item For every $(\bmell,\bmj)=(\bmell_{\bmx},\bmell_{\bmv},\bmj_{\bmx},\bmj_{\bmv})\in\Theta$, build the $(k+1)^{2d}$ scalars, indexed with $\bmi=(\bmi_{\bmx},\bmi_{\bmv})$, by
    \begin{equation}\label{eqn:hybrid_interpolation:analytic_integral}
    \begin{split}
    a_{\bmell,\bmj}^{\bmi} &= (M[\bmrho(\widetilde{\bmx}_{\bmell_{\bmx},\bmj_{\bmx}}^{\bmi_{\bmx}})](\cdot),g_{\bmell_{\bmv},\bmj_{\bmv}}^{\bmi_{\bmv}}(\cdot))_{\Omega_v} \\
    &= n(\widetilde{\bmx}_{\bmell_{\bmx},\bmj_{\bmx}}^{\bmi_{\bmx}})\prod_{m=1}^d\int_{-L_v}^{L_v} \frac{1}{\sqrt{2\pi\theta(\widetilde{\bmx}_{\bmell_{\bmx},\bmj_{\bmx}}^{\bmi_{\bmx}})}}\exp\bigg(\frac{-(v_m-u_m(\widetilde{\bmx}_{\bmell_{\bmx},\bmj_{\bmx}}^{\bmi_{\bmx}}))^2}{2\theta(\widetilde{\bmx}_{\bmell_{\bmx},\bmj_{\bmx}}^{\bmi_{\bmx}})}\bigg)g_{\ell_{v_m},j_{v_m}}^{i_{v_m}}(v_m)\,\dx{v_m}.
    \end{split}
    \end{equation}
    where $\widetilde{\bmx}_{\bmell_{\bmx},\bmj_{\bmx}}^{\bmi_{\bmx}}\in \widetilde{\bmX}_{\bmell_{\bmx},\bmj_{\bmx}}$.
    
    If $\bmell_{v_i}$ is 0, then we extend the respective velocity integral in \eqref{eqn:hybrid_interpolation:analytic_integral} to $\mathbb{R}$ and extend the Legendre polynomials $g_{0,0}^i\in L^2(-L_v,L_v)$ to a global polynomial on $\mathbb{R}$.
    This will guarantee the conservation invariants for the discrete collision operator, see \eqref{eqn:BGK_conservation}.
    \label[step]{defn:hybrid_interpolation:evaluation}

    \item Determine the coefficient expansion with respect to the basis $\{\psi_{\bmell_{\bmx},\bmj_{\bmx}}g_{\bmell_{\bmv},\bmj_{\bmv}}\}_{(\bmell,\bmj)\in\Theta}$.
    This is achieved by performing \Cref{defn:interpolation}~\Cref{defn:interpolation:to_interpolets} on only the position dimensions.
    Let $\beta_{\bmell,\bmj}$ denote the expansion with respect to $\{\psi_{\bmell_{\bmx},\bmj_{\bmx}}g_{\bmell_{\bmv},\bmj_{\bmv}}\}$
    \label[step]{defn:hybrid_interpolation:to_interpolets}

    \item Determine   the coefficient expansion with respect to the Alpert multiwavelet basis $g_{\bmell,\bmj}$.
    This is achieved by performing \Cref{defn:interpolation}~\Cref{defn:interpolation:to_alpert} on only the position dimensions.
    \label[step]{defn:hybrid_interpolation:to_alpert}
\end{enumerate}
\end{defn}

\begin{remark}
    The integrals in \eqref{eqn:hybrid_interpolation:analytic_integral} are exactly evaluated using error functions.
\end{remark}

The preservation of collision invariants \eqref{eqn:BGK_conservation} is now shown.
\begin{prop}\label{prop:discrete_maxwellian_conservation}
Let $\Theta$ be an ancestor complete grid, $k\geq 2$, and $f\in \bmV_\Theta$.
Suppose $M_\Theta[\bmrho_f]\in\bmV_\Theta$ is constructed with \Cref{defn:hybrid_interpolation},
then the following holds:
\begin{equation}\label{eqn:discrete_maxwellian_conservation:0}
    (M_\Theta[\bmrho_f],{\bfe})_{\Omega_v}=\bmrho_f.
\end{equation}
\end{prop}

\begin{proof}
    Fix $r=0,\ldots,d+1$, and denote $\bmrho_r^* = (M_\Theta[\bmrho_f],{\bfe}_r)_{\Omega_v}$. Decompose $\bmW_{\bf 0,0}=\bmW_{\bf 0,0}^{\bmx}\otimes\bmW_{\bf 0,0}^{\bmv}$.  Since $k\geq 2$, ${\bfe}\in\bmW_{\bf 0,0}^{\bmv}$.
    Thus ${\bfe}_r$ is a linear combination of $g_{\bf 0,0}^{\bmi_{\bmv}}$, namely
    \begin{equation} \label{eqn:discrete_maxwellian_conservation:1}
        {\bfe}_r(\bmv) = \sum_{\bmi_{\bmv}}\gamma_r^{\bmi_{\bmv}}g_{\bf 0,0}^{\bmi_{\bmv}}(\bmv).
    \end{equation}
    Thus the coefficients of $\bmrho_r^*\in \bmV_{\Theta_{\bmx}}$, denoted $\bar{c}_{\Theta_{\bmx}}$, with respect to the (orthonormal) Alpert wavelet basis for $(\bmell_{\bmx},\bmj_{\bmx})$ is
    \begin{equation} \label{eqn:discrete_maxwellian_conservation:2}
        \bar{c}_{\bmell_{\bmx},\bmj_{\bmx}}^{\bmi_{\bmx}} = (\bmrho_r^*,g_{\bmell_{\bmx},\bmj_{\bmx}}^{\bmi_{\bmx}})_{\Omega_x}=(M_\Theta[\bmrho_f],g_{\bmell_{\bmx},\bmj_{\bmx}}^{\bmi_{\bmx}}{\bfe}_r)_{\Omega}
        = \sum_{\bmi_{\bmv}}\gamma_r^{\bmi_{\bmv}} (M_\Theta[\bmrho_f],g_{\bmell_{\bmx},\bmj_{\bmx}}^{\bmi_{\bmx}} g_{\bf 0,0}^{\bmi_{\bmv}})_{\Omega}
        = \sum_{\bmi_{\bmv}}\gamma_r^{\bmi_{\bmv}}c_{\bmell_{\bmx},{\bf 0},\bmj_{\bmx},{\bf 0}}^{\bmi_{\bmx},\bmi_{\bmv}}.
    \end{equation}
    where $c_{\bmell,\bmj}$ denote the coefficients of $M_\Theta[\bmrho_f]$ in $g_{\bmell,\bmj}$.
    Since the hybrid interpolation only acts on the position dimensions, the map $\{\tilde{c}_{\ell_{\bmx},\bmj_{\bmx}}:(\ell_{\bmx},\bmj_{\bmx})\in\Theta_{\bmx}\}\to\{\bmrho_r^*(\widetilde{\bmx}):\widetilde{\bmx}\in \widetilde{\bmX}_{\Theta_{\bmx}}\}$ also maps $\{c_{\bmell,\bmj}:(\bmell,\bmj)\in\Theta\}\to \{a_{\bmell,\bmj}:(\bmell,\bmj)\in\Theta\}$, where $a_{\bmell,\bmj}$ is given in \eqref{eqn:hybrid_interpolation:analytic_integral}.
    Thus evaluating $\bmrho_r^*$ at $\widetilde{\bmx}_{\bmell_{\bmx},\bmj_{\bmx}}^{\bmi_{\bmx}}\in \widetilde{\bmX}_{\bmell_{\bmx},\bmj_{\bmx}}$ is given by
    \begin{equation}\label{eqn:discrete_maxwellian_conservation:3}
        \bmrho_r^*(\widetilde{\bmx}_{\bmell_{\bmx},\bmj_{\bmx}}^{\bmi_{\bmx}}) = \sum_{\bmi_{\bmv}}\gamma_r^{\bmi_{\bmv}}a_{\bmell_{\bmx},{\bf 0},\bmj_{\bmx},{\bf 0}}^{\bmi_{\bmx},\bmi_{\bmv}}.
    \end{equation}
    Since $\bmell_{\bmv}={\bf 0}$ for all the $a$ coefficients in  \eqref{eqn:discrete_maxwellian_conservation:2}, 
    we apply \eqref{eqn:hybrid_interpolation:analytic_integral} over $\bbR^d$ to recover 
    \begin{equation}
        \bmrho_r^*(\widetilde{\bmx}_{\bmell_{\bmx},\bmj_{\bmx}}^{\bmi_{\bmx}}) = \sum_{\bmi_{\bmv}}\gamma_r^{\bmi_{\bmv}}(M[\bmrho_f(\widetilde{\bmx}_{\bmell_{\bmx},\bmj_{\bmx}}^{\bmi_{\bmx}})],g_{\bf 0,0}^{\bmi_{\bmv}})_{\bbR^d} =(M[\bmrho_f(\widetilde{\bmx}_{\bmell_{\bmx},\bmj_{\bmx}}^{\bmi_{\bmx}})],{\bfe}_r)_{\bbR^d} = ({\bmrho_f})_r(\widetilde{\bmx}_{\bmell_{\bmx},\bmj_{\bmx}}^{\bmi_{\bmx}}).
    \end{equation}
    Thus $\bmrho^*(\widetilde{\bmx})=\bmrho_f(\widetilde{\bmx})$ for all $\widetilde{\bmx}\in\widetilde{\bmX}_{\Theta_{\bmx}}$ which implies $\bmrho^*=\bmrho_f$ in $[\bmV_{\Theta_{\bmx}}]^{d+2}$ by \Cref{prop:interp_equiv}.
\end{proof}

\subsection{Refinement and coarsening}
\label{subsec:refine_and_coarsen}

We now list the procedure used for refining and coarsening the adaptive grid $\Theta$ as it deviates from the procedure in \cite{schnake2024SparsegridDiscontinuous}.
One advantage of sparse-grid methods is that the coefficient representation in the wavelet basis decay over finer level, and the DG sparse-grid method leverages this fact to build PDE agnostic coarsening and refinement indicators in time.

To explain the indicators, let $\tau_\text{abs}\geq 0$ and $\tau_\text{rel}\geq 0$ denote the absolute and relative adaptive tolerances, and let $f\in \bmV_{\Theta}$ be a function with coefficient expansion $\{c_{\bmell,\bmj}\}$ in the Alpert basis.
An element $(\bmell,\bmj)\in\Theta$ is marked for refinement provided
\begin{align}\label{eqn:refine_indicator}
    \|c_{\bmell,\bmj}\|_2 \geq \tau_\text{rel}\|f\|_{L^2} + \tau_\text{abs},~~\text{where}~~\|f\|_{L^2}^2=\sum_{(\bmell,\bmj)\in\Theta}\|c_{\bmell,\bmj}\|_2^2.
\end{align}
If $(\bmell,\bmj)$ is marked for refinement, then all children of $(\bmell,\bmj)$ are added to $\Theta$.
An element $(\bmell,\bmj)\in\Theta$ is marked for coarsening, and thus removed from $\Theta$, provided
\begin{align}\label{eqn:coarsen_indicator}
    \|c_{\bmell',\bmj'}\|_2 \leq \tau_\text{rel}\|f\|_{L^2} + \tau_\text{abs}
\end{align}
for all $(\bmell',\bmj')=(\bmell,\bmj)$ and all parents of $(\bmell,\bmj)$.
Unlike the implementation in \cite{schnake2024SparsegridDiscontinuous}, refinement and coarsening happen at the same time, and elements are added (or not removed by coarsening) to ensure $\Theta$ is ancestor complete (see \Cref{defn:parents_and_children}).
We call a round of refinement and coarsening the grid per the above procedure as \textit{adapting} the grid.

Due to the lack of separability of the interpolation procedure, the grid to well approximate $f$ may not be the same to approximate $I_\Theta\mathcal{F}(f)$.
Thus, we also adapt $\Theta$ so that $I_\Theta\mathcal{F}(f)$ is well resolved. 
Let $b_{\bmell,\bmj}$ denote the coefficient expansion with respect to $\psi_{\bmell,\bmj}^{\bmi}$ of $I_\Theta\mathcal{F}(f)$, then we use a similar refinement/coarsening strategy as the Alpert case, but instead using the infinity norm, i.e.,
\begin{equation}\label{eqn:interp_indicator}
    \|b_{\bmell,\bmj}\|_{\infty} \geq(\leq) ~\tau_\text{rel}\max_{(\bmell',\bmj')\in\Theta}\|b_{\bmell',\bmj'}\|_\infty + \tau_\text{abs}.
\end{equation}
For hybrid interpolation, we do the same but with the coefficients $\beta_{\bmell,\bmj}$ as defined in \Cref{defn:hybrid_interpolation} \Cref{defn:hybrid_interpolation:to_interpolets}.
Whether using the standard or hybrid interpolation, we refer to refining $\Theta$ using $\mathcal{F}(f)$ as \emph{non-linear refinement using $\mathcal{F}(f)$.}

We now provide an overview of when the grid is adapted.
Let $N$ be the maximum level of the grid, the initial grid $\Theta$ is set as to produce the sparse-grid DG space $\hat{\bmV}_N$ in \eqref{eqn:sparse_grid_def}.
The grid $\Theta$ goes through several rounds of adapting with respect to $f_h^{0}$ and the non-linear refinement function $I_\Theta \mathcal{F}(f_h^{0})$ until adapting the grid does not add or remove any indices.  We use an analytic form of $f_h^0$ to compute coefficients of any new elements that may be added.

With a well-resolved initial condition, the grid adapts at the end of every IMEX timestep.  Given $f_h^{\mfn}\in \bmV_\Theta$, we advance to $f_h^{\mfn+1}\in \bmV_\Theta$ via \eqref{eqn:IMEX_RK}, and then $\Theta$ is adapted to resolve $f_h^{\mfn+1}$ and $I_\Theta \mathcal{F}(f_h^{\mfn+1})$ using the indicators in \eqref{eqn:refine_indicator}, \eqref{eqn:coarsen_indicator}, and \eqref{eqn:interp_indicator}.
If a new element $(\bmell,\bmj)$ is added to $\Theta$, then its associated coefficients $f_{\bmell,\bmj}^{\mfn+1}$ are set to 0.

\begin{remark}
    If $\|c_{(\bmell_{\bmx},{\bm 0},\bmj_{\bmx},{\bm 0})}\|_2$ is sufficiently small, then $(\bmell_{\bmx},{\bm 0},\bmj_{\bmx},{\bm 0})$ may be removed from $\Theta$ during coarsening, and the moments $\bmrho_f$ could change by a factor of the truncation tolerance within a timestep.
    However, $(\bmell_{\bmx},{\bm 0},\bmj_{\bmx},{\bm 0})$ is only removed provided its coefficients and the coefficients of its parents are below the adaptive threshold.
    This fact means that removed elements usually have size much smaller than the adaptive threshold, and, as verified in \Cref{sec:numerical_experiments}, global conservation quantities are often preserved well below the truncation tolerance.
\end{remark}

We summarize the adaptive sparse-grid workflow from timestep $t^{\mfn}$ to $t^{\mfn+1}$ in \Cref{alg:DG-IMEX}.

\algrenewcommand\algorithmicrequire{\textbf{Input:}}
\algrenewcommand\algorithmicensure{\textbf{Output:}}
\begin{algorithm}[t]
\caption{Adaptive sparse-grid DG--IMEX solver for the BGK model from $t^{\mfn}$ to $t^{\mfn+1}$}
\label{alg:adaptive-bgk}
\begin{algorithmic}[1]
\Require $\Theta^{\mfn}$ and $f_h^{\mfn}\in \bmV_{\Theta^{\mfn}}$.
\Ensure $\Theta^{\mfn+1}$ and $f_h^{\mfn+1}\in \bmV_{\Theta^{\mfn+1}}$.

\State Set $\Theta\gets \Theta^{\mfn}$
\State Calculate $f_h^{(1,*)}\in \bmV_{\Theta}$ from \eqref{eqn:IMEX_RK:ex1}, i.e., solve
\begin{equation}
(f_h^{(1,*)},g_h) = (f_h^{\mfn},g_h) - \dt \AP(f_h^{\mfn},g_h)\quad\forall g_h\in \bmV_\Theta.
\end{equation}
\State Calculate the moments $\bmrho_{f_h^{(1,*)}}\in [\bmV_{\Theta_{\bmx}}]^{d+2}$ and apply the hybrid interpolation map \Cref{defn:hybrid_interpolation} to produce $M_{\Theta}[\bmrho_{f_h^{(1,*)}}]\in \bmV_{\Theta}$.
\State Calculate $f_h^{(1)}\in \bmV_{\Theta}$ from \eqref{eqn:IMEX_RK:im1} via
\begin{equation}
f_h^{(1)} = \frac{1}{1+\nu\dt}f_h^{(1,*)} + \frac{\nu\dt}{1+\nu\dt} M_{\Theta}[\bmrho_{f_h^{(1,*)}}].
\end{equation}
\State Calculate $f_h^{(2,*)}\in \bmV_{\Theta}$ from \eqref{eqn:IMEX_RK:ex2}, i.e., solve
\begin{equation}
(f_h^{(2,*)},g_h) = \tfrac{1}{2}(f_h^{\mathfrak{n}},g_h) + \tfrac{1}{2}\big((f_h^{(1)},g_h)-\Delta{t}\AP(f_h^{(1)},g_h)\big)\quad\forall g_h\in \bmV_\Theta.
\end{equation}
\State Calculate the moments $\bmrho_{f_h^{(2,*)}}\in [\bmV_{\Theta_{\bmx}}]^{d+2}$ and apply the hybrid interpolation map \Cref{defn:hybrid_interpolation} to produce $M_{\Theta}[\bmrho_{f_h^{(2,*)}}]\in \bmV_{\Theta}$.
\State Calculate $f_h^{(2)}\in \bmV_{\Theta}$ from \eqref{eqn:IMEX_RK:im2} via 
\begin{equation}
f_h^{(2)} = \frac{1}{1+\frac12\nu\dt}f_h^{(2,*)} + \frac{\frac12\nu\dt}{1+\frac12\nu\dt} M_{\Theta}[\bmrho_{f_h^{(2,*)}}].
\end{equation}
\State Adapt $\Theta$ based on $f_h^{(2)}$ and the indicators in  \eqref{eqn:refine_indicator} and \eqref{eqn:coarsen_indicator} such that $\Theta$ remains ancestor complete (see \Cref{defn:parents_and_children}).
\State Refine $\Theta$ using the non-linear refinement function $\mathcal{F}(f)$ (usually proportional to $M_\Theta[\bmrho_{f_h^{(2)}}]$) and the indicators in \eqref{eqn:interp_indicator} such that $\Theta$ remains ancestor complete.
\State Set $f_h^{\mfn+1}\gets f_h^{(2)}$ and $\Theta^{\mfn+1}\gets \Theta$.
\end{algorithmic}
\label{alg:DG-IMEX}
\end{algorithm}

\section{Numerical experiments}
\label{sec:numerical_experiments}

The following $2x2v$ and $3x3v$ numerical experiments were computed on a workstation with a 16 core Intel Xeon w5-3433 and 256GB of RAM. 
The simulations were all computed using the C++ Adaptive Sparse-Grid Discretization (ASGarD) library \cite{hahn2024ASGarDAdaptive} compiled using \texttt{gcc11.4.0}.
The timings presented are the wall clock time minus the input/output times to write snapshots of the simulation to the disk.

Below we list general comments that are used for all simulations
\begin{itemize}
\item The maximum explicit timestep for the DG discretization of the operator $\bmv\cdot\grad_{\bmx}$ is 
\begin{equation}\label{eqn:maximum_explicit_timestep}
    \dt_\text{expl} = h_N\frac{1}{d}\frac{1}{L_v}\frac{1}{2k+1},
\end{equation}
where $h_N$ is the smallest mesh length over position dimensions, defined from the maximum level.

\item We report the velocity DoFs per position DoF, averaged per dimension, by
\begin{equation}
    \dofvavg := \bigg(\frac{\text{phase-space DoF}}{\text{position DoF}}\bigg)^{\frac1d} = (k+1)\bigg(\frac{|\Theta|}{|\Theta_{\bmx}|}\bigg)^{\frac1d}
\end{equation}

\item We often switch between using $\bmx=(x_1,x_2,x_3)$, $\bmv=(v_1,v_2,v_3)$ and $\bmx=(x,y,z)$, $\bmv=(v_x,v_y,v_z)$.

\item Derivatives used for plotting were computed using a 5 point central difference formula rather than directly differentiating the discrete solution.

\item In fluid regimes, Euler reference solutions are computed using the Runge--Kutta DG code described and verified in \cite{endeve_etal_2019,pochik_etal_2021} with quadratic elements and third-order time integration.
\end{itemize}

\subsection{\texorpdfstring{$\bm{3x3v}$}{3x3v} relaxation test}  
\label{subsec:relaxation_test}
We first test the discrete BGK collision operator via a $3x3v$ relaxation problem
\begin{equation}\label{eqn:relaxation_problem}
    \partial_t f = \nu(M[\bmrho_f]-f),
\end{equation}
with $\Omega_x = [-1,1]^3$ and $\Omega_v=[-6,6]^3$.
We implement a Backward Euler time discretization of \eqref{eqn:relaxation_problem}; namely,
\begin{equation}\label{eqn:discrete_backward_euler}
    (f_h^{\mfn+1},g_h) + \nu\dt (f_h^{\mfn+1},g_h)  = (f_h^{\mfn},g_h) + \nu\Delta t(M_\Theta[\bm{\rho}_{f_h^{\mfn}}],g_h) \quad\forall g_h\in \bmV_\Theta.
\end{equation}
We use the following separable initial condition $f(\bmx,\bmv,0) = X(x)Y(y)Z(z)V_x(v_x)V_y(v_y)V_z(v_z)$ where
\begin{subequations}
\begin{alignat}{3}
    X(x) &= 1 + 0.2\sin(3x), &\qquad Y(y) &= 1 + 0.3\cos(5y), &\qquad Z(z) &= 1+ 0.25\sin(4.5z),
\end{alignat}
\begin{align}
    V_x(v_x) &= \tfrac12 \Big(
        \exp\Big(\tfrac{-(v_x-2.0)^2}{2\cdot0.6^2}\Big) + 
        \exp\Big(\tfrac{-(v_x+1.0)^2}{2\cdot1.4^2}\Big)
    \Big),  \\
    V_y(v_y) &= \tfrac12 \Big(
        \exp\Big(\tfrac{-(v_y-1.5)^2}{2\cdot0.7^2}\Big) + 
        \exp\Big(\tfrac{-(v_y+2.0)^2}{2\cdot1.2^2}\Big)
    \Big),  \\
    V_z(v_z) &= \tfrac12 \Big(
        \exp\Big(\tfrac{-(v_z-0.5)^2}{2\cdot0.5^2}\Big) + 
        \exp\Big(\tfrac{-(v_z+0.5)^2}{2\cdot1.0^2}\Big)
    \Big).
\end{align}
\end{subequations}

\begin{table}[htbp]
    \centering
    \caption{\Cref{subsec:relaxation_test} -- $3x3v$ Relaxation Problem.  Relative $L^2$ error of the moments $\bmrho_{f_h^{1}}$ after collisions compared against the moments of adapted initial condition $\bmrho_{f_h^{0}}$ before collisions for adaptive interpolation $I_\Theta M[\bmrho]$ (\Cref{defn:interpolation}) and hybrid interpolation $M_\Theta[\bmrho]$ (\Cref{defn:hybrid_interpolation}).  We omit the $\rho_2$ and $\rho_3$ errors are they are similar to $\rho_1$.
    The DoFs of the level 6 full grid $\bmV_N$ for the given $k$ are labeled ``FG DoFs''.}

    \renewcommand{\arraystretch}{1.1}
    \resizebox{\textwidth}{!}{%
    \begin{tabular}{c|c|c|c|c|c|c|c}
        \multicolumn{2}{c|}{$k$} & \multicolumn{2}{c|}{$1$} & \multicolumn{2}{c|}{$2$} & \multicolumn{2}{c}{$3$} \\ \hline
        \multicolumn{2}{c|}{$\tau_\text{rel}$} & $10^{-3}$ & $10^{-5}$ & $10^{-3}$ & $10^{-5}$ & $10^{-3}$ & $10^{-5}$ \\ 
        \multicolumn{2}{c|}{$|\Theta|$} & \num{17004} & 
                 \num{481226} &
                 \num{4414} &
                 \num{66764} &
                 \num{1715} & 
                 \num{17051} \\
        \multicolumn{2}{c|}{DoFs} & \num[exponent-mode = scientific,round-mode= places,round-precision= 2]{1088256} & 
                 \num[exponent-mode = scientific,round-mode= places,round-precision= 2]{30798464} &
                 \num[exponent-mode = scientific,round-mode= places,round-precision= 2]{3217806} &
                 \num[exponent-mode = scientific,round-mode= places,round-precision= 2]{48670956} &
                 \num[exponent-mode = scientific,round-mode= places,round-precision= 2]{7024640} & 
                 \num[exponent-mode = scientific,round-mode= places,round-precision= 2]{69840896} \\
        \multicolumn{2}{c|}{FG DoFs} & \num[exponent-mode = scientific,round-mode= places,round-precision= 2]{4.3980e12} & 
                 \num[exponent-mode = scientific,round-mode= places,round-precision= 2]{4.3980e12} &
                 \num[exponent-mode = scientific,round-mode= places,round-precision= 2]{5.0096e13} &
                 \num[exponent-mode = scientific,round-mode= places,round-precision= 2]{5.0096e13} &
                 \num[exponent-mode = scientific,round-mode= places,round-precision= 2]{2.8147e14} & 
                 \num[exponent-mode = scientific,round-mode= places,round-precision= 2]{2.8147e14} \\\hline\hline
        \multirow{4}{*}{\rotatebox[origin=c]{90}{~~~~$M_\Theta[\bmrho]$}}
        & $\rho_0$ & \num{6.02e-15} & \num{7.06e-15} & \num{4.25e-15} & \num{5.79e-15} & \num{4.03e-15} & \num{3.91e-15} \\
        & $\rho_1$ & \num{6.84e-15} & \num{8.42e-15} & \num{8.35e-16} & \num{2.52e-15} & \num{6.02e-16} & \num{1.16e-15} \\
        & $\rho_4$ & \num{7.69e-03} & \num{4.64e-03} & \num{5.43e-15} & \num{1.03e-14} & \num{2.52e-15} & \num{3.08e-15} \\ \hline

        \multirow{4}{*}{\rotatebox[origin=c]{90}{~~~~$I_\Theta M[\bmrho]$}}
        & $\rho_0$ & \num{1.06e-02} & \num{1.65e-03} & \num{3.01e-02} & \num{1.55e-03} & \num{1.56e-03} & \num{1.38e-03} \\
        & $\rho_1$ & \num{6.46e-02} & \num{7.09e-03} & \num{5.60e-02} & \num{5.82e-03} & \num{1.01e-02} & \num{5.64e-03} \\
        & $\rho_4$ & \num{3.57e-02} & \num{8.55e-03} & \num{8.81e-02} & \num{7.72e-03} & \num{8.04e-03} & \num{7.44e-03} \\ \hline
    \end{tabular}
    }
    \label{tab:conservation_errors}
\end{table}

Even though there is no advection, we choose a spatially varying initial condition to thoroughly test \Cref{prop:discrete_maxwellian_conservation} with a non-trivial adaptive sparse grid space. 
We set a maximum level of 6 in all dimensions, $\nu=10^7$, $\dt=1.0$, and $\tau_\text{abs}=0$.
We consider linear $k=1$, quadratic $k=2$, and cubic $k=3$ polynomials.  We adapt the initial condition for a specified $\tau_\text{rel}$, run \eqref{eqn:discrete_backward_euler} for one timestep, and do not adapt the grid after the timestep.
Therefore, according to \Cref{prop:discrete_maxwellian_conservation}, $\bmrho_{f_h^{0}}=\bmrho_{f_h^{1}}$ for $k\geq 2$.
To verify this claim, we list the relative error of $\bmrho_{f_h^{1}}$ against $\bmrho_{f_h^{0}}$ in \Cref{tab:conservation_errors} for $\tau_\text{rel}=10^{-3}$ and $\tau_\text{rel}=10^{-5}$.
The table shows that density $\bmrho_0$ and momentum (only $\bmrho_1$ shown) are well conserved for all polynomial degrees, while energy $\bmrho_4$ is only conserved for $k\geq 2$, as expected.
We additionally run the same test by using the interpolated Maxwellian $I_\Theta M[\bmrho_f]$ in phase-space (see \Cref{defn:interpolation}).  The errors are reported in \Cref{tab:conservation_errors} and show that phase-space adaptive interpolation does not preserve the moments to a satisfactory tolerance on a grid that sufficiently resolves $f_h^0$.

\subsection{\texorpdfstring{$\bm{2x2v}$}{2x2v} Sod shock tube}
\label{subsec:sod}

Here we consider a radial variant of the classical Sod Shock tube problem \cite{sod_1978}, see \cite[Section 17.1]{toro2009RiemannSolvers}, that features an initial contact discontinuity in position space.

Let $\Omega_{\bmx}=[-1,1]^2$ and $\Omega_v=[-6,6]^2$.  
Define the radial function $\zeta(\alpha,\beta;r)$ that provides a $C^1$ transition from $\alpha$ to $\beta$ by
\begin{equation}\label{eqn:radial_ic_func}
    \zeta(\alpha,\beta;r) = 
    \begin{cases}
        \alpha &\text{ if } r < 0.39,\\
        \alpha + \frac{-6(\beta-\alpha)}{(\Delta r)^3}\big(\frac{1}{3}(r-0.39)^3 - \frac{\Delta r}{2}(r-0.39)^2\big) &\text{ if }0.39\leq r \leq 0.41,\\
        \beta &\text{ if } r > 0.41,
    \end{cases}
\end{equation}
where $\Delta r = 0.41-0.39=0.02$.
The initial condition is a Maxwellian generated by the fluid variables 
\begin{equation}
\label{eqn:fluid_vars_sod}
    n(\bmx) = \zeta(1.0,0.125;|\bmx|),\qquad\bmu(\bmx) = {\bf 0},\qquad \theta(\bmx)=\zeta(1.0,0.8;|\bmx|).
\end{equation}
We use a smoothed contact discontinuity so that no additional preparation of the initial condition is needed.
Addtionally, we specify the inflow boundary condition as a Maxwellian with $n=0.125$, $\bmu=0$, and $\theta=0.8$.

In the results below we plot the kinetic velocity distribution at points in position space.  
An asymptotic Chapman-Enskog expansion of $f$ in \eqref{eqn:bgk} yields the following \cite{bardos1991FluidDynamic}: $f = M[\bmrho_f] + g[\bmrho_f] + \mathcal{O}(\nu^{-2})$ where
\begin{align}\label{eqn:chapman_enskog}
    g[\bmrho](\bmv) =  - \frac{1}{\nu \theta}\bigg(\bmc^\top\grad\bmu\bmc-\frac1d|\bmc|^2\grad\cdot\bmu + \Big(\frac{|\bmc|^2}{2\theta}-\frac{d+2}{2}\Big)\bmc\cdot\grad\theta \bigg)M[\bmrho](\bmv)
\end{align}
and $\bmc=\bmv-\bmu$ is the relative velocity.
We will refer to $g[\bmrho_f]$ as the compressible Navier--Stokes perturbation since the resulting equations can be derived by setting $f=M[\bmrho_f]+g[\bmrho_f]$ in \eqref{eqn:bgk} and testing by $\bfe$.

\paragraph{Discretization Specifics} We set $k=2$ and the maximum level to $8$ for all dimensions.  The full-grid DoF is then $(2^8\cdot3)^4\approx \num{3.48e11}$. 
We set $\dt=10^{-4}$ which is 76.8\% of the maximum explicit timestep in \eqref{eqn:maximum_explicit_timestep} and run for 1500 timesteps.
We set $\tau_\text{abs} = 10^{-5}$ and $\tau_\text{rel} = 0$.

\paragraph{Results}

Taking the Knudsen number $\Kn=\nu^{-1}$, we report the radial fluid variables, position grid adaptivity, and DoFs in time for the following regimes \cite{antman2005MicroflowsNanoflows}:
\begin{enumerate}
    \item The fluid regime $\nu=1000$ $(\Kn=0.001)$ where continuum models such as Navier-Stokes well approximate the BGK model.
    \item The boundary of the fluid/transition regime $\nu=10$ $(\Kn=0.1)$ where the Navier--Stokes equations start to break down.
    \item The transitional regime $\nu=\Kn=1$, where both free-streaming and continuum models are unreliable \cite{cercignani1988BoltzmannEquation,sharipov2016RarefiedGas}.
\end{enumerate}
\begin{figure}[htbp]
    \centering
    \begin{subfigure}{0.245\textwidth}
    \centering
        \includegraphics[width=\textwidth]{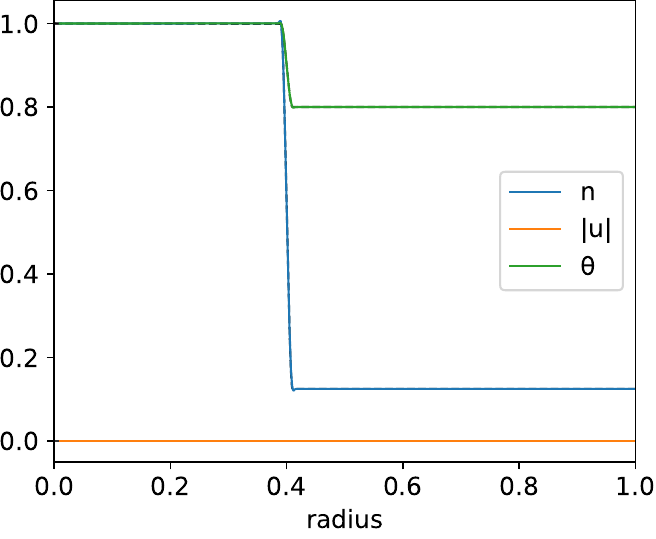}
        \caption{$t=0.0$}
    \end{subfigure}%
    \begin{subfigure}{0.245\textwidth}
    \centering
        \includegraphics[width=\textwidth]{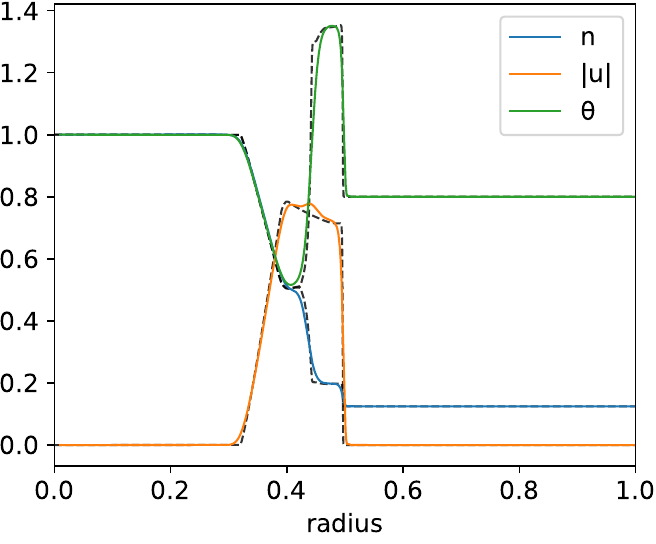}
        \caption{$t=0.05$}
    \end{subfigure}%
    \begin{subfigure}{0.245\textwidth}
    \centering
        \includegraphics[width=\textwidth]{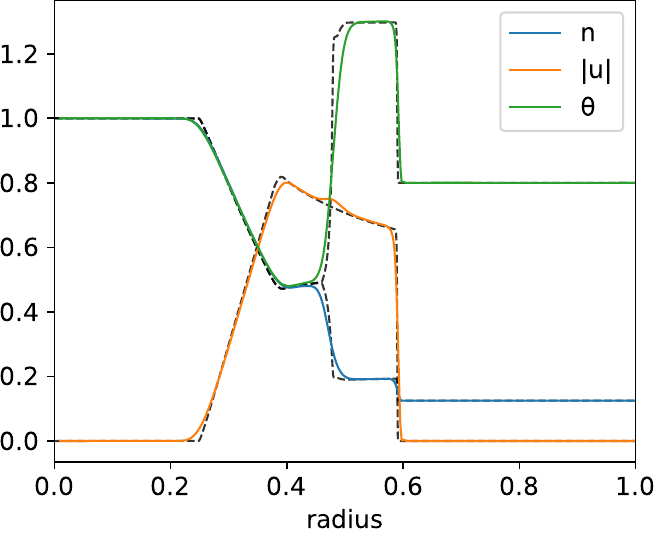}
        \caption{$t=0.10$}
    \end{subfigure}%
    \begin{subfigure}{0.245\textwidth}
    \centering
        \includegraphics[width=\textwidth]{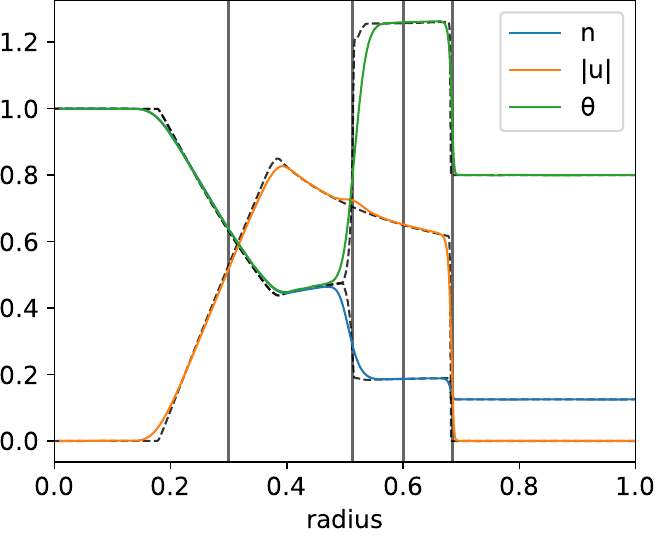}
        \caption{$t=0.15$}
        \label{fig:sod_shock:nu_1000:fluid_vars:comp:t_0p15}
    \end{subfigure}%
    \caption{\Cref{subsec:sod:nu_1000} -- $2x2v$ Sod shock tube, $\nu=1000$.  Line plots of fluid variables along the ray $(\cos(\pi/5),\sin(\pi/5))$.  The black dashed lines are from a high-resolution DG solver for the compressible Euler equations.  The vertical lines in \Cref{fig:sod_shock:nu_1000:fluid_vars:comp:t_0p15} denote the radii where the velocity plots in \Cref{fig:sod_shock:nu_1000:velocity} are taken.}
    \label{fig:sod_shock:nu_1000:fluid_vars:comp}
\end{figure}

For lower $\nu$, the initial contact discontinuity in position space advects into phase-space, inducing non-equilibrium behavior and requiring sufficient velocity resolution to capture large gradients.
Rather than comparing the fluid variables of the BGK model and continuum equations like Navier--Stokes \cite{bird2023MolecularGas,xiong2016HighOrder}, we instead compare velocity distributions created by the adaptive sparse-grid BGK method and the analytic Chapman--Enskog expansion $f=M[\bmrho_f]+g[\bmrho_f]$ where $\bmrho_f$ is derived from the simulation.

\subsubsection{Fluid regime: \texorpdfstring{$\nu=1000$}{ν=1000}}
\label{subsec:sod:nu_1000}

We first take $\nu=1000$ which puts us close to the compressible Euler limit where the evolution of the fluid variables are well governed by the compressible Navier--Stokes equations.
The nonlinear refinement function is $0.1M_\Theta[\bmrho_f]$ (see \Cref{subsec:refine_and_coarsen}).

\Cref{fig:sod_shock:nu_1000:fluid_vars:comp} plots the fluid variables along a radial cut against a high-resolution compressible Euler solver with the same initial condition and shows that the kinetic fluid variables are smoother, due to the BGK collision operator, but agree well with the Euler reference.

\begin{figure}[htbp]
    \centering
    \begin{subfigure}{0.245\textwidth}
        \includegraphics[width=\textwidth]{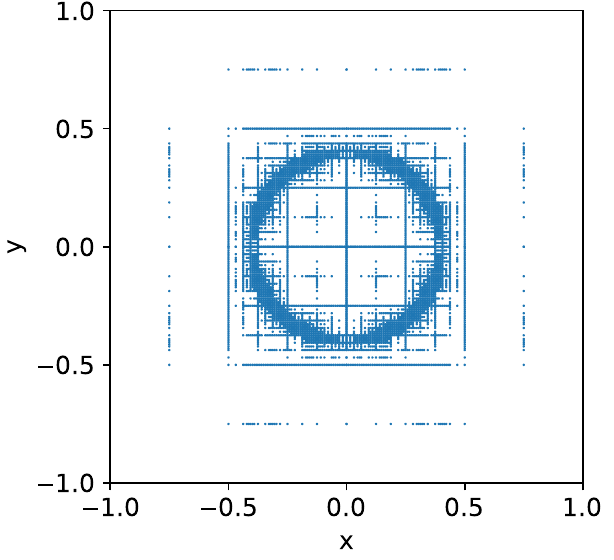}
        \caption{$t=0.00$}
        \label{fig:sod_shock:nu_1000:fluid_vars:grid_0.00}
    \end{subfigure}%
    \begin{subfigure}{0.245\textwidth}
        \includegraphics[width=\textwidth]{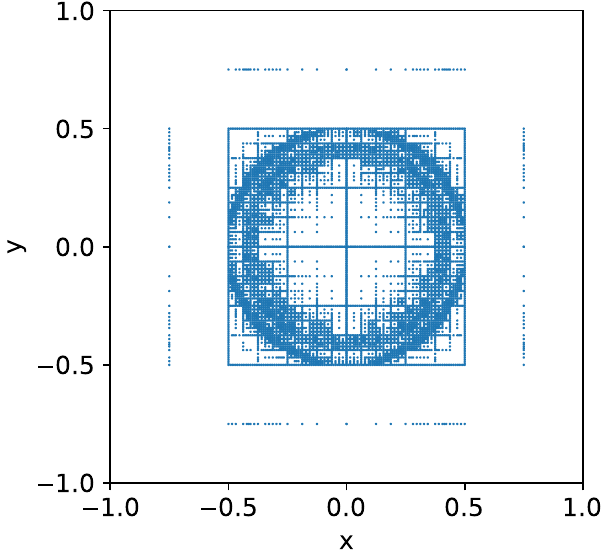}
        \caption{$t=0.05$}
        \label{fig:sod_shock:nu_1000:fluid_vars:grid_0.05}
    \end{subfigure}%
    \begin{subfigure}{0.245\textwidth}
        \includegraphics[width=\textwidth]{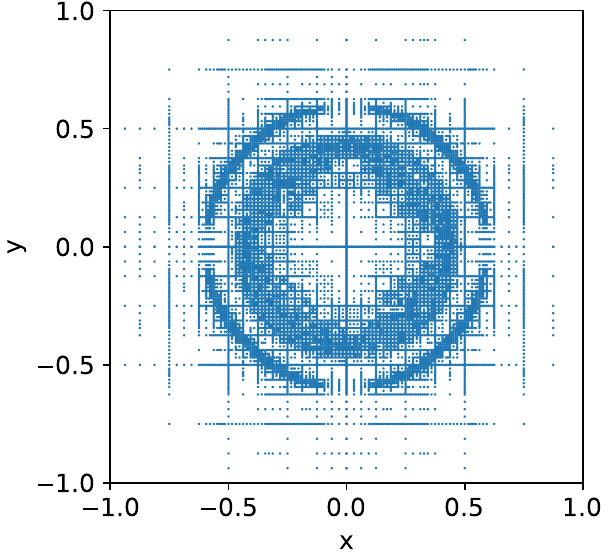}
        \caption{$t=0.10$}
        \label{fig:sod_shock:nu_1000:fluid_vars:grid_0.10}
    \end{subfigure}%
    \begin{subfigure}{0.245\textwidth}
        \includegraphics[width=\textwidth]{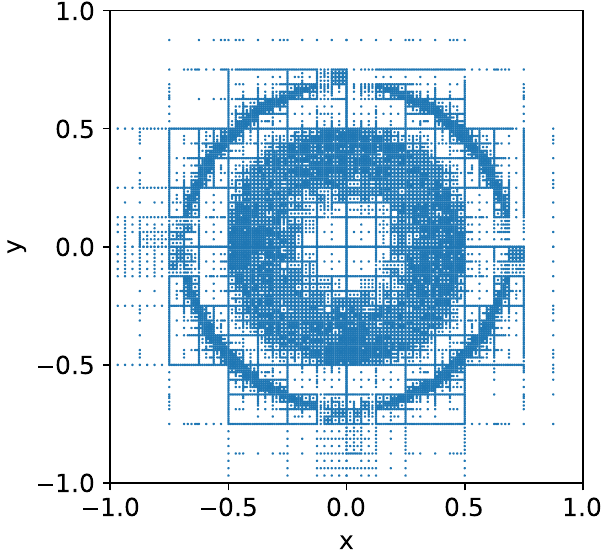}
        \caption{$t=0.15$}
        \label{fig:sod_shock:nu_1000:fluid_vars:grid_0.15}
    \end{subfigure}%
    \caption{\Cref{subsec:sod:nu_1000} -- $2x2v$ Sod shock tube, $\nu=1000$. Position grid $\Theta_{\bmx}$ (see \eqref{eqn:position_grid}) denoted by the barycenters of the support of each active element for multiple times. }
    \label{fig:sod_shock:nu_1000:fluid_vars}
\end{figure}

\Cref{fig:sod_shock:nu_1000:fluid_vars} plots the adaptive position grid $\Theta_{\bmx}$ at $t=0.0, 0.05, 0.10, 0.15$.
The shock front is well captured by the position grid refinement.
The region between the shock front and contact discontinuity is sparse while the grid is dense starting at the contact discontinuity and progressing throughout the rarefaction wave.
For $t=0.15$, the position grid is not radially symmetric due to the lack of symmetry in the interpolation points; however, we observe good radial symmetry of $n$, $|\bmu|$, and $\theta$.

\begin{figure}[htbp]
    \centering
    \begin{subfigure}{0.245\textwidth}
        \includegraphics[width=\textwidth]{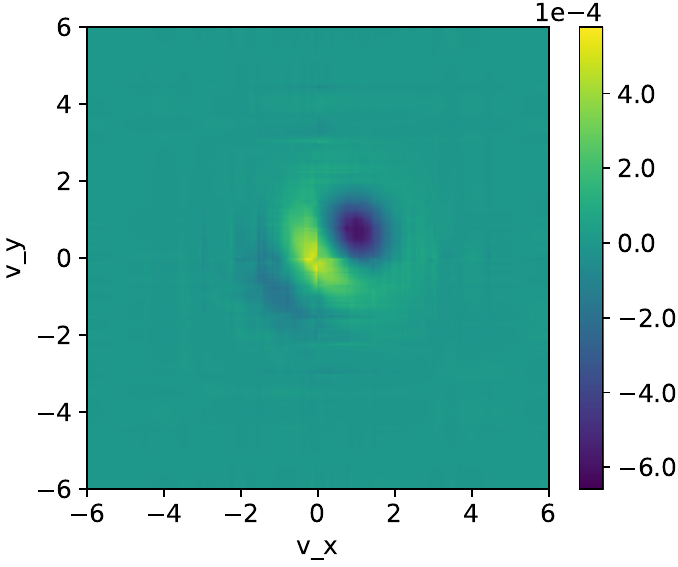}
        \captionsetup{justification=centering}
        \caption{K.P. $r=0.3$}
    \label{fig:sod_shock:nu_1000:velocity:Kinetic_0.3}
    \end{subfigure}%
    \begin{subfigure}{0.245\textwidth}
        \includegraphics[width=\textwidth]{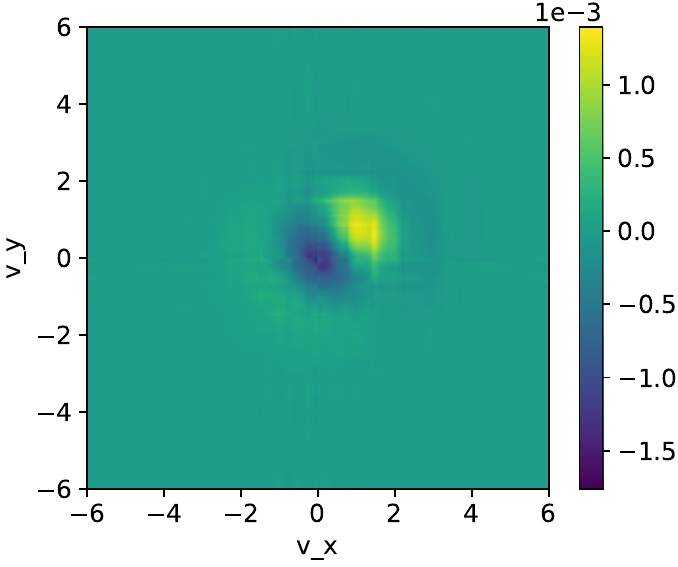}
        \captionsetup{justification=centering}
        \caption{K.P. $r=0.5125$}
    \label{fig:sod_shock:nu_1000:velocity:Kinetic_0.5125}
    \end{subfigure}%
    \begin{subfigure}{0.245\textwidth}
        \includegraphics[width=\textwidth]{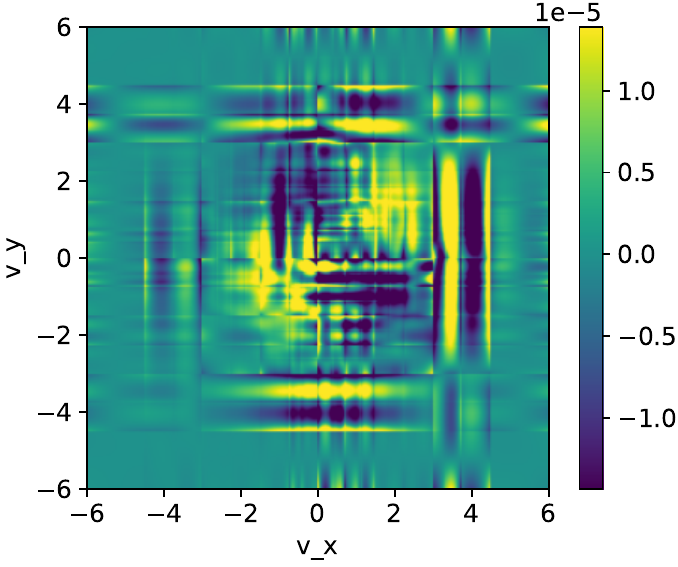}
        \captionsetup{justification=centering}
        \caption{K.P. $r=0.6$}
    \label{fig:sod_shock:nu_1000:velocity:Kinetic_0.6}
    \end{subfigure}%
    \begin{subfigure}{0.245\textwidth}
        \includegraphics[width=\textwidth]{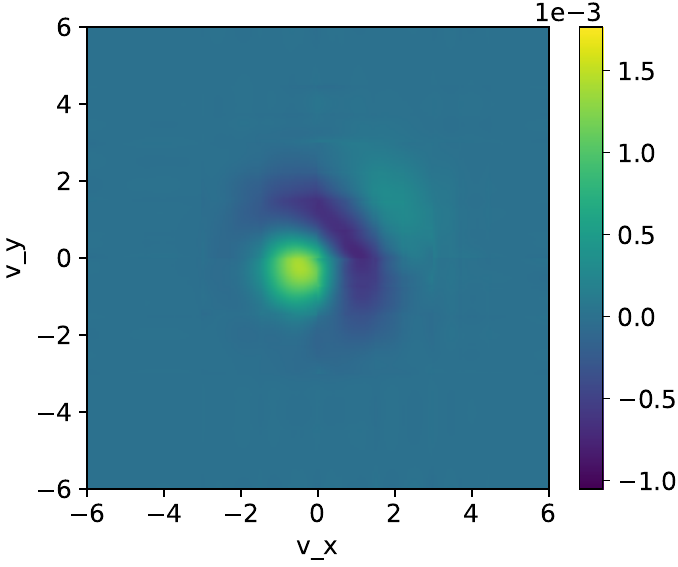}
        \captionsetup{justification=centering}
        \caption{K.P. $r=0.685$}
    \label{fig:sod_shock:nu_1000:velocity:Kinetic_0.685}
    \end{subfigure}%

    \vspace{2ex}

    \begin{subfigure}{0.245\textwidth}
        \includegraphics[width=\textwidth]{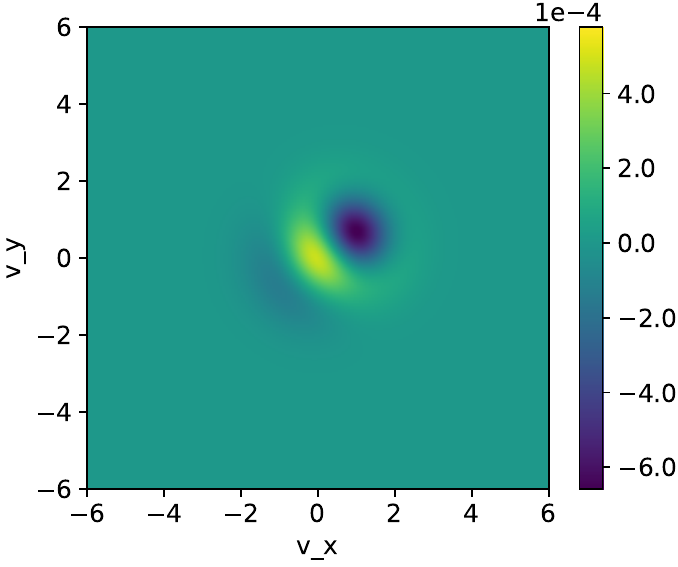}
        \caption{$g$. $r=0.3$}
    \label{fig:sod_shock:nu_1000:velocity:CNS_0.3}
    \end{subfigure}%
    \begin{subfigure}{0.245\textwidth}
        \includegraphics[width=\textwidth]{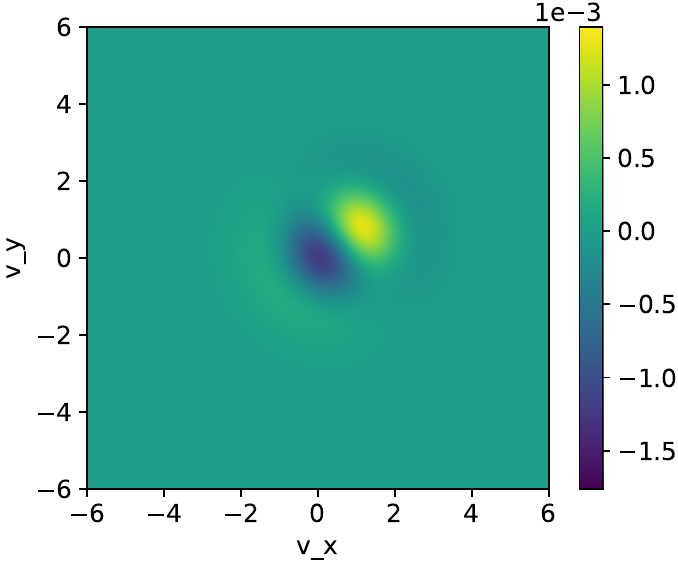}
        \caption{$g$. $r=0.5125$}
    \label{fig:sod_shock:nu_1000:velocity:CNS_0.5125}
    \end{subfigure}%
    \begin{subfigure}{0.245\textwidth}
        \includegraphics[width=\textwidth]{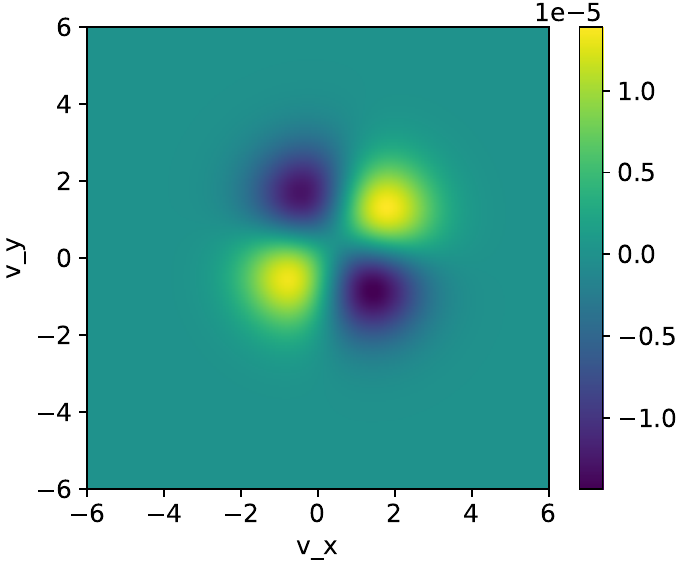}
        \caption{$g$. $r=0.6$}
    \label{fig:sod_shock:nu_1000:velocity:CNS_0.6}
    \end{subfigure}%
    \begin{subfigure}{0.245\textwidth}
        \includegraphics[width=\textwidth]{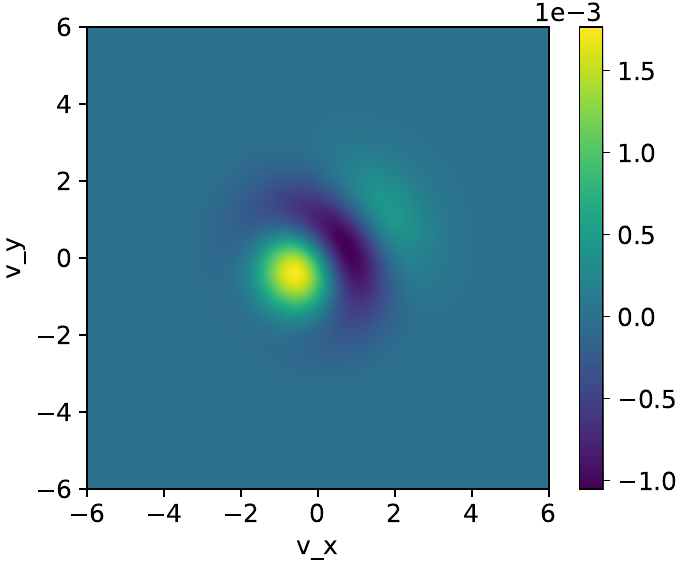}
        \caption{$g$. $r=0.685$}
    \label{fig:sod_shock:nu_1000:velocity:CNS_0.685}
    \end{subfigure}%

    \caption{\Cref{subsec:sod:nu_1000} -- $2x2v$ Sod shock tube, $\nu=1000$. Deviations from Maxwellian at $\bmx=(r\cos(\tfrac{\pi}{5}),r\sin(\tfrac{\pi}{5}))$ for various radii $r$ at $t=0.15$. 
    Here K.P.~is the kinetic perturbation $f(\bmx,\cdot)-M[\bmrho_f(\bmx)](\cdot)$. The perturbation $g[\bmrho_f]$ is given in \eqref{eqn:chapman_enskog}.}
    \label{fig:sod_shock:nu_1000:velocity}
\end{figure}

\Cref{fig:sod_shock:nu_1000:velocity} plots the deviation from the Maxwellian along a radial cut of the position space at an angle $\pi/5$ from the positive $x$-axis.
We set $t=0.15$ and consider the radius $r=0.3$ along the rarefaction wave, $r=0.5125$ at the contact discontinuity, $r=0.6$ in the flat region between the contact and shock wave where the grid is coarse, and $r=0.685$ at the shock wave; see \Cref{fig:sod_shock:nu_1000:fluid_vars:comp:t_0p15}.
\Cref{fig:sod_shock:nu_1000:velocity:Kinetic_0.3,fig:sod_shock:nu_1000:velocity:Kinetic_0.5125,fig:sod_shock:nu_1000:velocity:Kinetic_0.6,fig:sod_shock:nu_1000:velocity:Kinetic_0.685} plot the perturbation $f-M[\bmrho_f]$ where $M$ is calculated analytically using the moments $\bmrho_f$ given from the discrete distribution.
The plots show that $f-M[\bmrho_f]$ is at most $\mathcal{O}(\nu^{-1})$ which is constistent with the Chapman--Enskog expansion.
\Cref{fig:sod_shock:nu_1000:velocity:CNS_0.3,fig:sod_shock:nu_1000:velocity:CNS_0.5125,fig:sod_shock:nu_1000:velocity:CNS_0.6,fig:sod_shock:nu_1000:velocity:CNS_0.685} plot the Navier-Stokes perturbation $g$ in \eqref{eqn:chapman_enskog} using $\bmrho_f$.
For all radii but $r=0.6$, there is good agreement between $g$ and the discrete kinetic perturbation; additionally, the discrete kinetic perturbation only shows mild numerical artifacts.  For $r=0.6$, the size of $g$ matches the adaptive tolerance $\tau_\text{abs}=10^{-5}$, and thus the discrete kinetic perturbation is dominated by noise.

\begin{figure}[htbp]
    \centering
    \begin{subfigure}{0.35\textwidth}
        \centering
        \includegraphics[width=\textwidth]{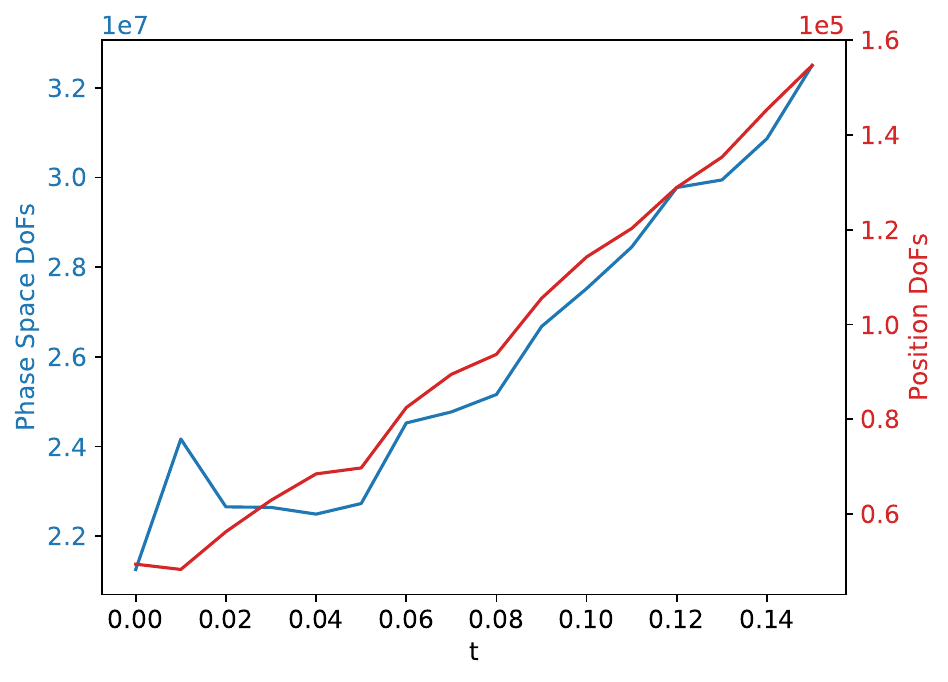}
        \caption{Degrees of freedom over time for the 4D phase space grid and the 2D moment grid.\\}
        \label{fig:sod_shock:nu_1000:dof}
    \end{subfigure}
    \hspace{4em}
    \begin{subfigure}{0.35\textwidth}
        \centering
        \includegraphics[width=\textwidth]{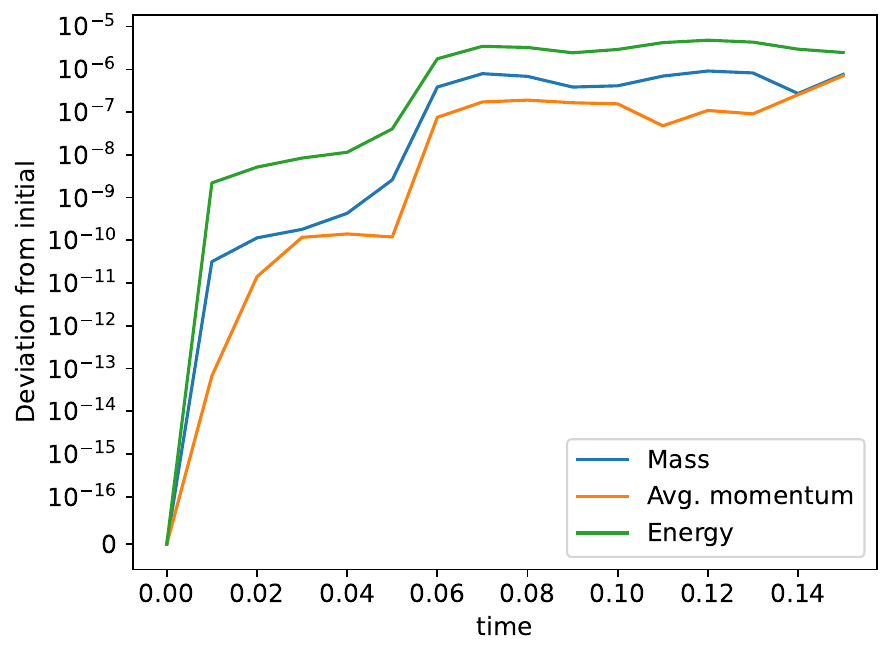}
        \caption{Absolute deviation of the total mass, average momentum, and energy with respect to the initial values.}
        \label{fig:sod_shock:nu_1000:cons}
    \end{subfigure}
    \caption{\Cref{subsec:sod:nu_1000} -- $2x2v$ Sod shock tube, $\nu=1000$. Degrees of freedom and conservation quantities as a function of time.
    A level 8,8,6,6 full grid (corresponding to the finest level obtained in the simulation) has approximately \num{2.17e10} DoFs.}
    \label{fig:sod_shock:nu_1000:dof_and_cons}
\end{figure}

\Cref{fig:sod_shock:nu_1000:dof} plots the DoFs of the phase-space distribution and the moments $\bmrho_f$.  
Since $f\approx M[\bmrho_f]$, we expect that the position and phase-space DOFs to be correlated.
For short times, the initial layer causes kinetic effects which requires more velocity resolution per position degree of freedom, but for longer times the phase space and position DoFs are correlated.
At $t=0.15$ the phase space and position DoFs are approximately 32.5 million and 145 thousand respectively. This gives an average per velocity dimension DoF of $\dofvavg\approx 14.5$.
The simulation never refined beyond level 6 in velocity; thus the 32.5 million DoFs correspond to approximately 0.15\% of the DoFs in a level 8,8,6,6 full-grid simulation.
The simulation took 81 minutes 35 seconds.

\Cref{fig:sod_shock:nu_1000:cons} plots the deviation in the total mass, momentum, and energy from the initial condition.
Since the inflow boundary condition is the initial condition, there is no provable global conservation law, but, even though the resolution near the boundary is quite coarse, the method preserves the global moments up to the expected adaptive tolerance.
If $I_\Theta M[\bmrho_{f_h^{(\cdot,*)}}]$ is used in place of $M_\Theta[\bmrho_{f_h^{(\cdot,*)}}]$, then oscillations in the flat regions of the distribution in position space build up over time which ends up adapting to a dense position grid $\Theta_{\bmx}$ by $t=0.15$.

\subsubsection{Boundary of fluid/transitional regime \texorpdfstring{$\nu=10$}{ν=10}}
\label{subsec:sod:nu_10}

We next take $\nu=10$ which puts us on the edge of where Navier--Stokes is considered a viable approximation and more into the rarefied regime.  
The nonlinear refinement function is set as $0.1n_f^{-1} M_\Theta[\bmrho_f]$ (see \Cref{subsec:refine_and_coarsen}).
We include the $n_f^{-1}$ factor to remove numerical artifacts in the fluid variables in a portion of the low density region.

\begin{figure}[htbp]
    \centering
    \begin{subfigure}{0.245\textwidth}
        \includegraphics[width=\textwidth]{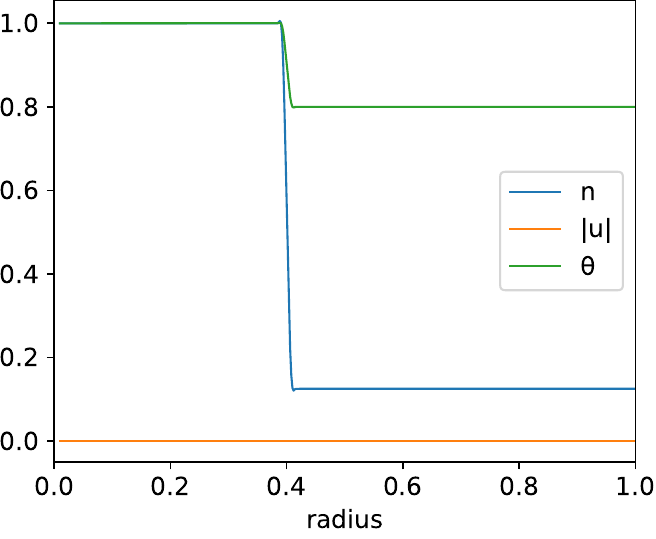}
        \caption{$t=0.00$}
        \label{fig:sod_shock:nu_10:fluid_vars_line:0.00}
    \end{subfigure}%
    \begin{subfigure}{0.245\textwidth}
        \includegraphics[width=\textwidth]{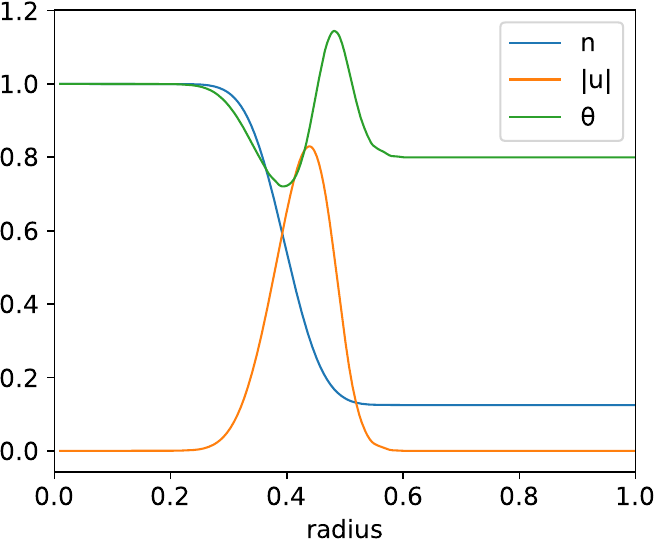}
        \caption{$t=0.05$}
        \label{fig:sod_shock:nu_10:fluid_vars_line:0.05}
    \end{subfigure}%
    \begin{subfigure}{0.245\textwidth}
        \includegraphics[width=\textwidth]{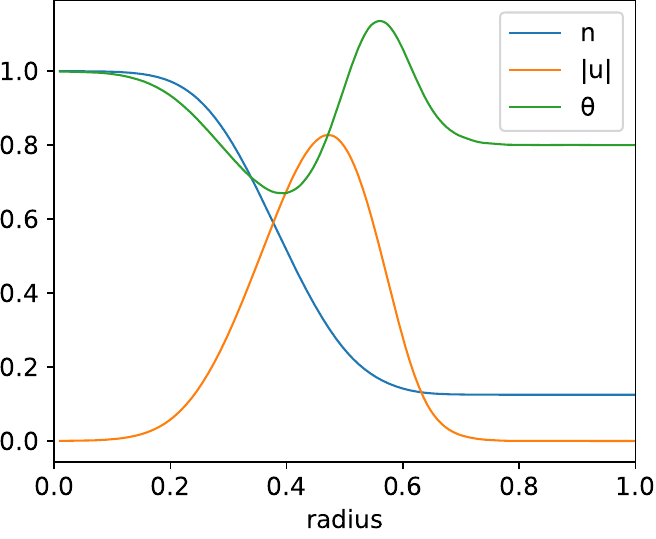}
        \caption{$t=0.10$}
        \label{fig:sod_shock:nu_10:fluid_vars_line:0.10}
    \end{subfigure}%
    \begin{subfigure}{0.245\textwidth}
        \includegraphics[width=\textwidth]{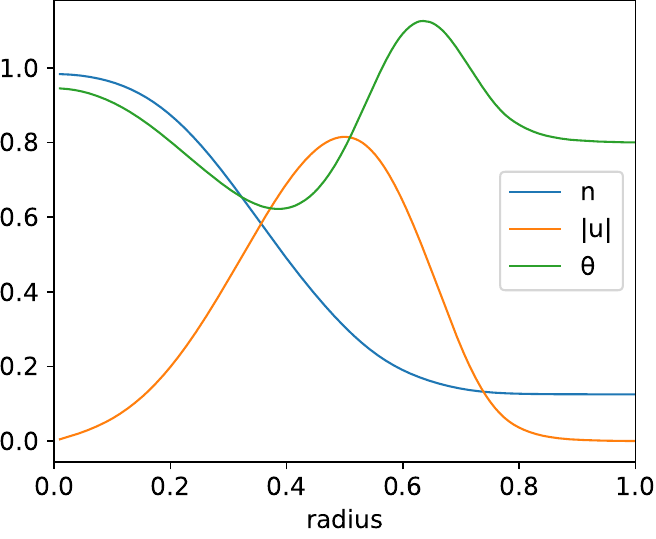}
        \caption{$t=0.15$}
        \label{fig:sod_shock:nu_10:fluid_vars_line:0.15}
    \end{subfigure}%
    \caption{\Cref{subsec:sod:nu_10} -- $2x2v$ Sod shock tube, $\nu=10$. Line plots of fluid variables along the ray $(\cos(\pi/5),\sin(\pi/5))$.}
    \label{fig:sod_shock:nu_10:fluid_vars_line}
\end{figure}

\Cref{fig:sod_shock:nu_10:fluid_vars_line} and \Cref{fig:sod_shock:nu_10:fluid_vars} plot the fluid variables and the adaptive position grid at $t=0.0, 0.05, 0.10, 0.15$.  At $t=0.05$, there is still a large gradient in the fluid variables, which accounts for the refined position grid around a radius of 0.5. 
However, for longer times the initial discontinuities are dissipated enough to let position grid become much coarser.

\begin{figure}[htbp]
    \centering
    \begin{subfigure}{0.245\textwidth}
        \includegraphics[width=\textwidth]{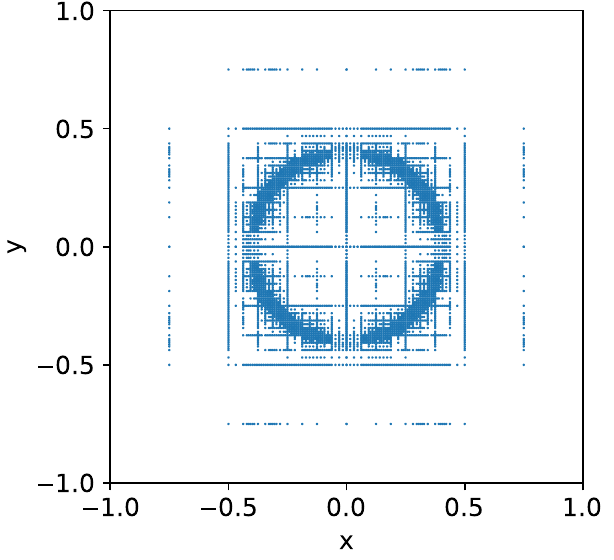}
        \caption{$t=0.00$}
        \label{fig:sod_shock:nu_10:fluid_vars:grid_0.00}
    \end{subfigure}%
    \begin{subfigure}{0.245\textwidth}
        \includegraphics[width=\textwidth]{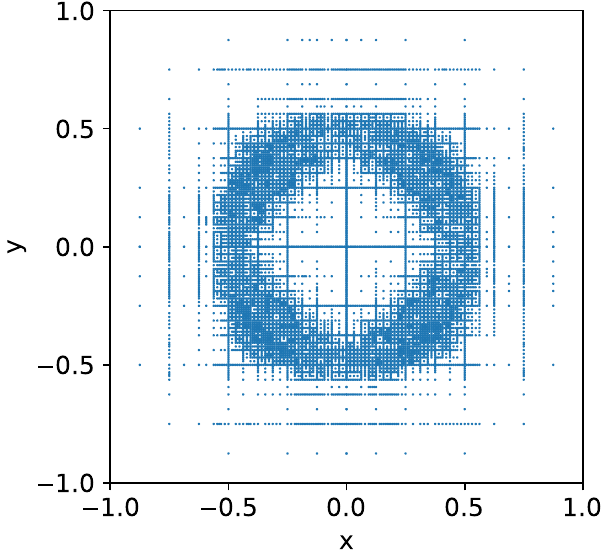}
        \caption{$t=0.05$}
        \label{fig:sod_shock:nu_10:fluid_vars:grid_0.05}
    \end{subfigure}%
    \begin{subfigure}{0.245\textwidth}
        \includegraphics[width=\textwidth]{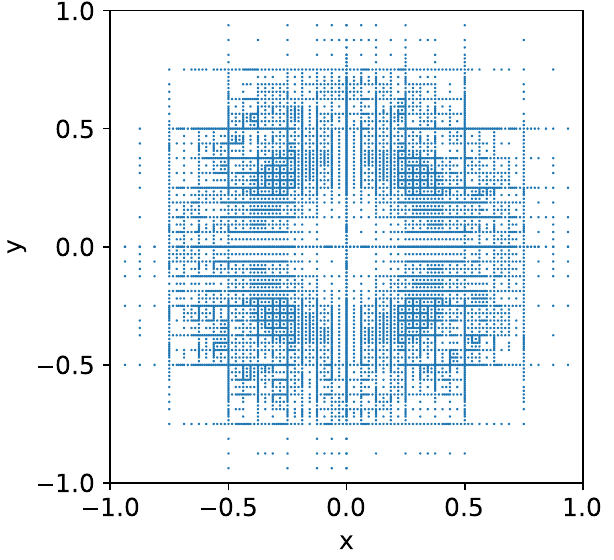}
        \caption{$t=0.10$}
        \label{fig:sod_shock:nu_10:fluid_vars:grid_0.10}
    \end{subfigure}%
    \begin{subfigure}{0.245\textwidth}
        \includegraphics[width=\textwidth]{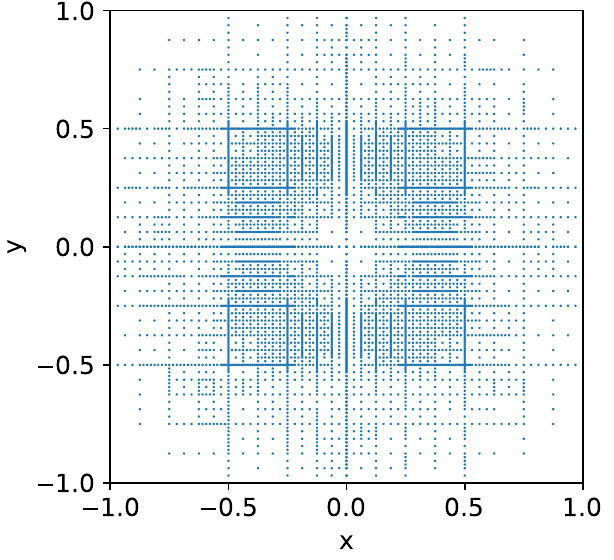}
        \caption{$t=0.15$}
        \label{fig:sod_shock:nu_10:fluid_vars:grid_0.15}
    \end{subfigure}%
    \caption{\Cref{subsec:sod:nu_10} -- $2x2v$ Sod shock tube, $\nu=10$. Position grid $\Theta_{\bmx}$ (see \eqref{eqn:position_grid}) denoted by the barycenters of the support of each active element for multiple times.}
    \label{fig:sod_shock:nu_10:fluid_vars}
\end{figure}

\Cref{fig:sod_shock:nu_10:velocity} plots the velocity distribution and perturbations along a radial cut of the position space at an angle $\pi/5$ from the positive $x$-axis; the perturbations are calculated as in \Cref{subsec:sod:nu_10}.
Unlike in the fluid regime, the velocity distribution features sharp discontinuities that the adaptive sparse-grid method captures well.
For $r=0.3$, the kinetic and CNS perturbations agree qualitatively, but the discontinuity is lost in the CNS perturbation.
For $r=0.5$, the kinetic and CNS perturbations look most dissimilar. 
For $r=0.7$, some numerical artifacts in the kinetic perturbation become apparent.

\begin{figure}[htbp]
    \centering
    \begin{subfigure}{0.25\textwidth}
        \includegraphics[width=\textwidth]{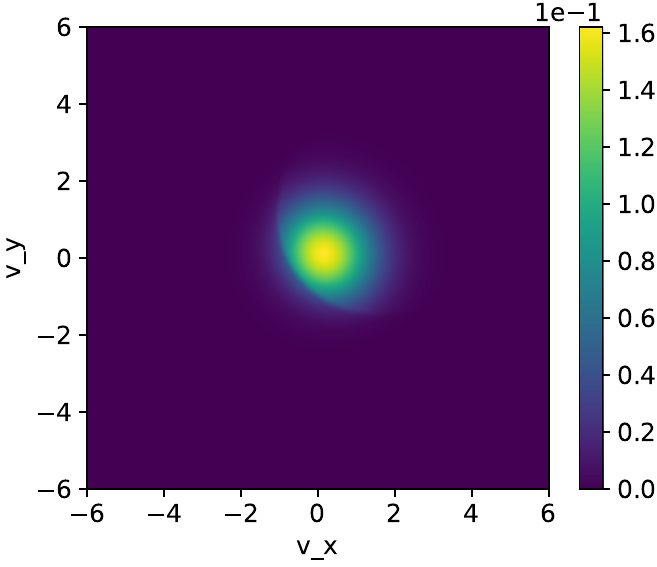}
        \caption{$f$, $r=0.3$}
    \label{fig:sod_shock:nu_10:velocity:dist_0.3}
    \end{subfigure}
    \hspace{0.25in}
    \begin{subfigure}{0.25\textwidth}
        \includegraphics[width=\textwidth]{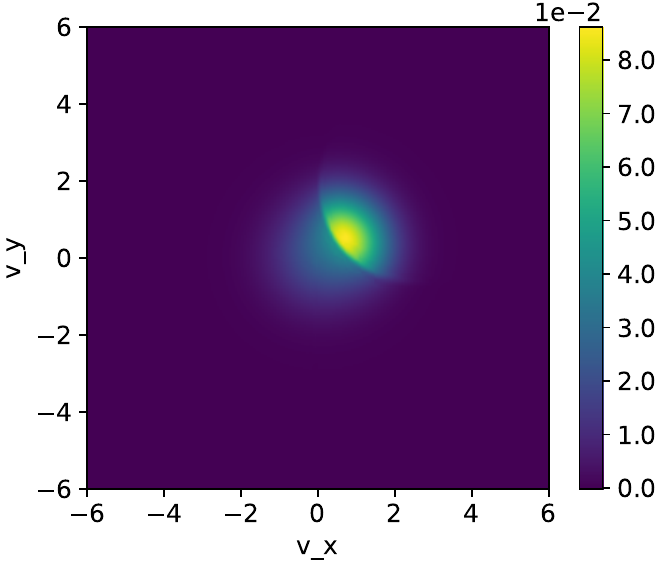}
        \caption{$f$, $r=0.5$}
    \label{fig:sod_shock:nu_10:velocity:dist_0.5}
    \end{subfigure}
    \hspace{0.25in}
    \begin{subfigure}{0.25\textwidth}
        \includegraphics[width=\textwidth]{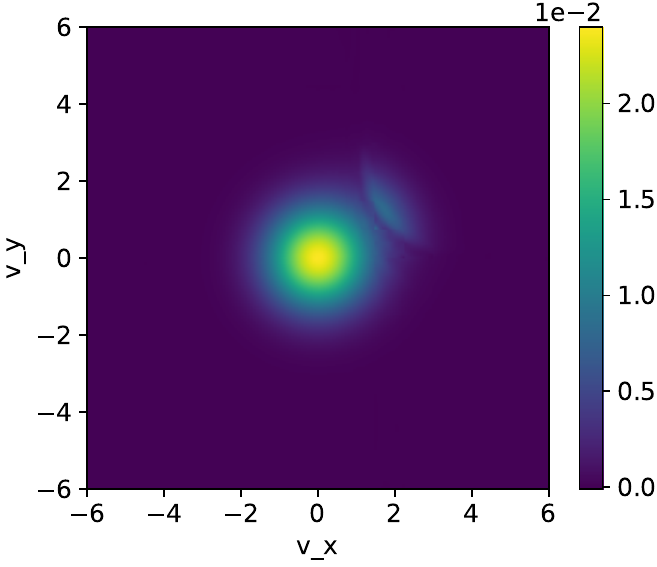}
        \caption{$f$, $r=0.7$}
    \label{fig:sod_shock:nu_10:velocity:dist_0.7}
    \end{subfigure}

    \vspace{2ex}
    
    \begin{subfigure}{0.25\textwidth}
        \includegraphics[width=\textwidth]{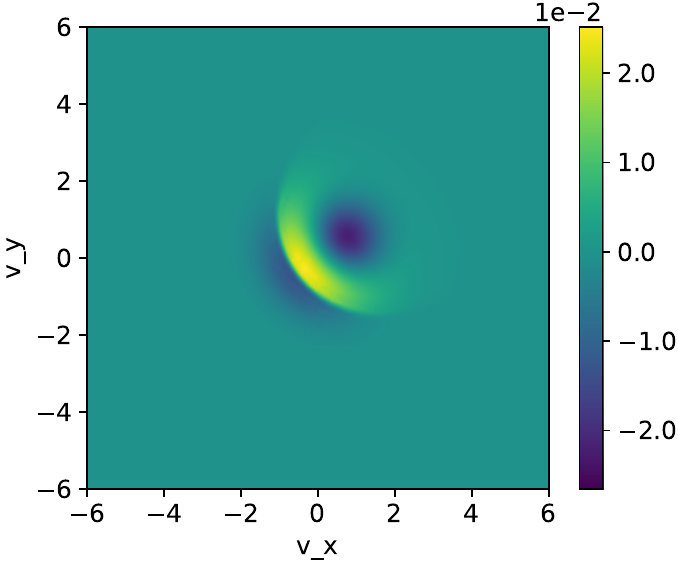}
        \captionsetup{justification=centering}
        \caption{K.P., $r=0.3$}
    \label{fig:sod_shock:nu_10:velocity:Kinetic_0.3}
    \end{subfigure}
    \hspace{0.25in}
    \begin{subfigure}{0.25\textwidth}
        \includegraphics[width=\textwidth]{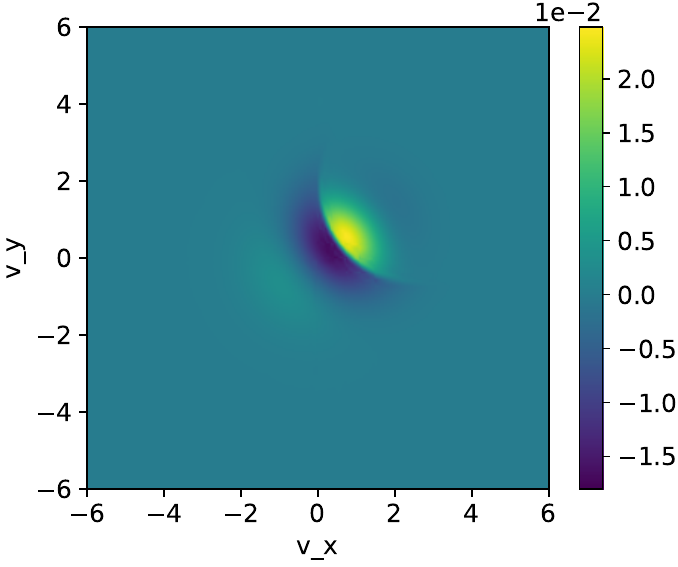}
        \captionsetup{justification=centering}
        \caption{K.P., $r=0.5$}
    \label{fig:sod_shock:nu_10:velocity:Kinetic_0.5}
    \end{subfigure}
    \hspace{0.25in}
    \begin{subfigure}{0.25\textwidth}
        \includegraphics[width=\textwidth]{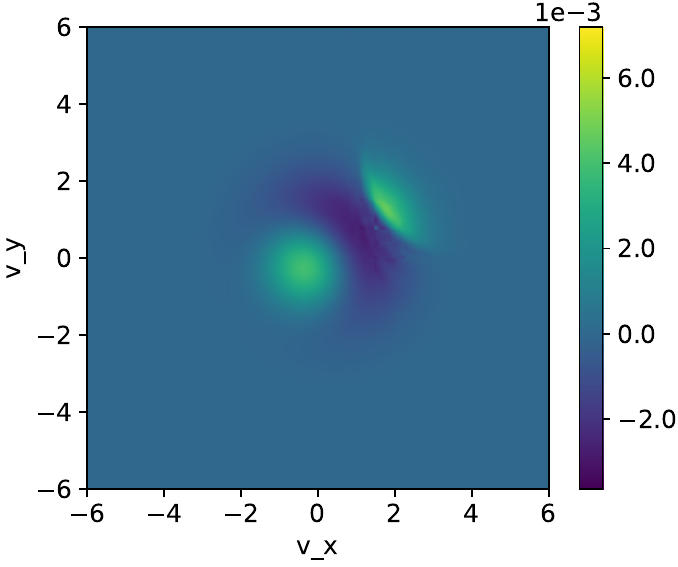}
        \captionsetup{justification=centering}
        \caption{K.P., $r=0.7$}
    \label{fig:sod_shock:nu_10:velocity:Kinetic_0.7}
    \end{subfigure}

    \vspace{2ex}

    \begin{subfigure}{0.25\textwidth}
        \includegraphics[width=\textwidth]{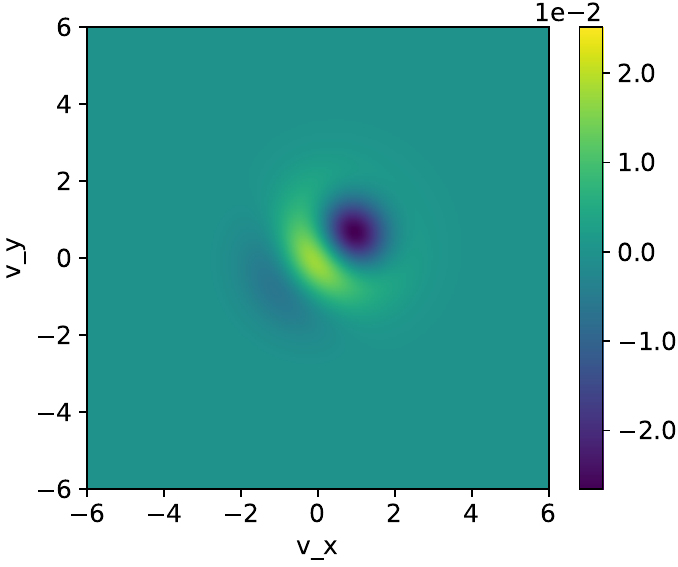}
        \caption{$g$, $r=0.3$}
    \label{fig:sod_shock:nu_10:velocity:CNS_0.3}
    \end{subfigure}
    \hspace{0.25in}
    \begin{subfigure}{0.25\textwidth}
        \includegraphics[width=\textwidth]{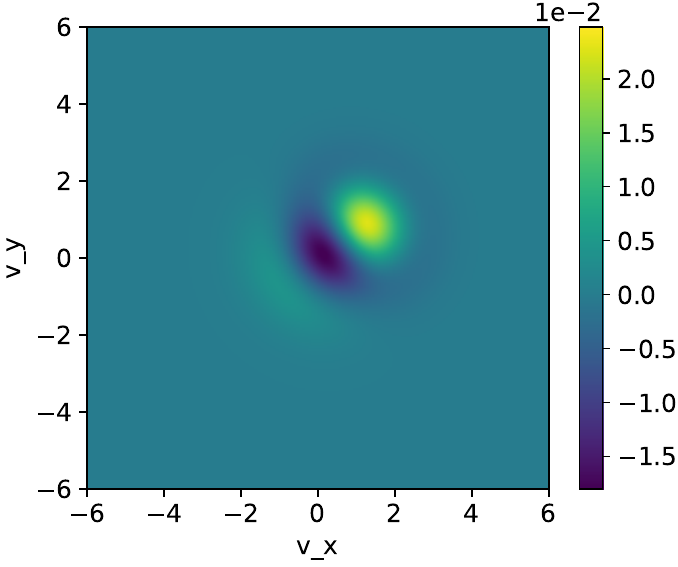}
        \caption{$g$, $r=0.5$}
    \label{fig:sod_shock:nu_10:velocity:CNS_0.5}
    \end{subfigure}
    \hspace{0.25in}
    \begin{subfigure}{0.25\textwidth}
        \includegraphics[width=\textwidth]{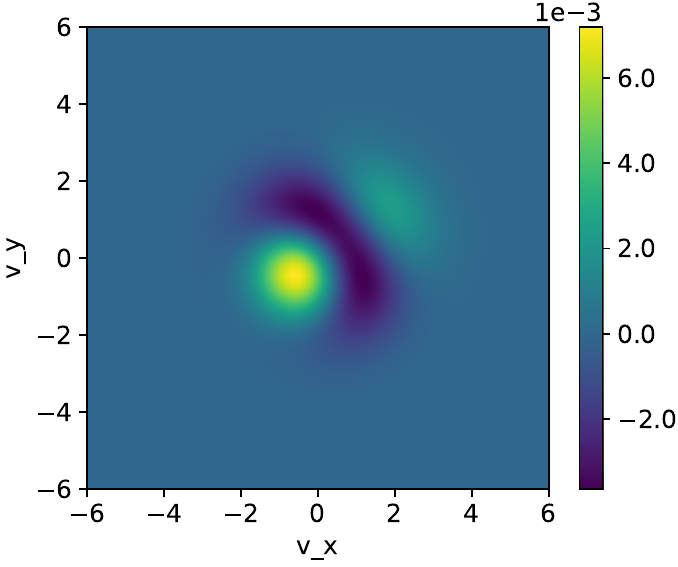}
        \caption{$g$, $r=0.7$}
    \label{fig:sod_shock:nu_10:velocity:CNS_0.7}
    \end{subfigure}

    \caption{\Cref{subsec:sod:nu_10} -- $2x2v$ Sod shock tube, $\nu=10$. Velocity Perturbations at $\bmx=(r\cos(\tfrac{\pi}{5}),r\sin(\tfrac{\pi}{5}))$ for various radii $r$ at $t=0.15$. 
    Here K.P.~refers to the kinetic perturbation $f(\bmx,\cdot)-M[\bmrho_f(\bmx)](\cdot)$. The compressible Navier-Stokes perturbation $g = g[\bmrho_f]$ is given in \eqref{eqn:chapman_enskog}.}
    \label{fig:sod_shock:nu_10:velocity}
\end{figure}

\Cref{fig:sod_shock:nu_10:dof} plots the DoFs of the phase space distribution and the moments $\bmrho_f$.
Unlike the fluid case (and verified by \Cref{fig:sod_shock:nu_10:velocity}) the velocity distribution is not at equilibrium, and thus we expect much more velocity DoFs will be required to resolve the distribution.
For short times, the opposite is the case -- the position DoFs outgrow the phase-space DoFs. 
We attribute this to the initial discontinuity largely living in the positional variables.
However, after $t=0.05$, the discontinuity has significantly shifted away from position space, and the physical DoFs start to decrease linearly while the phase-space DoFs keep increasing.  
At $t=0.15$ the phase space and position DoFs are approximately 56.2 million and 48 thousand respectively.
This gives an average per velocity dimension DoF of $\dofvavg\approx 34$, over two times the fluid case.
The simulation refines up to level 8 in velocity by $t=0.15$; hence the maximal 61.5 million DoFs at $t=0.15$ correspond to approximately 0.018\% of the DoFs in a level 8 full-grid simulation.
The simulation took 198 minutes and 47 seconds.

\begin{figure}[htbp]
    \centering
    \begin{subfigure}{0.35\textwidth}
        \centering
        \includegraphics[width=\textwidth]{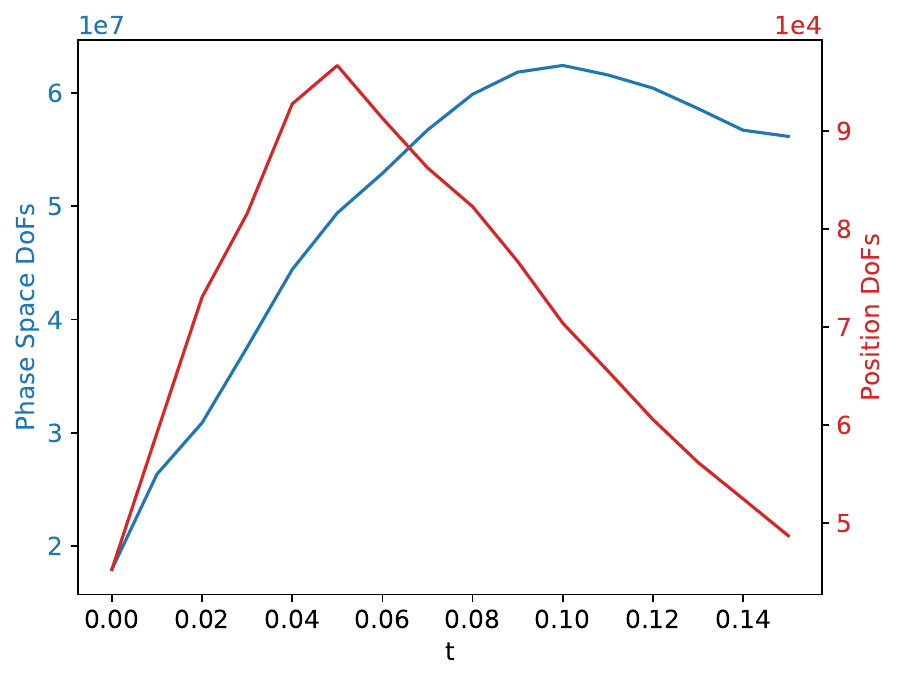}
        \caption{Degrees of freedom over time for the 4D phase space grid and the 2D moment grid.\\}
        \label{fig:sod_shock:nu_10:dof}
    \end{subfigure}
    \hspace{4em}
    \begin{subfigure}{0.35\textwidth}
        \centering
        \includegraphics[width=\textwidth]{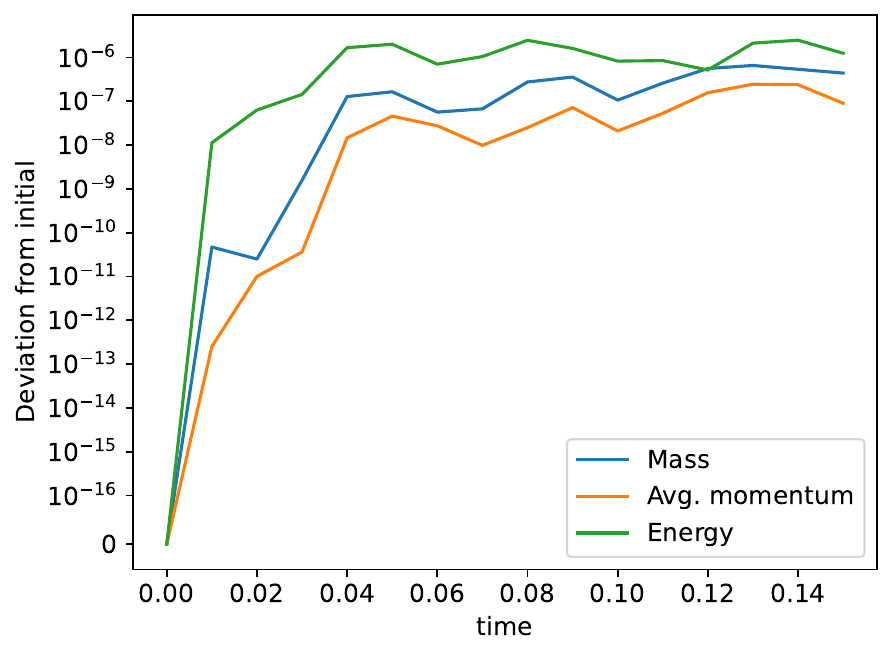}
        \caption{Absolute deviation of the total mass, average momentum, and energy with respect to the initial values.}
        \label{fig:sod_shock:nu_10:cons}
    \end{subfigure}
    \caption{\Cref{subsec:sod:nu_10} -- $2x2v$ Sod shock tube, $\nu=10$. Degrees of freedom and conservation quantities. 
    A level 8,8,8,8 full grid has approximately 348 billion (\num{3.48e11}) DoFs.}
    \label{fig:sod_shock:nu_10:dof_and_cons}
\end{figure}

\Cref{fig:sod_shock:nu_10:cons} plots the deviation in the total mass, momentum, and energy from the initial condition and shows results similar to the fluid case.

\subsubsection{Transitional regime \texorpdfstring{$\nu=1$}{ν=1}}
\label{subsec:sod:nu_1}

We finally take $\nu=1$ where kinetic descriptions are required and fluid models like Navier--Stokes are generally not applicable. 
The nonlinear refinement function is set as $n_f^{-1}0.1 M_\Theta[\bmrho_f]$ (see \Cref{subsec:refine_and_coarsen}).

\begin{figure}[htbp]
    \centering
    \begin{subfigure}{0.25\textwidth}
    \centering
        \includegraphics[width=\textwidth]{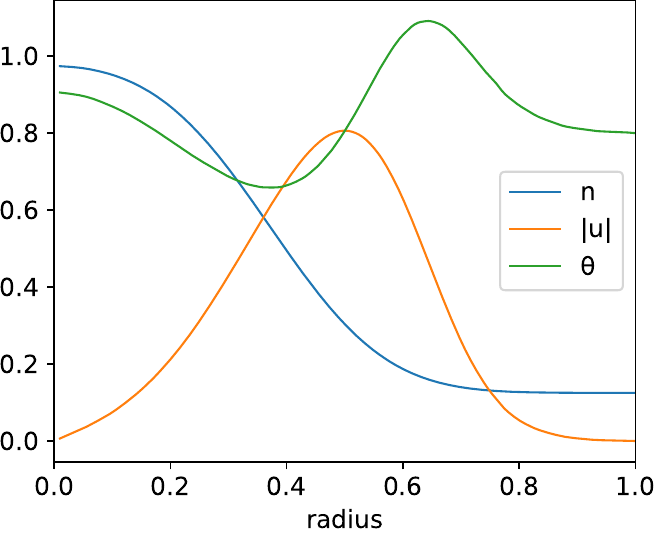}
        \caption{Fluid Variables}
        \label{fig:sod_shock:nu_1:fluid_vars_line}
    \end{subfigure}
    \hspace{0.25in}
    \begin{subfigure}{0.25\textwidth}
    \centering
        \includegraphics[width=0.9\textwidth]{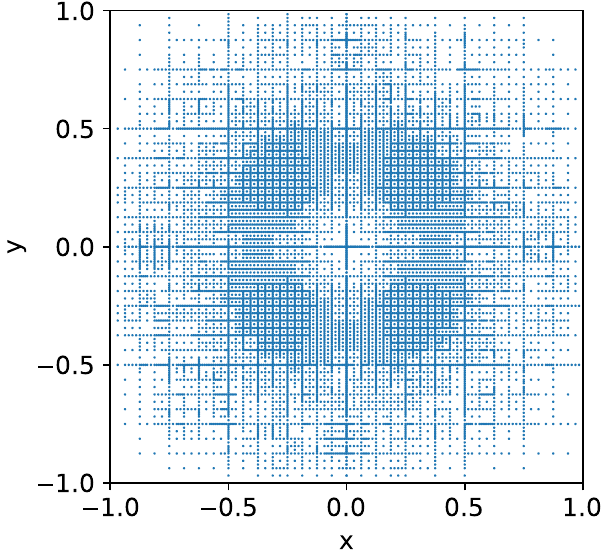}
        \caption{Position grid $\Theta_{\bmx}$}
        \label{fig:sod_shock:nu_1:grid}
    \end{subfigure}
    \hspace{0.25in}
    \begin{subfigure}{0.30\textwidth}
    \centering
        \includegraphics[width=\textwidth]{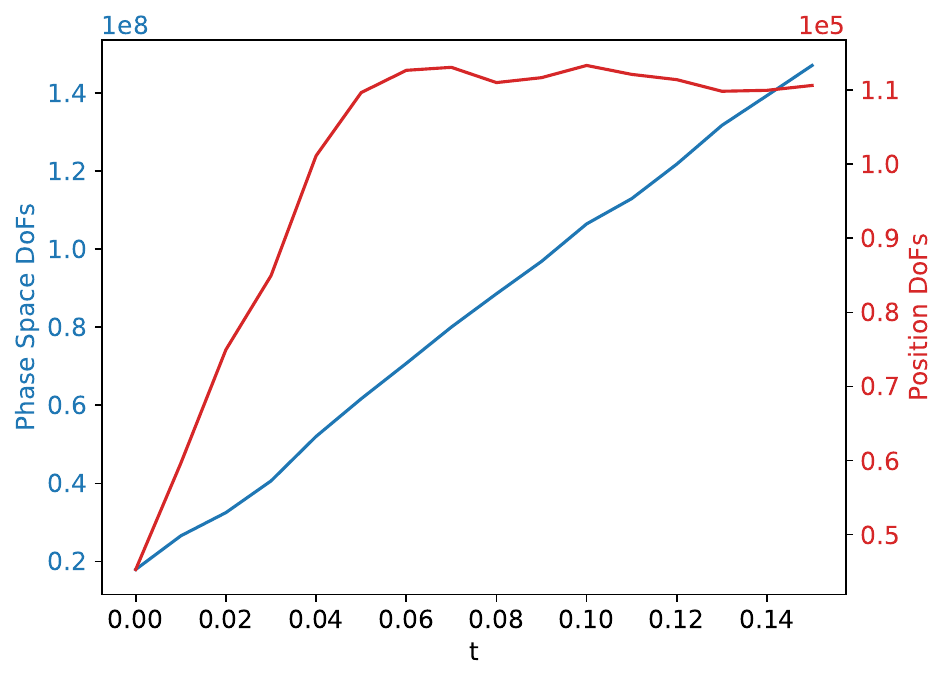}
        \caption{DoFs over time}
        \label{fig:sod_shock:nu_1:dofs}
    \end{subfigure}
    \caption{\Cref{subsec:sod:nu_1} -- $2x2v$ Sod shock tube, $\nu=1$. Plots of radial fluid variables and position grid at $t=0.15$ and DoFs over time.  A level 8,8,8,8 full-grid has approximately 347 billion (\num{3.47e11}) DoFs.}
    \label{fig:sod_shock:nu_1:fluid_vars_and_dofs}
\end{figure}

For brevity, we only plot the fluid variables and grid in \Cref{fig:sod_shock:nu_1:fluid_vars_line,fig:sod_shock:nu_1:grid} for $t=0.15$.  There is a slight difference in the fluid variables as compared to $\nu=10$, c.f.~\Cref{fig:sod_shock:nu_10:fluid_vars_line:0.15}; however, the position grid $\Theta_x$, c.f.~\Cref{fig:sod_shock:nu_10:fluid_vars:grid_0.15}, is denser.
We conjecture the increased fill in $\Theta_{\bmx}$ is because the initial contact discontinuity, being advected into phase-space, is smoothed less aggressively than in the $\nu=10$ case. 
Resolving the discontinuity in phase-space while maintaining ancestor completeness of $\Theta$ would promote more DoFs in position space.
To provide further evidence we list the times at which a finer velocity level was required to resolve the discontinuity.  The initial condition adapts to level 6 in velocity for both $\nu=10$ and $\nu=1$.
For $\nu=10$, the simulation adapts to level 7 and 8 at $t=0.05$ and $0.15$ respectively, while for $\nu=1$, the same resolution is required at $t=0.04$ and $0.08$ respectively.

\begin{figure}[htbp]
    \centering
    \begin{subfigure}{0.25\textwidth}
        \includegraphics[width=0.97\textwidth]{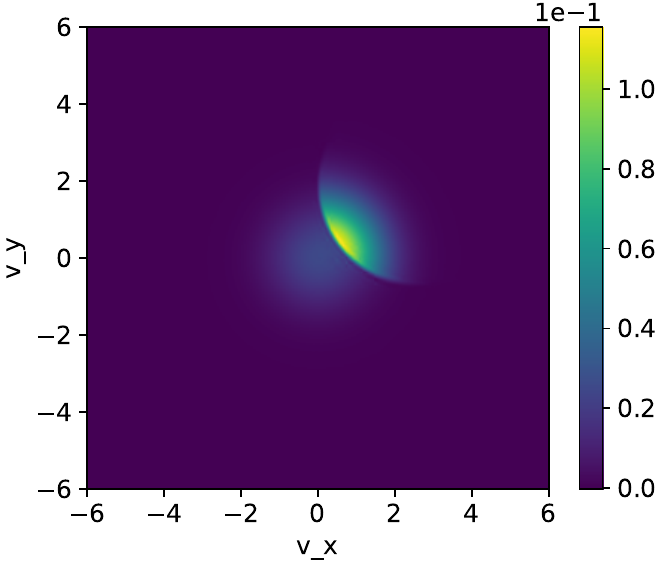}
        \caption{$f$.}
    \label{fig:sod_shock:nu_1:velocity:dist}
    \end{subfigure}
    \hspace{0.25in}
    \begin{subfigure}{0.25\textwidth}
        \includegraphics[width=\textwidth]{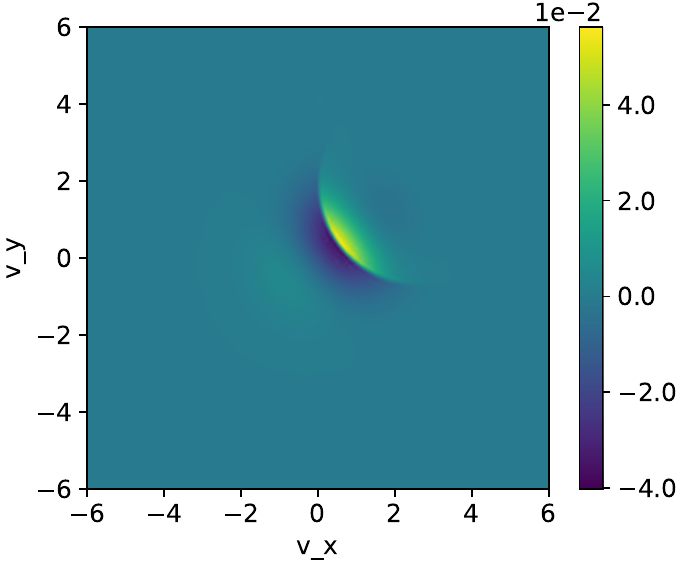}
        \caption{$f-M[\bm{\rho}_f]$.}
    \label{fig:sod_shock:nu_1:velocity:pertrub}
    \end{subfigure}
    \hspace{0.25in}
    \begin{subfigure}{0.25\textwidth}
        \includegraphics[width=\textwidth]{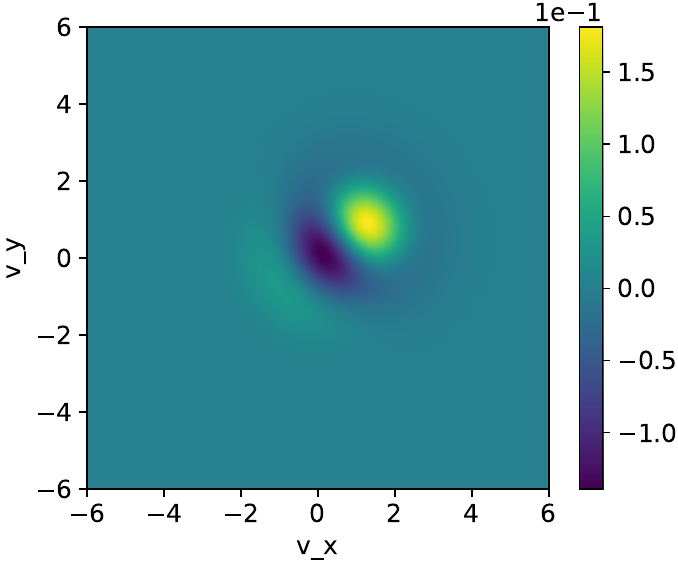}
        \caption{$g[\bmrho_f]$}
    \label{fig:sod_shock:nu_1:velocity:CNS}
    \end{subfigure}

    \caption{\Cref{subsec:sod:nu_1} -- $2x2v$ Sod shock tube, $\nu=1$. Velocity Perturbations at $\bmx=(0.5\cos({\pi}/{5}),0.5\sin({\pi}/{5}))$ at $t=0.15$. The perturbation $g[\bmrho_f]$ is given in \eqref{eqn:chapman_enskog}.}
    \label{fig:sod_shock:nu_1:velocity}
\end{figure}

\Cref{fig:sod_shock:nu_1:velocity} plots the velocity distribution as well as the kinetic perturbation $f-M[\bmrho_f]$ and Navier-Stokes perturbations $g[\bmrho_f]$ at a radius of $0.5$ and $t=0.15$.
Unlike the $\nu=1000$ and $\nu=10$ cases, the size of kinetic and CNS perturbation are off by nearly by a factor of 4, and the overall shapes do not agree.
Compared to the $\nu=10$ case (\Cref{fig:sod_shock:nu_10:velocity:Kinetic_0.5}) the support of kinetic perturbation (\Cref{fig:sod_shock:nu_1:velocity:pertrub}) is further concentrated and features a sharper gradient that is still well captured by the adaptivity procedure.
Moreover, the ansatz $M[\bmrho_f]+g[\bmrho_f]$ used in the compressible Navier--Stokes equation is negative at some $\bmv$ for this choice of $\nu$ and $\bmrho_f$.

\Cref{fig:sod_shock:nu_1:dofs} plots the phase-space and position DoFs over time.  There is a saturation of the position DoFs and a steady increase of the phase-space DoFs.
At $t=0.15$ the phase space and position DoFs are approximately 150 million and 110 thousand respectively; this gives an average per velocity dimension DoF of $\dofvavg\approx 36.4$.
The simulation with maximal phase-space DoFs of 148 million uses about 0.04\% the DoFs of a level 8 full-grid simulation.
The simulation took 6 hours and 38 minutes. 
We do not list the global conservation quantities as they are similar to \Cref{fig:sod_shock:nu_1000:cons}.

\subsection{\texorpdfstring{$\bm{2x2v}$}{2x2v} shear flow}
\label{subsec:shearflow}

We consider the shear flow problem from \cite[Section 4.2.1]{dektor2026InterpolatoryDynamical} and \cite[Section 7.2]{einkemmer2021EfficientDynamical}.
This problem tests the adaptive sparse-grid method's ability to capture small perturbations of the distribution and maintain long-time conservation of the collision invariants.

The position and velocity domains are given by $\W_{\bmx}=[0,1]^2$ and $\W_{\bmv}=[-6,6]^2$.
The initial condition is a Maxwellian with $n=\theta=1$ and bulk velocity components
\begin{equation}\label{eqn:shear_flow:initial_velocity}
    u_x(x,y) = v_0
    \begin{cases}
        \tanh\big(\frac{y-0.25}{\Delta}\big)&\text{if }y\leq 0.5 \\[0.2em]
        \tanh\big(\frac{0.75-y}{\Delta}\big)&\text{if }y> 0.5,
    \end{cases}
    \qquad\text{and}\qquad
    u_y(x,y) = \delta\sin(2\pi x),
\end{equation}
where $v_0=0.1$, $\Delta=1/30$, and $\delta=5\times 10^{-3}$.
We consider the highly collisional case of $\nu=10^4$.

\paragraph{Discretization Specifics} 
We set $k=3$ and a maximum level of $5,5,2,2$; this corresponds to a full-grid DoF of \num{4194304}.
We set $\dt = 10^{-4}$, which is approximately 27\% the maximum explicit timestep, and run for \num{120000} timesteps to a time $t=12.0$.
We set $\tau_\text{rel} = 10^{-5}$ and $10^{-6}$ and $\tau_\text{abs} = 0$.  
We do not set a nonlinear refinement function.

\paragraph{Results}

\Cref{fig:shear_flow:fluid_vars} plots the density, bulk velocity, and vorticity $\omega=\grad\times\bmu=\partial_{x}u_y-\partial_yu_x$ for $\tau_\text{rel}$ of $10^{-5}$ and $10^{-6}$.
Comparing the adaptive sparse grid with $\tau_\text{rel}=10^{-5}$ to \cite[Figure 2]{dektor2026InterpolatoryDynamical} and \cite[Figure 1]{einkemmer2021EfficientDynamical}, there is good agreement between the bulk velocity while the vorticity shows under-resolved behavior in the adaptive sparse grid around $(x,y)=(0.5,0.75)$ (see \Cref{fig:shear_flow:fluid_vars:vort_1e-5}).
Decreasing the relative threshold to $10^{-6}$ resolves the vorticity  on a finer position grid shown with red dots (see \Cref{fig:shear_flow:fluid_vars:vort_1e-6}).
We notice a discrepancy between the density in \cite[Figure 2]{dektor2026InterpolatoryDynamical} and \Cref{fig:shear_flow:fluid_vars:den_1e-5} which we believe is attributed to a mass conservation error on the order of the density perturbation $10^{-3}$, c.f.~\cite[Figure 3]{dektor2026InterpolatoryDynamical}, for the interpolative DLR scheme. 

\begin{figure}[htbp]
    \centering
    \begin{subfigure}{0.24\textwidth}
        \includegraphics[width=\textwidth]{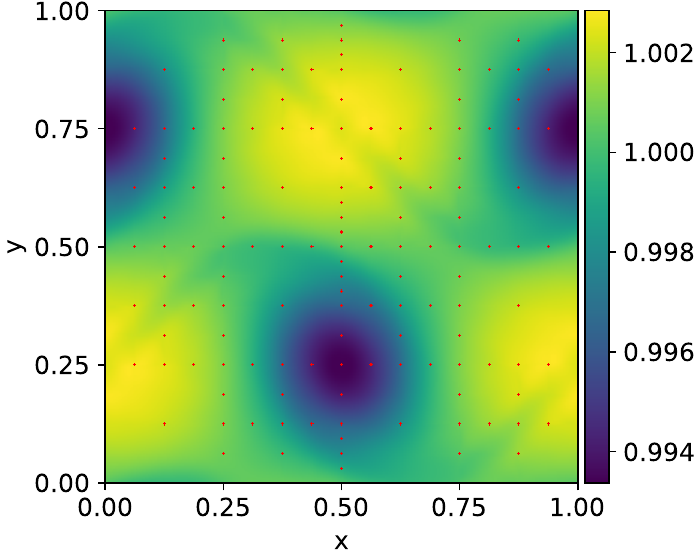}
        \caption{$n$, $\tau_\text{rel}=10^{-5}$}
        \label{fig:shear_flow:fluid_vars:den_1e-5}
    \end{subfigure}
    \begin{subfigure}{0.245\textwidth}
        \includegraphics[width=\textwidth]{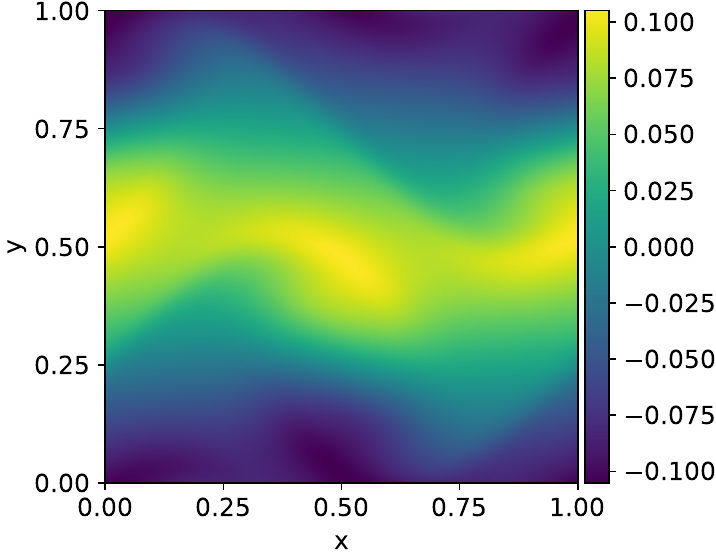}
        \caption{$u_x$, $\tau_\text{rel}=10^{-5}$}
        \label{fig:shear_flow:fluid_vars:ux_1e-5}
    \end{subfigure}
    \begin{subfigure}{0.245\textwidth}
        \includegraphics[width=\textwidth]{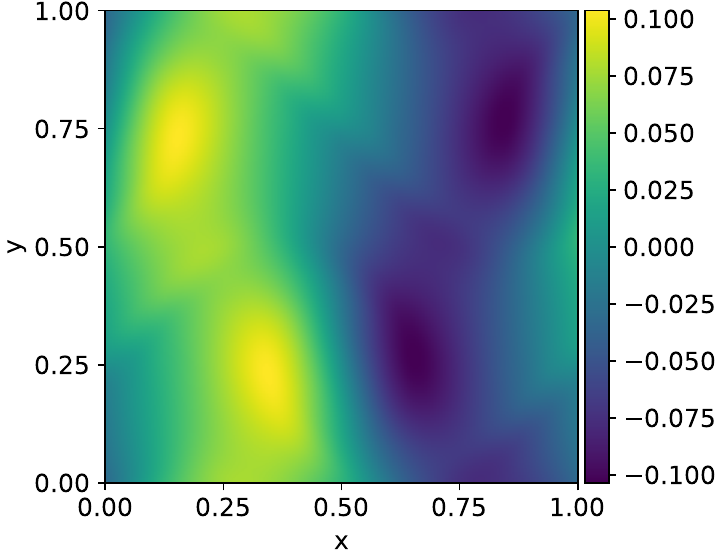}
        \caption{$u_y$, $\tau_\text{rel}=10^{-5}$}
        \label{fig:shear_flow:fluid_vars:uy_1e-5}
    \end{subfigure}
    \begin{subfigure}{0.235\textwidth}
        \includegraphics[width=\textwidth]{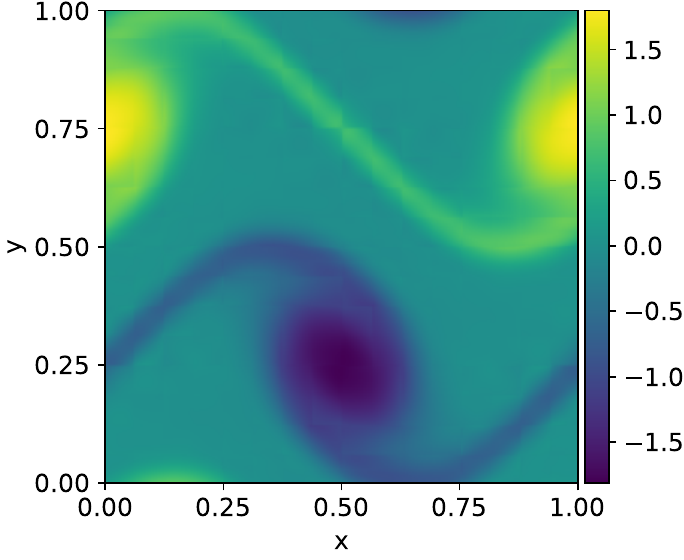}
        \caption{$\omega$, $\tau_\text{rel}=10^{-5}$}
        \label{fig:shear_flow:fluid_vars:vort_1e-5}
    \end{subfigure}

    \vspace{2ex}

    \begin{subfigure}{0.24\textwidth}
        \includegraphics[width=\textwidth]{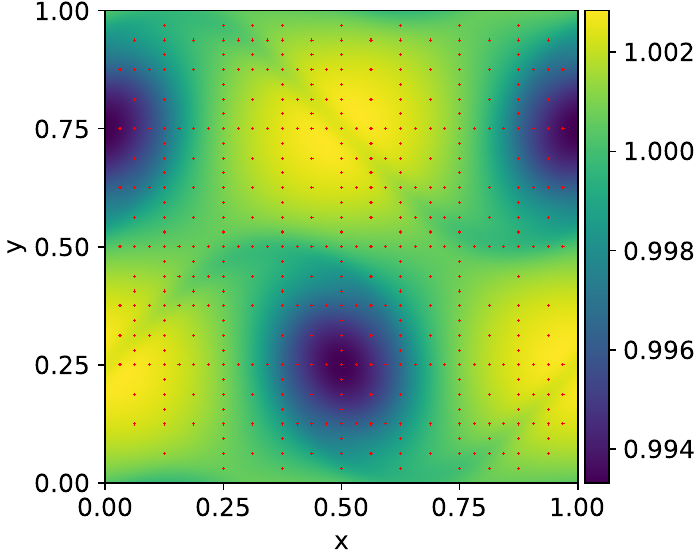}
        \caption{$n$, $\tau_\text{rel}=10^{-6}$}
        \label{fig:shear_flow:fluid_vars:den_1e-6}
    \end{subfigure}
    \begin{subfigure}{0.245\textwidth}
        \includegraphics[width=\textwidth]{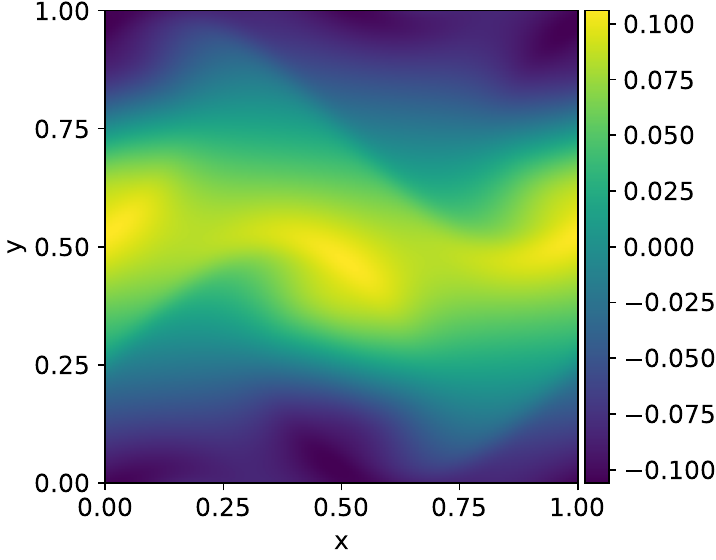}
        \caption{$u_x$, $\tau_\text{rel}=10^{-6}$}
        \label{fig:shear_flow:fluid_vars:ux_1e-6}
    \end{subfigure}
    \begin{subfigure}{0.245\textwidth}
        \includegraphics[width=\textwidth]{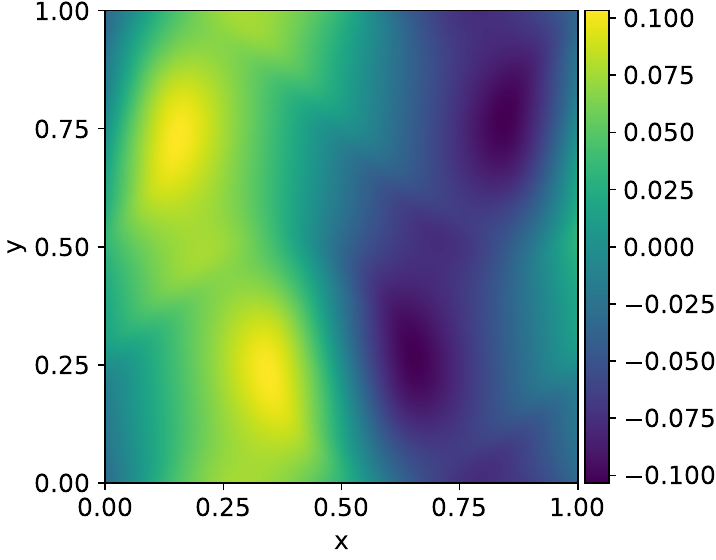}
        \caption{$u_y$, $\tau_\text{rel}=10^{-6}$}
        \label{fig:shear_flow:fluid_vars:uy_1e-6}
    \end{subfigure}
    \begin{subfigure}{0.235\textwidth}
        \includegraphics[width=\textwidth]{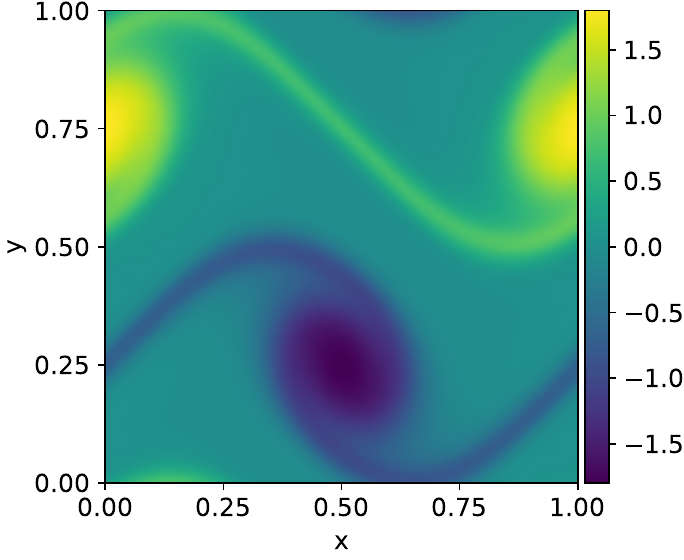}
        \caption{$\omega$, $\tau_\text{rel}=10^{-6}$}
        \label{fig:shear_flow:fluid_vars:vort_1e-6}
    \end{subfigure}

    \caption{\Cref{subsec:shearflow} -- $2x2v$ Shear flow, $\nu=10^4$. Plots of position quantities at $t=12.0$ and relative adaptive thresholds $\tau_\text{rel}=10^{-5}$ and $10^{-6}$.  The red dots in the density correspond to the position grid.}
    \label{fig:shear_flow:fluid_vars}
\end{figure}

We now test the effects on solution quality if the phase-space interpolated Maxwellian $I_\Theta M[\bmrho_f]$ is used instead of $M_\Theta[\bmrho_f]$.
We use the moments of $f_h^{(\cdot,*)}$ to build the sources in \eqref{eqn:IMEX_RK}, thus the conservation laws will fail to hold.
We set the maximum level as $5,5,4,4$ and include an adapt weight of $I_\Theta M[\bmrho_f]$; we find the maximum level of 4 in velocity is needed to accurately interpolate the Maxwellian.
\Cref{fig:shear_flow:interp_densities} plots the density at $t=12.0$ for $\tau_\text{rel}$ of \num{1e-6}, \num{5e-7}, and \num{1e-7}, and the results show that only $\tau_\text{rel}=\num{1e-7}$ is usable; looser tolerances exhibit numerical artifacts well above the true density perturbation.

\begin{figure}[htbp]
    \centering
    \begin{subfigure}{0.32\textwidth}
        \centering
        \includegraphics[width=\textwidth]{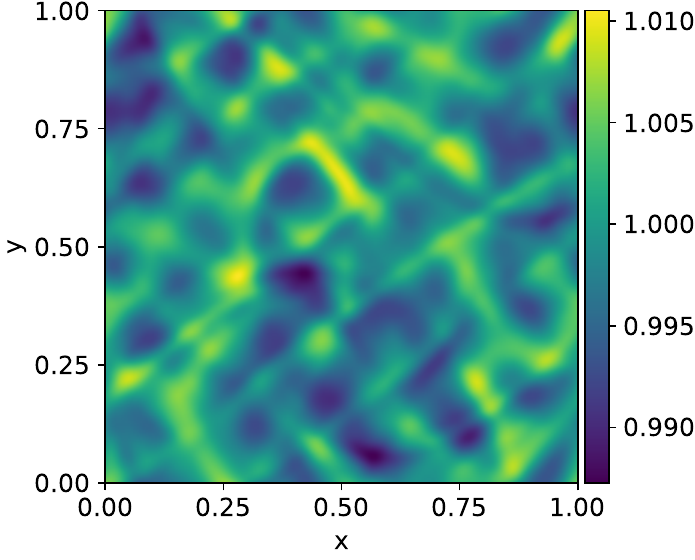}
        \caption{$\tau_\text{rel}=10^{-6}$}
        \label{fig:shear_flow:interp_densities_1e-6}
    \end{subfigure}
    \hfill
    \begin{subfigure}{0.32\textwidth}
        \centering
        \includegraphics[width=\textwidth]{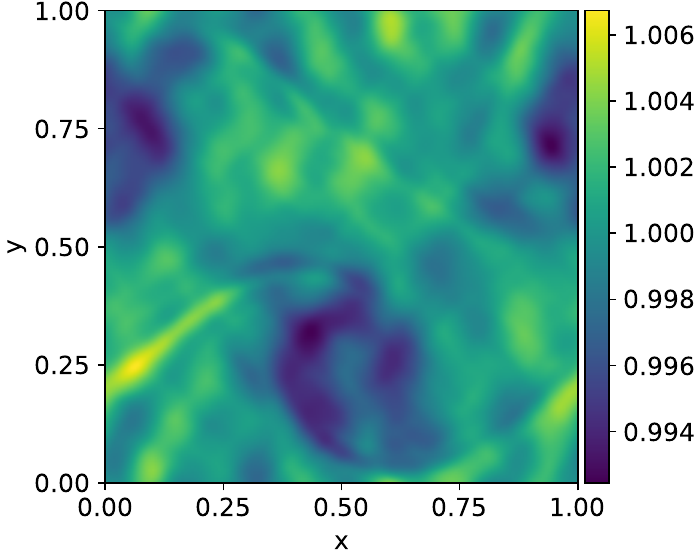}
        \caption{$\tau_\text{rel}=5\times 10^{-7}$}
        \label{fig:shear_flow:interp_densities_5e-7}
    \end{subfigure}
    \hfill
    \begin{subfigure}{0.32\textwidth}
        \centering
        \includegraphics[width=\textwidth]{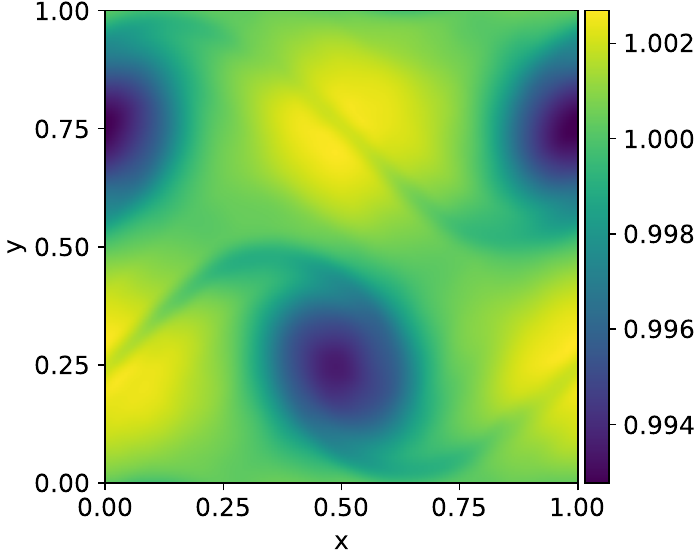}
        \caption{$\tau_\text{rel}=10^{-7}$}
        \label{fig:shear_flow:interp_densities_1e-7}
    \end{subfigure}
    \caption{\Cref{subsec:shearflow} -- $2x2v$ Shear flow, $\nu=10^4$. Plots of the density $n_f$ at $t=12.0$ using the phase-space interpolated Maxwellian $I_\Theta M[\bmrho_f]$ at various $\tau_\text{rel}$.}
    \label{fig:shear_flow:interp_densities}
\end{figure}

\Cref{fig:shear_flow:dofs_and_cons:dofs} plots the phase-space and position DoFs for the hybrid Maxwellian with $\tau_\text{rel}=10^{-5}$ and $10^{-6}$ and the phase-space interpolated Maxwellian with $\tau_\text{rel}=10^{-7}$.
Fine-scale oscillations formed by the lack of conservation in phase-space interpolated collision lead to a dense position grid by $t\approx 1$ with phase-space DoFs being near a constant multiple of position DoFs.
The hybrid Maxwellian requires significantly fewer DoFs to sufficiently approximate the Maxwellian in the fluid regime and leads to several orders of magnitude fewer phase-space DoFs than the phase-space interpolated Maxwellian being required to capture fluid behavior. 

\begin{figure}[htbp]
    \centering
    \begin{subfigure}{0.40\textwidth}
        \centering
        \includegraphics[height=1.5in]{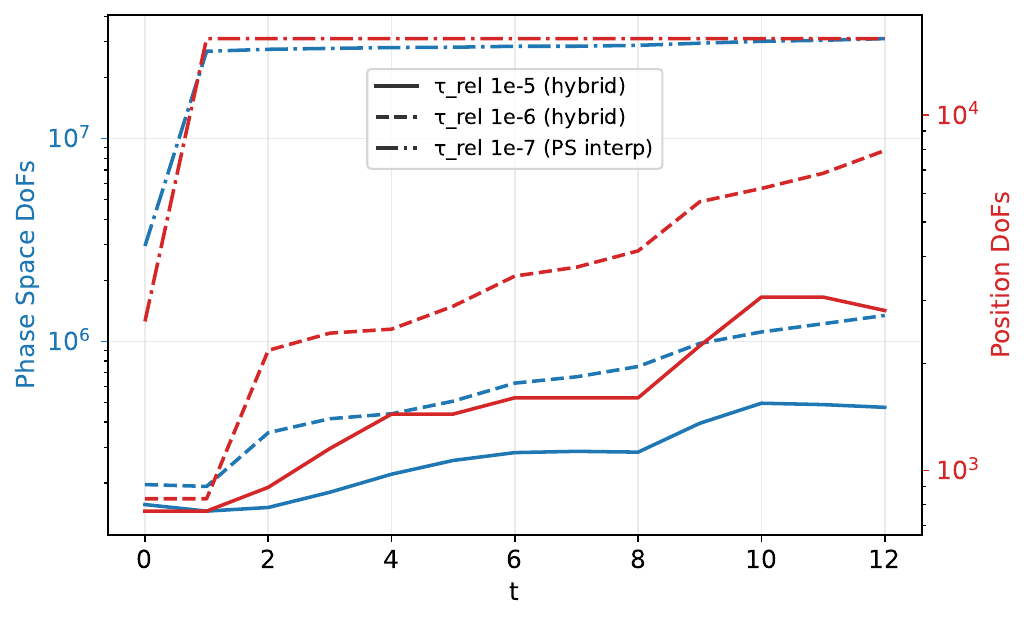}
        \caption{DoFs}
        \label{fig:shear_flow:dofs_and_cons:dofs}
    \end{subfigure}
    \begin{subfigure}{0.40\textwidth}
        \centering
        \includegraphics[height=1.5in]{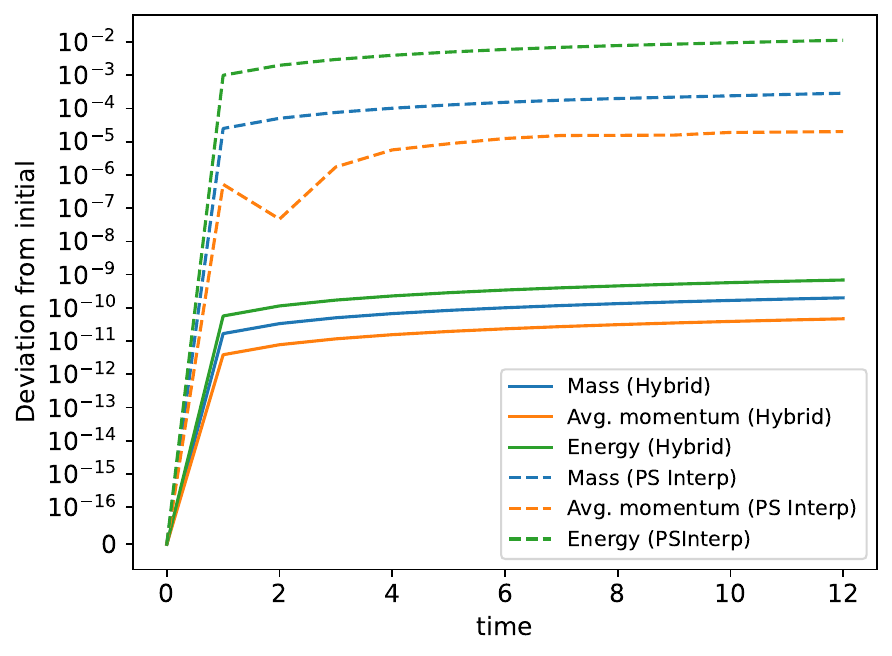}
        \caption{Conservation.}
        \label{fig:shear_flow:dofs_and_cons:cons}
    \end{subfigure}
    \caption{\Cref{subsec:shearflow} -- $2x2v$ Shear flow, $\nu=10^4$. Plots of DoFs and conservation deviations for $\tau_\text{rel}=10^{-5}$ and $10^{-6}$ with hybrid interpolation $M_\Theta[\bmrho]$ (with maximum level 5,5,2,2) and $\tau_\text{rel}=10^{-7}$ with phase-space interpolation $I_\Theta M[\bmrho]$ (with maximum level 5,5,4,4).
    The deviations for $\tau=10^{-6}$ with hybrid interpolation is similar to $\tau=10^{-5}$ and is omitted.
    }
    \label{fig:shear_flow:dofs_and_cons}
\end{figure}

We now consider only the hybrid Maxwellian runs.
We report \num{472576} and \num{1311184} phase-space DoFs at $t=12.0$ for $\tau_\text{rel}$ of $10^{-5}$ and $10^{-6}$ respectively and note the DoF required are significantly fewer than the level 5,5,2,2 full-grid DoFs of 4.19 million (\num{4.19e6}).
The adaptive interpolatory DLR approach in \cite{dektor2026InterpolatoryDynamical} had a maximal rank of 22 which corresponds to \num{366564} phase-space DoFs.
Overall, the DoF increase in \Cref{fig:shear_flow:dofs_and_cons:dofs} follows a similar trend to the rank increase in \cite[Figure 3]{dektor2026InterpolatoryDynamical}.
The average per velocity dimension DoF of $\dofvavg\approx 13.0$ for both tolerances, showing that the velocity grid is mostly dense and the computational savings are primarily from the adaptivity in position space.
The simulation with $\tau_\text{rel}=10^{-5}$ took 35 minutes 39 seconds.

\Cref{fig:shear_flow:dofs_and_cons:cons} plots the error in the total mass, momentum, and energy from the initial condition.
Since this problem is equipped with periodic boundary conditions, it has a provable global conservation property which well captured by the results for the hybrid Maxwellian.
The phase-space interpolatory Maxwellian shows conservation error several orders of magnitude above the relative tolerance.

\subsection{Expansion problem comparison}

We use the problem setup from \cite[Section 4.3]{dektor2026InterpolatoryDynamical} in order to compare the trends in DoFs between low-rank and sparse-grid approaches over varying collision frequencies. 
Let $\W_x=[-3,3]^d$ and $\W_v=[-6,6]^d$ and consider the separable initial condition
\begin{equation}\label{eqn:expansion:separable_ic}
    f(\bmx,\bmv,0) = \frac{1}{m}\bigg(1+0.25\prod_{i=1}^d\exp\Big(-\frac{x_i^2}{2\sigma^2}\Big)\bigg)\bigg(\prod_{i=1}^d\exp\Big(-\frac{v_i^2}{2}\Big)\bigg)
\end{equation}
where $\sigma=0.25$ and $m$ is chosen such that $\int_{\Omega_{\bmx}}\int_{\Omega_{\bmv}} f\,\mathrm{d}\bmv\,\mathrm{d}\bmx=1$.
We set periodic boundary conditions.

\subsubsection{\texorpdfstring{$\bm{2x2v}$}{2x2v} with \texorpdfstring{$\bm{\nu=10}$}{ν=10}}
\label{subsec:expansion:2x2v}

\paragraph{Discretization Specifics}

We let $d=2$ and $\nu=10$.  We use a maximum level of $5,5,3,3$ with $k=3$ to match the full-grid DoF as in \cite{dektor2026InterpolatoryDynamical}.
We set $\dt=2\times 10^{-3}$, which is approximately 90\% of $\dt_\text{expl}$ in \eqref{eqn:maximum_explicit_timestep}, and run for 375 steps.
We do not set a nonlinear refinement function and set $\tau_\text{abs}=0$.

\begin{figure}[!htbp]
    \centering
    \begin{subfigure}{0.321\textwidth}
        \includegraphics[width=\textwidth]{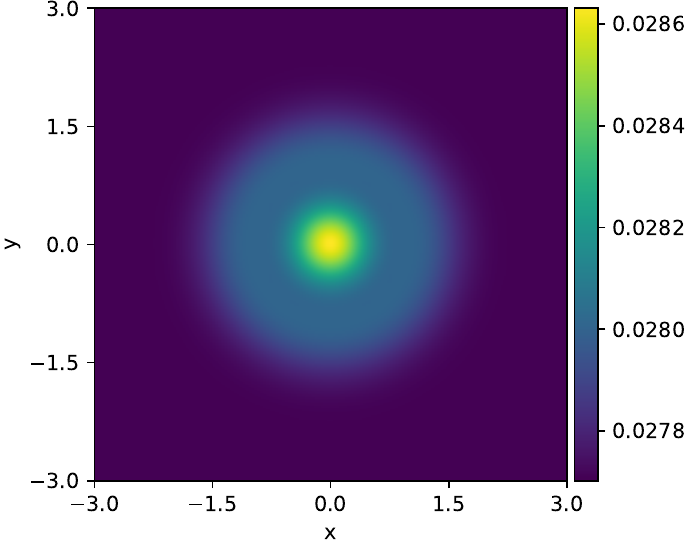}
        \caption{$n$}
        \label{fig:expansion:2x2v:fluid_vars:den_1e-5}
    \end{subfigure}
    \begin{subfigure}{0.315\textwidth}
        \includegraphics[width=\textwidth]{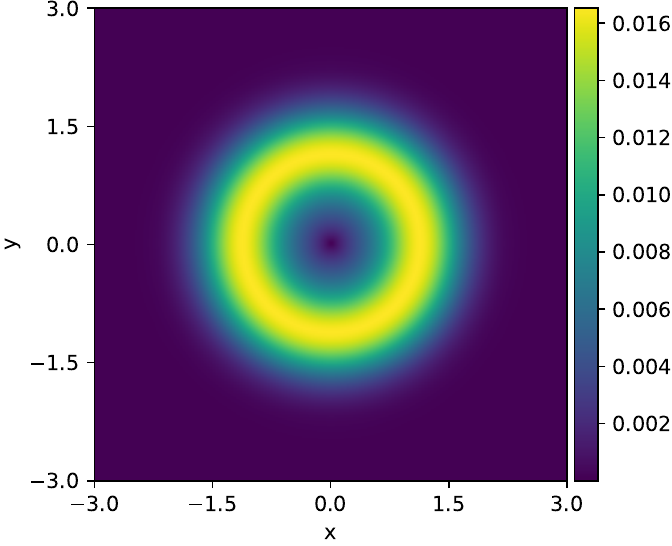}
        \caption{$|\bmu|$}
        \label{fig:expansion:2x2v:fluid_vars:vel_1e-5}
    \end{subfigure}
    \begin{subfigure}{0.31\textwidth}
        \includegraphics[width=\textwidth]{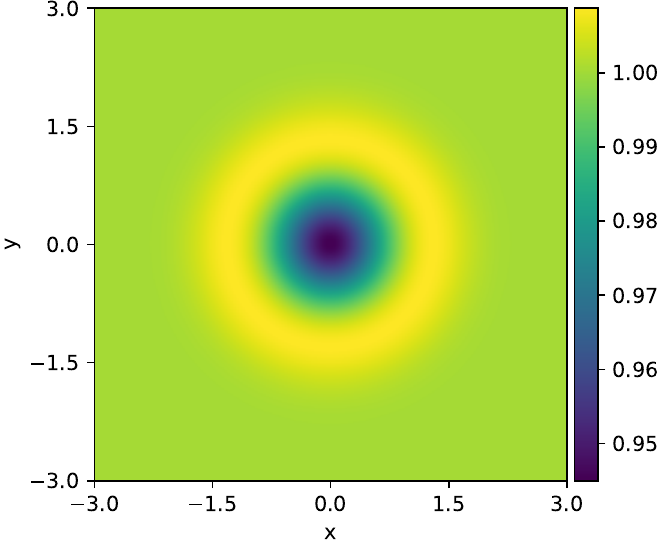}
        \caption{$\theta$}
        \label{fig:expansion:2x2v:fluid_vars:temp_1e-5}
    \end{subfigure}
    \caption{\Cref{subsec:expansion:2x2v} -- $2x2v$ Expansion problem, $\nu=10$. Plots of the fluid variables at $t=0.75$ and $\tau_\text{rel}=10^{-5}$.}
    \label{fig:expansion:2x2v:fluid_vars}
\end{figure}

\paragraph{Results}

\Cref{fig:expansion:2x2v:fluid_vars} plots the fluid variables at $t=0.75$ and $\tau_\text{rel}=10^{-5}$.
There is good agreement with the density plot in \cite[Figure 7]{dektor2026InterpolatoryDynamical}, but we observe a steeper density gradient and a higher density peak in the adaptive sparse-grid version.

\begin{figure}[!htbp]
    \centering
    \includegraphics[width=0.4\linewidth]{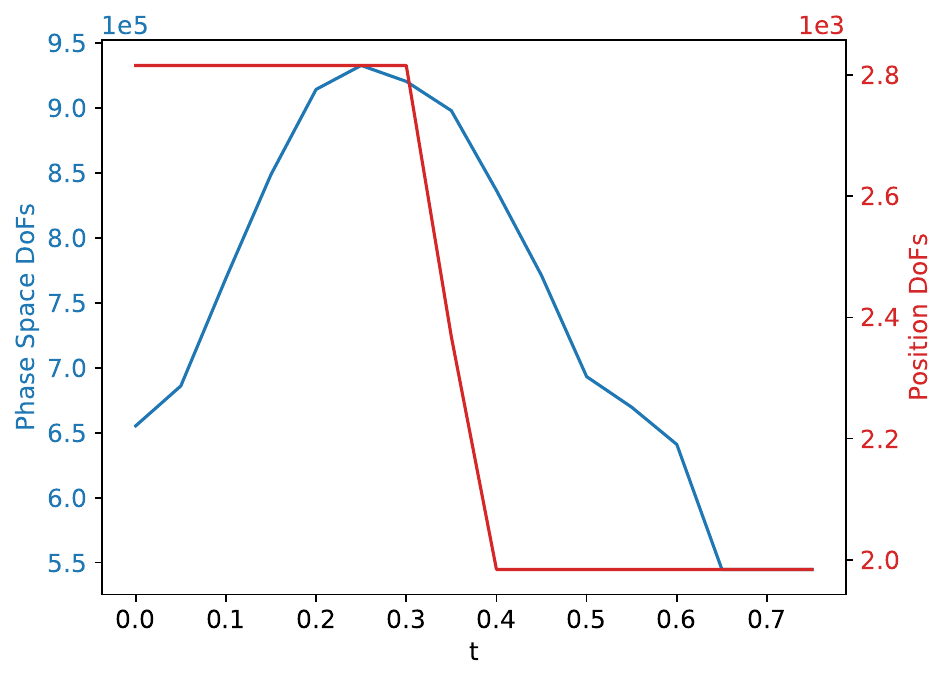}
    \caption{\Cref{subsec:expansion:2x2v} -- $2x2v$ Expansion problem, $\nu=10$. Degrees of Freedom over time for the $2x2v$ expansion problem with $\nu=10$ and $\tau_\text{rel}=10^{-5}$.  A level $5,5,3,3$ full-grid corresponds to 16.7 million DoFs}
    \label{fig:expansion:2x2v:dofs}
\end{figure}

The phase space and position DoFs against time are plotted in \Cref{fig:expansion:2x2v:dofs}.
There is a rise in the phase-space DoFs until $t=0.25$, followed by a steady decrease until plateauing at $t=0.65$.
This behavior is in contrast to \cite[Figure 7]{dektor2026InterpolatoryDynamical} where the rank, which is directly correlated to phase-space DoFs, only increases in time until hitting the user defined maximum rank ceiling.
Therefore, while the distribution exhibits increasing rank behavior across position and velocity coupling, it is still smooth and therefore well compressed by adaptive sparse grids.
We report an average velocity resolution of $\dofvavg\approx16.6$.

To verify accuracy, we compare the adaptive sparse-grid simulation to a reference solution computed with full-grid simulation of level $6,6,4,4$ and timestep $\dt=5\times 10^{-4}$.
\Cref{fig:expansion:2x2v:errors:phase_space} plots the relative $L^2$ error to the reference as a function of the phase-space DoFs for various $\tau_\text{rel}$ and shows starting  $\tau_\text{rel}=3.16\times 10^{-5}$.
\Cref{fig:expansion:2x2v:errors:position_space} plots the error against the reference moments and shows similar behavior as the phase space case, but we note a slightly tighter $\tau_\text{rel}$ is needed for saturation.
Additionally, the relative error of the density $\rho_1$ and energy $\rho_5$ agree well with $\tau_\text{rel}$ until saturation.
The larger magnitude in the momentum error is due to $\|\bmu\|_{L^2}\ll 1$.
The maximum phase-space DoFs for $\tau_\text{rel}=10^{-5}$, found at $t=0.25$, is \num{932864}, which is about $5.6\%$ the level $5,5,3,3$ full-grid DoFs.
The simulation for $\tau_\text{rel}=10^{-5}$ took 13.4 seconds.

We omit conservation plots but remark that deviations from the initial condition are at most $1.35\times 10^{-15}$.

\begin{figure}[!htbp]
    \centering
    \begin{subfigure}{0.4\textwidth}
        \includegraphics[width=\textwidth]{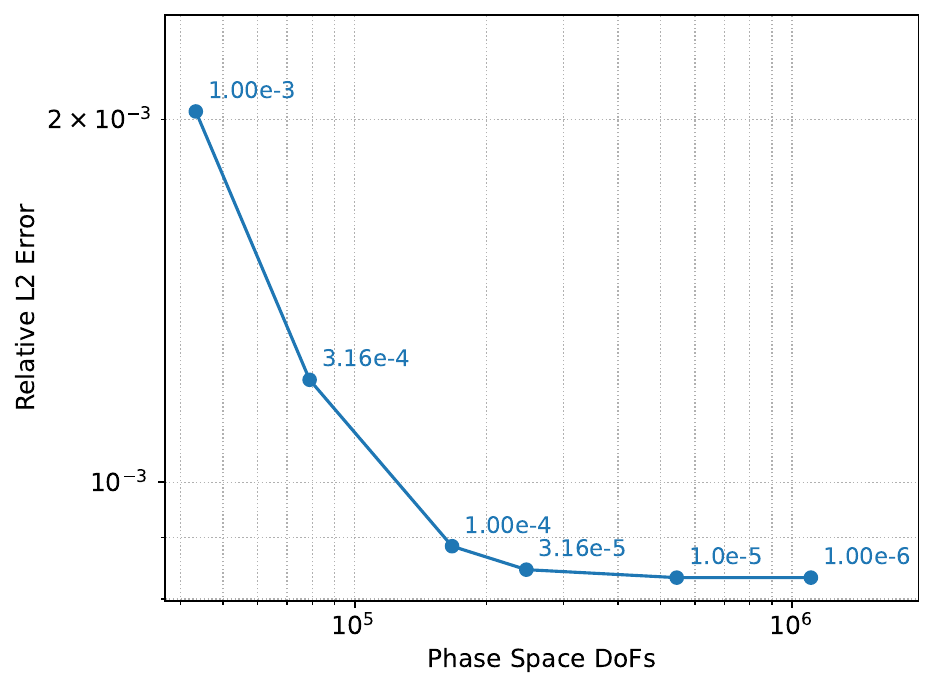}
        \caption{Error of $f$ in $L^2(\Omega)$}
        \label{fig:expansion:2x2v:errors:phase_space}
    \end{subfigure}
    \hspace{0.75in}
    \begin{subfigure}{0.4\textwidth}
        \includegraphics[width=\textwidth]{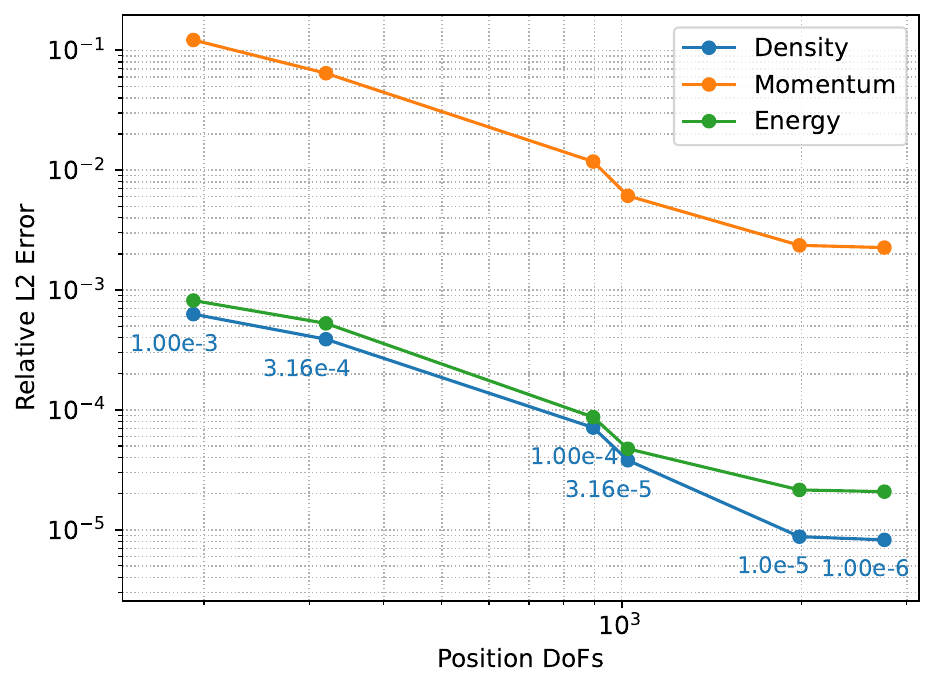}
        \caption{Error of $\bmrho_f$ in $L^2(\Omega_x)$}
        \label{fig:expansion:2x2v:errors:position_space}
    \end{subfigure}
    \caption{\Cref{subsec:expansion:2x2v} -- $2x2v$ Expansion problem, $\nu=10$. Relative $L^2$ errors against DoFs of the adaptive sparse grid method at $t=0.75$.  The markers on the plot indicated the relative adaptive tolerance $\tau_\text{rel}$. The reference solution is a full-grid simulation at one level finer than the adaptive sparse-grid maximum level in all dimensions.}
    \label{fig:expansion:2x2v:errors}
\end{figure}

\subsubsection{\texorpdfstring{$\bm{3x3v}$}{3x3v} with \texorpdfstring{$\bm{\nu=1000}$}{ν=1000}}
\label{subsec:expansion:3x3v}

We consider the same initial condition but as a $3x3v$ problem and high collisionality.

\paragraph{Discretization Specifics}

Let $d=3$ and $\nu=1000$.
We use a maximum level of $6,6,6,3,3,3$ and $k=3$ which aligns with the per dimension DoF given in \cite[Section 4.3.2]{dektor2026InterpolatoryDynamical}; this corresponds to $2^{29}\cdot 4^6\approx 550$ billion DoFs.
We set $\dt=1/1350\approx 7.407\times 10^{-4}$ which is $99.5\%$ $\dt_\text{expl}$, but
in the following tests, the position levels never exceeded 5; hence $\dt$ is only $49\%$ $\dt_\text{expl}$ with position level 5.
The final time is $t=1.0$.  
We set $\tau_\text{rel}=5\times 10^{-6}$ and $\tau_\text{abs} = 0$, and we do not set a nonlinear refinement function.

\begin{figure}[!htbp]
    \centering
    \begin{subfigure}{0.315\textwidth}
        \includegraphics[width=\textwidth]{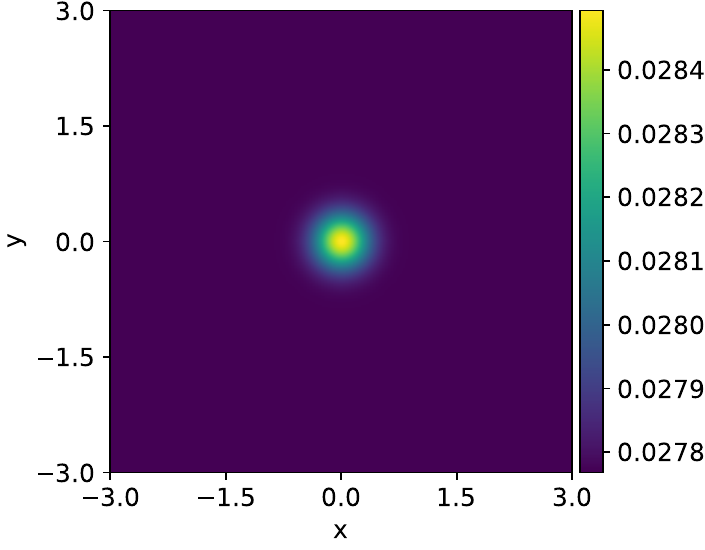}
        \caption{$t=0.0$}
        \label{fig:expansion:3x3v:density:t_0p0}
    \end{subfigure}
    \begin{subfigure}{0.32\textwidth}
        \includegraphics[width=\textwidth]{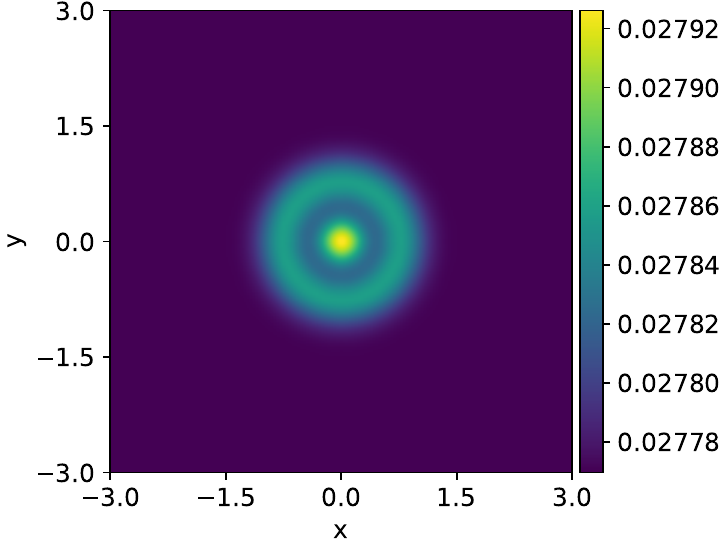}
        \caption{$t=0.5$}
         \label{fig:expansion:3x3v:density:t_0p5}
    \end{subfigure}
    \begin{subfigure}{0.32\textwidth}
        \includegraphics[width=\textwidth]{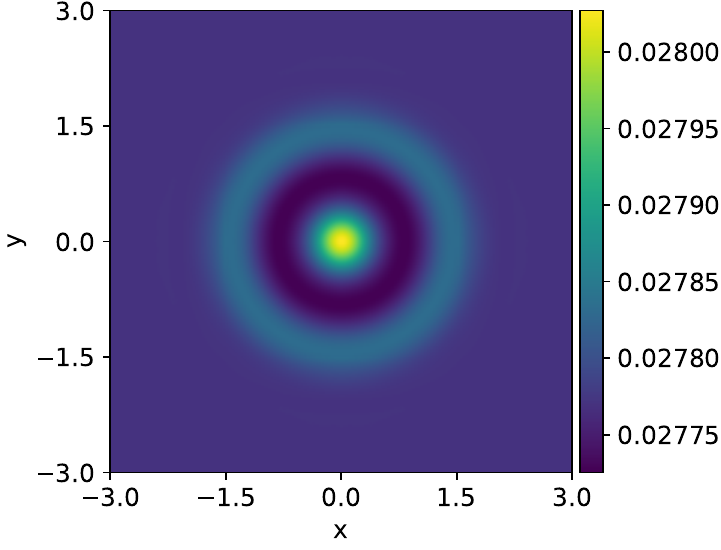}
        \caption{$t=1.0$}
         \label{fig:expansion:3x3v:density:t_1p0}
    \end{subfigure}
    \caption{\Cref{subsec:expansion:3x3v} -- $3x3v$ Expansion problem, $\nu=1000$. Plots of the marginal $(x,y)$ density $\int_{-3}^3 n_f(x,y,z)\dx{z}$ for various times.}
    \label{fig:expansion:3x3v:density}
\end{figure}

\paragraph{Results}

\Cref{fig:expansion:3x3v:density} plots the marginal $(x,y)$ density $\int_{-3}^3 n_f(x,y,z)\dx{z}$ for various times, and the plots largely agree with \cite[Figure 9]{dektor2026InterpolatoryDynamical}.
To further confirm correctness, \Cref{fig:expansion:3x3v:fluid_vars_comp} compares a radial cut of the fluid variables to the high resolution DG compressible Euler solver and shows good agreement away from the origin.

\begin{figure}[!htbp]
    \centering
    \begin{subfigure}{0.25\textwidth}
        \includegraphics[width=\textwidth]{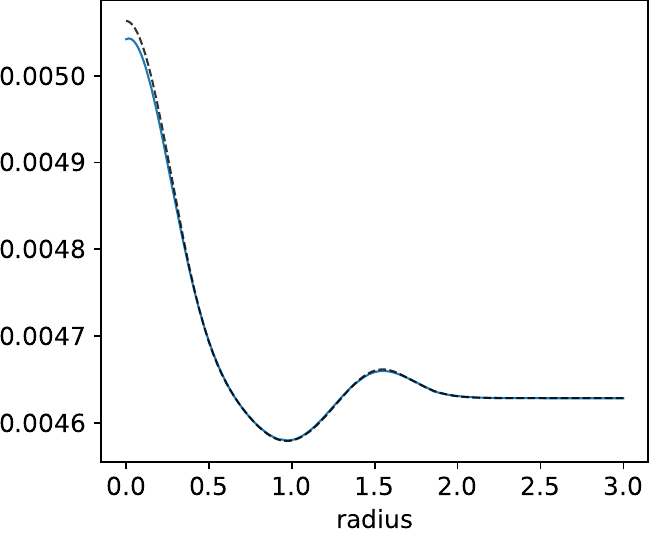}
        \caption{Density $n$}
        \label{fig:expansion:3x3v:fluid_vars_comp:density}
    \end{subfigure}
    \hspace{0.25in}
    \begin{subfigure}{0.25\textwidth}
        \includegraphics[width=\textwidth]{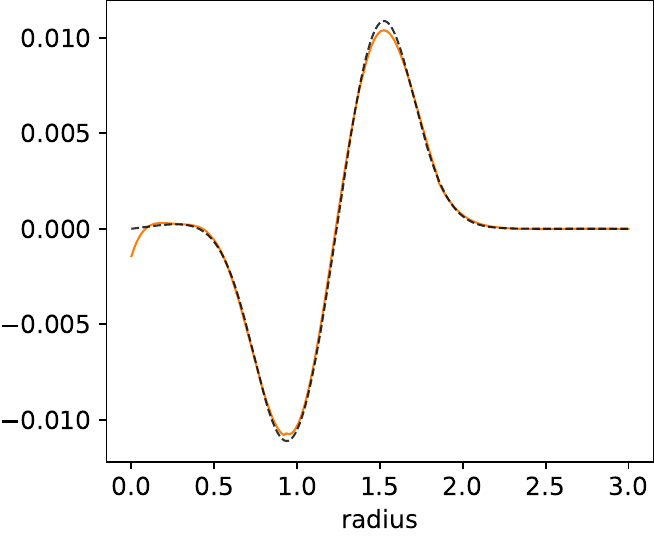}
        \caption{Radial velocity $u_r$}
        \label{fig:expansion:3x3v:fluid_vars_comp:velocity}
    \end{subfigure}
    \hspace{0.25in}
    \begin{subfigure}{0.237\textwidth}
        \includegraphics[width=\textwidth]{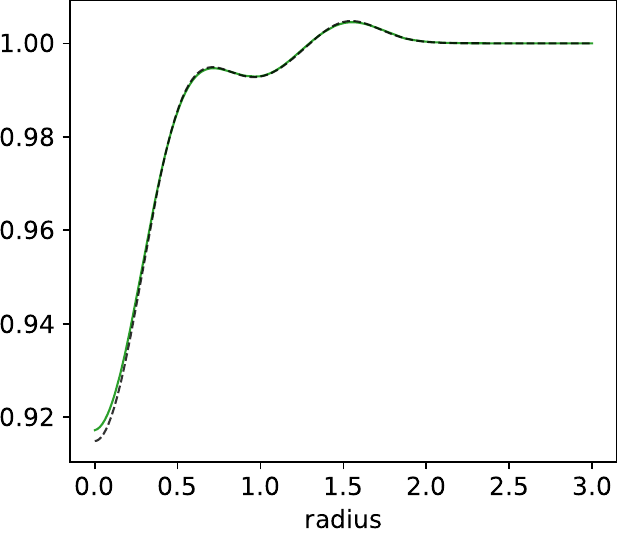}
        \caption{Temperature $\theta$}
        \label{fig:expansion:3x3v:fluid_vars_comp:temperature}
    \end{subfigure}
    \caption{\Cref{subsec:expansion:3x3v} -- $3x3v$ Expansion problem, $\nu=1000$. Radial line plots of the fluid variables along the ray $(0.2939,0.5090,0.8090)\in S^2$ at $t=1.0$.  Black dashed lines are from high resolution DG solution of the compressible Euler equations.}
    \label{fig:expansion:3x3v:fluid_vars_comp}
\end{figure}

\Cref{fig:expansion:3x3v:dofs} plots the phase space and position degrees of freedom over time.
Unlike the other collisional plots, the phase space and positional DoFs are less correlated.
We attribute this to the smoothness of the simulation, as the average velocity DoFs per dimension is only $\dofvavg\approx12.8$ at $t=0.4$.
The spike at $t=1.0$ is to the distribution starting to push out of the radius $r=1.5$ (see \Cref{fig:expansion:3x3v:density:t_1p0}). 
The trend in phase-space DoFs is similar to the interpolatory low-rank scheme \cite[Figure 8]{dektor2026InterpolatoryDynamical} if the last increase of DoFs is ignored.

The maximum number of phase space DoFs used is approximately 83 million at $t=0.4$, 
Since this simulation never refined beyond level 5 in position, the 83 million DoFs correspond to approximately $0.12\%$ the full-grid level 5,5,5,3,3,3 DoFs.
The simulation with $\tau_\text{rel}=5\times 10^{-6}$ took 73 minutes and 17 seconds.

We report the conservation errors are at most $7\times 10^{-12}$ for $0\leq t\leq 1$.

\begin{figure}[!htbp]
    \centering
    \includegraphics[width=0.4\linewidth]{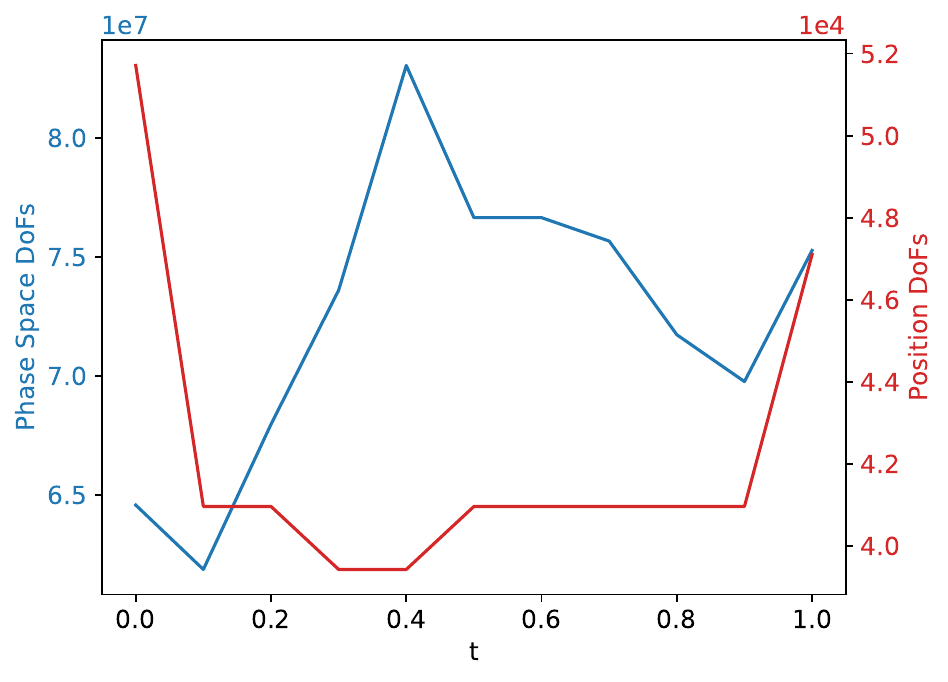}
    \caption{\Cref{subsec:expansion:3x3v} -- $3x3v$ Expansion problem, $\nu=1000$. Phase-space and position DoFs over time for $\tau_\text{rel}=5\times 10^{-6}$.  The level 5,5,5,3,3,3 full-grid corresponds to approximately 68 billion (\num{6.8e10}) DoFs. }
    \label{fig:expansion:3x3v:dofs}
\end{figure}

\section{Conclusion}
\label{sec:conclusion}

In this work, we have presented an adaptive sparse-grid DG method for the BGK model.
The results in this paper utilized the Adaptive Sparse-Grid Discretization (\asgard) library.
As demonstrated in \Cref{sec:numerical_experiments}, the adaptive sparse-grid method decreases the storage cost of DG numerical approximations to fractions of a percent without compromising accuracy across both fluid and rarefied regimes.
The method was demonstrated on a Sod shock tube problem and showed that the adaptive sparse-grid DG method can resolve sharp gradients in both position space for the fluid regime and phase space for the rarefied regime.
We additionally tested the method on a shear flow and expansion problem in the low-rank literature and found that smoothness can still be used to compress distributions in rarefied regimes that often exhibit increasing rank.
Moreover, the hybrid interpolation of the Maxwellian was shown to preserve the proper collision invariants of the BGK collision operator, leading to physically relevant solutions.

Future work includes distributed-computed solutions to larger-scale $3x3v$ models, implementation of kinetic boundary conditions (e.g., specular reflection \cite{cercignani1988BoltzmannEquation}), and embedded methods for non-trivial domain boundaries and geometries.

\printcredits

\section*{Declaration of generative AI and AI-assisted technologies in the manuscript preparation process.}
During the preparation of this work the author(s) used ChatGPT 5-2--5.6 for code assistance and review of the manuscript. After using this tool/service, the authors reviewed and edited the content as needed and take full responsibility for the content of the published article.

\bibliographystyle{abbrv}
\bibliography{ref}

\end{document}